\documentclass[12pt]{article}
 \usepackage{tikz-cd}
\usepackage[margin=1in]{geometry}
\usepackage{amsmath,amssymb,amsfonts,amsthm,mathtools}
\usepackage{enumitem}
\usepackage{amsmath,amssymb}
\usepackage{amscd}  
\usepackage{hyperref}
\usepackage{cleveref}
\usepackage{comment}
\usepackage{mathtools}  
  \DeclareMathAlphabet      {\mathbf}{OT1}{cmr}{bx}{n}
\newtheorem{theorem}{Theorem}[section]
\newtheorem{thm}{Theorem}

\newtheorem{lemma}[theorem]{Lemma}
\newtheorem{corollary}[theorem]{Corollary}

\newtheorem{remark}[theorem]{Remark}
\newtheorem{definition}[theorem]{Definition}

\numberwithin{equation}{section}
\usepackage{microtype}

\newcommand{\R}{\mathbb R}
\newcommand{\Z}{\mathbb Z}
\newcommand{\F}{\mathbb F}
\newcommand{\Q}{\mathbb Q}

\newcommand{\ECH}{\operatorname{ECH}}
\newcommand{\ECC}{\operatorname{ECC}}

\newcommand{\Gr}{\operatorname{Gr}}
\newcommand{\SW}{\operatorname{SW}}
\newcommand{\Sym}{\operatorname{Sym}}
\newcommand{\coker}{\operatorname{coker}}

\newcommand{\degop}{\operatorname{deg}}

\begin{document}

\title{Computing the  $U$ maps  on ECH of prequantization bundles}
\author{Guanheng Chen}
\date{}
\maketitle

\begin{abstract}
The purpose of this paper is to study the $U$-module structure of the embedded contact homology of a class of prequantization circle bundles over closed symplectic surfaces. We assume that the Euler number of each circle bundle is strictly less than the Euler characteristic of the corresponding surface. We give an explicit description of the $U$ map with respect to a basis of filtered ECH defined via cobordism maps. We use the formula for the $U$ map to compute the ECH spectrum and, as an application, the ECH capacities of certain complements of symplectic divisors. Furthermore, we show that the ECH capacities of closed  $4$-manifolds with rational symplectic forms  are infinite.
\end{abstract}
\tableofcontents

\section{Introduction and Main results }

Let $Y$ be a contact  $3$-manifold with a contact form $\lambda$ satisfying $\lambda \wedge d\lambda>0$.  M. Hutchings introduces a homology group $\ECH(Y, \lambda, \Gamma)$ for each  $\Gamma \in H_1(Y, \Z)$ \cite{H1, H4}, together with a degree-$(-2)$ endomorphism
$$
U:\ECH(Y,\lambda,\Gamma)\to 
\ECH(Y,\lambda,\Gamma).
$$
Hutchings  calls the above  homology group  \textbf{embedded contact homology}, abbreviated as ECH.  Using the $U$ map and the action filtration,  he further defines a nondecreasing sequence of numerical invariants 
$$c_1(Y, \lambda)\le c_2(Y, \lambda)\le c_3(Y, \lambda)....  \le \infty$$
called  \textbf{ECH spectrum} \cite{H3}.  Hutchings uses the  ECH spectrum  to define the \textbf{ECH capacities} for  symplectic manifolds $(X, \omega)$,    denoted by $c_k(X, \omega), k\in \Z_{\ge 1}$. 
The ECH capacities  are obstructions to symplectic embedding in the following sense: Suppose we have an embedding $\varphi:   (X_1, \omega_1)  \hookrightarrow  (X_2, \omega_2)$  such that $\varphi^*\omega_2=\omega_1$. Then 
\begin{equation} \label{eq42}
c_k(X_1, \omega_2) \le c_k(X_2, \omega_2)
\end{equation}
for all $k\ge 1$. 
ECH capacities have been successfully applied to a variety of four-dimensional symplectic embedding problems.  For examples,  see \cite{H3,CCFHV,DCG, FR}.

The purpose of this paper is to determine the $U$-module structure of
embedded contact homology for a class of prequantization bundles over
closed surfaces. We recap the definition of  prequantization bundles as follows.

Let $(\Sigma, \omega_{\Sigma}) $ be a closed  symplectic  surface of genus $g$ with integral symplectic form, i.e.,  $[\omega_{\Sigma}]  \in H^2(\Sigma, \Z)$.  Let
\begin{equation*}
       \pi_E: E \longrightarrow \Sigma
       \end{equation*}
be a complex line bundle with $c_1(E) =-[\omega_{\Sigma}]$.  The degree of this line bundle is 
\begin{equation*}
      e=  \langle c_1(E), \Sigma \rangle=-\int_{\Sigma}{\omega_{\Sigma}}<0.
 \end{equation*}
 Fix a Hermitian metric $h$ and a Hermitian connection 1-form $A_{\nabla}$   on $E$ such that $\frac{i}{2\pi} F_{A_{\nabla}} = -\omega_{\Sigma}$, where $ F_{A_{\nabla}} $ is the curvature of $ {A_{\nabla}} $.  The connection  $A_{\nabla}$   induces  a global angular form $\alpha_{E} \in \Omega^1(E \setminus \Sigma, \mathbb{R})$.  Under a unitary trivialization $U_{\alpha} \times \mathbb{C}$,  $\alpha_{E}$ is of the form 
 $$\frac{1}{2\pi} (d\theta  - i A_{\nabla} \vert_{U_{\alpha}} ),$$ 
 where  $d \theta$ is the angular form of $\mathbb{C}$ and $ A_{\nabla} \vert_{U_{\alpha} }$ is an $i\mathbb{R}$-valued $1$-form. Therefore, we have $d\alpha_{E} = \pi_E^*\omega_{\Sigma}$ over $E \setminus\Sigma$.

 Let $$\pi: Y=\{ v\in E \vert  h(v, v) =1\} \to \Sigma$$ be the  unit circle subbundle of $E$.  Set  $\lambda : =2\alpha_{E } \vert_Y$. Then $\lambda \wedge d\lambda \vert_{U_{\alpha} }=  \frac{2}{\pi} d\theta \wedge \omega_{\Sigma} \vert_{U_{\alpha} } >0$ locally.  Therefore,  $(Y, \lambda)$ is  a contact manifold and it is called  the \textbf{prequantization bundle} of $(\Sigma, \omega_{\Sigma})$. The number $e$ is the Euler number of the circle bundle $\pi: Y \to \Sigma. $

The first homology group of $Y$ is
$$H_1(Y;\mathbb \Z)
\cong \Z^{2g}\oplus  \Z/|e|  \Z; $$
see, for example, Lemma 3.7 of \cite{NW}. In this paper, we restrict attention
to classes in the torsion summand. We   fix
$ \Gamma\in\mathbb \Z/|e|  \Z $
and choose its standard representative
$ 0\le \Gamma\le  |e|-1$ throughout.

Based on  the frameworks of D. Farris \cite{Fa},  J. Nelson and M. Weiler \cite{NW}  give a complete description of the ECH of prequantization bundles as $\Z/2\Z$-modules. Specially,   they show that
$$\ECH(Y, \lambda, \Gamma) \cong \bigoplus_{d \ge 0}\Lambda^{d|e|+\Gamma}H_*(\Sigma, \Z/2\Z)$$ 
as  $\Z/2\Z$-modules. 
The first goal of this paper is to upgrade   their  isomorphisms to  $U$-modules isomorphisms under additional assumptions.  
\begin{thm}\label{thm1}
Let $(Y, \lambda)$ be a prequantization bundle over a symplectic surface  $(\Sigma, \omega_{\Sigma})$ with negative Euler number $e$. Assume $e< \chi(\Sigma)$. If $\Gamma =0$ or $\Gamma>2g-2$,  then we have a (ungraded)  $U$-module isomorphism 
\begin{equation*}
\begin{split}
\ECH(Y, \lambda, \Gamma) \cong  \bigoplus_{i= 1}^{2^{2g}}  \mathcal{T}^+,
\end{split}
\end{equation*}
where $g$ is the genus of $\Sigma$ and $\mathcal{T}^+:=\Q[U^{-1}, U]/U\Q[U]$.
\end{thm}

\paragraph{Coefficients} The ECH groups, ECH spectrum and ECH capacities  are defiened over   $\mathbb{F}=\Q$    throughout; unless otherwise stated.

The reason why we use rational coefficient is that our argument relies heavily  on a choice of Lefschetz basis of $H^*(\Sym^n\Sigma, \Q)$. Such a basis is no longer valid if we replace the coefficient to  $\F=\Z$ or $\F=\Z/2\Z$ (see Remark \ref{remark4}).  Some further works   should be required if we want change the coefficients to  $\F=\Z$ or $\F=\Z/2\Z$.   The assumptions  $e< \chi(\Sigma)$, and  $\Gamma =0$ or $\Gamma>2g-2$ are also attributed to our requirement of such a basis (see Remark \ref{remark5}).

In the proof of Theorem \ref{thm1}, we obtain an explicit formula for the $U$ map under a basis  of filtered ECH (Lemma \ref{lem7}).  This  formula  for  $U$ map  in the filtered ECH  help us to compute the explicit values of the  ECH spectrum.
\begin{thm} \label{thm3}
Let $(Y, \lambda)$ be a prequantization bundle over a symplectic surface $(\Sigma, \omega_{\Sigma})$ with negative Euler number $e$.  Assume $e<\chi(\Sigma)$.   Then the  $k$-th  value on the ECH spectrum  of $(Y, \lambda)$  is 
\begin{equation*}
c_k(Y, \lambda) =2d|e|, 
\end{equation*}
where  $d \ge 1$  is  the unique integer satisfying  
$$|e|\frac{d(d-1)}{2} +(d-1)(1-g) -g < k\le  |e|\frac{d(d+1)}{2} +d(1-g) -g.$$
\end{thm}

\begin{remark}
 In our convention, each fiber of $(Y, \lambda)$ has action $2$,  which accounts for the factor of  $2$ in  $c_k(Y, \lambda) =2d|e|$. 
\end{remark}

 In contract to the case $e<\chi(\Sigma) $,  recently  G. Beiner (Corollary 4.7 of \cite{GB})  show that if $\chi(\Sigma)\le e<0$, then $c_k(Y, \lambda) =\infty$.

 An application of Theorem \ref{thm3} is that we can compute the ECH capacities of certain  positive  divisor complements.

Let $(M, \Omega)$ be a   closed symplectic maniold. A codimension two connected symplectic closed submanifold $(\Sigma,  \omega_{\Sigma} = \Omega \vert_{\Sigma}) \subset M$ is called  a \textbf{symplectic hyperplane section} of $M$     if $[\Sigma]$  is Poincare dual (over $\Q$) to $k_{\Sigma}[\Omega]$ for some $k_{\Sigma}>0$.   The number   $k_{\Sigma}>0$ is called the \textbf{degree} of the symplectic hyperplane section.
 
\begin{corollary}\label{corollary1}
Let $(M, \Omega)$ be a closed  symplectic 4-maniold.   Suppose we have     a symplectic hyperplane section   $\Sigma$ with degree $k_{\Sigma}>0$ and the section  satisfies $[\Sigma]^2>2g(\Sigma)-2$.   Then the $k$-th ECH capacity of the  complement  of $\Sigma$ is 
$$c_k(M\setminus \Sigma, \Omega)=\frac{d}{k_{\Sigma}}[\Sigma]^2 =d k_{\Sigma} \int_M {\Omega}\wedge \Omega,$$
where  $d \ge 1$ is  the unique integer satisfying 
$$[\Sigma]^2\frac{d(d-1)}{2} +(d-1)(1-g(\Sigma)) -g(\Sigma) < k\le  [\Sigma]^2\frac{d(d+1)}{2} +d(1-g(\Sigma) ) -g(\Sigma) .$$
 
\end{corollary}

\begin{remark} \label{remark2}
Let $(\Sigma, \omega_{\Sigma}) \subset (M, \Omega)$ be a symplectic hyperplane section. By Adjunction formula, we  have 
$$k_{\Sigma}\langle c_1(TM), [\Omega] \rangle = \langle c_1(TM), [\Sigma] \rangle = \chi(\Sigma) + [\Sigma]^2= 2-2g(\Sigma) +  [\Sigma]^2. $$
Therefore, the condition   $[\Sigma]^2>2g(\Sigma)-2$ is equivalent to $\langle c_1(TM), [\Omega]  \rangle>0$.  A classification result by Liu\cite{L} and Ohta-Ono \cite{OO} states that  if  $\langle c_1(TM), [\Omega] \rangle>0$, then  $(M, \Omega)$  is   a rational surface, or a  ruled surface, or   a blow-up of one. 
\end{remark}

\paragraph{Examples} 
Here are some examples  that Corollary \ref{corollary1} applies.
\begin{itemize} 
 \item
 Let $(M, \Omega) = (\mathbb{CP}^2, \omega_{\rm FS})$, where  $\omega_{\rm FS}$ is the standard Fubini-Study K\"ahler form.  Fix an integer $n \ge 1$. Let $\Sigma_n$ be a symplectic surface with $[\Sigma_n] =nH $, where $H \in H_2(\mathbb{CP}^2, \Z)$ is the homology  class of a line in $\mathbb{CP}^2$.  Then 
 $\operatorname{PD}([\Sigma_n]) =n[\omega_{\rm FS}]$, i.e., $k_{\Sigma_n} =n$.  
 
By Adjunction formula, we know that 
$$g(\Sigma_n) =\max\{0 , \frac{(n-1)(n-2)}{2} \}.$$
Note that  $(\mathbb{CP}^2, \omega_{FS})$ is a     symplectic manifold with  $\langle c_1(T\mathbb{CP}^2), [\omega_{\rm FS}]  \rangle>0$. Thus, the assumption  $[\Sigma_n]^2>2g(\Sigma_n)-2$  is true by Remark \ref{remark2}. Applying  Corollary \ref{corollary1}, we obtain  
 $$c_k(\mathbb{CP}^2 \setminus \Sigma_n, \omega_{\rm FS}) =  n\left\lceil
\frac{\sqrt{8k+4n^2-12n+17}-3}{2n}
\right\rceil. $$

 \item
 Let $(M, \Omega) = (\mathbb{S}^2 \times \mathbb{S}^2 , \omega_{\mathbb{S}^2} \oplus \omega_{\mathbb{S}^2})$, where  $\mathbb{S}^2$ is a  volume form of $\mathbb{S}^2$ such that $\int_{\mathbb{S}^2} \omega_{\mathbb{S}^2} =1$.  Let  $\Delta$ denote the diagonal of $\mathbb{S}^2 \times \mathbb{S}^2$.  Fix an integer $n \ge 1$. Let $\Sigma_n$ be a symplectic surface with $[\Sigma_n] =n[\Delta] $.  Then 
 $\operatorname{PD}([\Sigma_n]) =n[\Omega]$, i.e., $k_{\Sigma_n} =n$.  
 
By Adjunction formula, we  have 
$$g(\Sigma_n) =(n-1)^2.$$
By Remark \ref{remark2},  the assumption  $[\Sigma_n]^2>2g(\Sigma_n)-2$  holds because $ (\mathbb{S}^2 \times \mathbb{S}^2 , \omega_{\mathbb{S}^2} \oplus \omega_{\mathbb{S}^2})$ is a ruled symplectic manifold.  

Applying  Corollary \ref{corollary1}, we obtain  
 $$c_k( \mathbb{S}^2 \times \mathbb{S}^2  \setminus \Sigma_n,  \omega_{\mathbb{S}^2} \oplus \omega_{\mathbb{S}^2}) =  2n\left\lceil
\frac{\sqrt{k+n^2-2n+2}-1}{n}
\right\rceil. $$
 \end{itemize}

 \begin{corollary} \label{corollary2}
 Let $(M, \Omega) = (\mathbb{CP}^2, \omega_{\rm FS})$ or   $(M, \Omega) = (\mathbb{S}^2 \times \mathbb{S}^2 , \omega_{\mathbb{S}^2} \oplus \omega_{\mathbb{S}^2})$. For $n \ge 1$,  let $\Sigma_n$ be  the corresponding symplectic hyperplane sections  in the above examples.   Then there is no symplecitc embedding 
 $$\varphi: (M \setminus \Sigma_{n_2}, \Omega) \hookrightarrow    (M \setminus \Sigma_{n_1}, \Omega),$$
 provided that $n_2 >n_1$.  
 \end{corollary}
 \begin{proof}
 In the above examples, we know that 
 \begin{equation*}
c_1(M \setminus \Sigma_n, \Omega)=
\begin{cases}
 n & \mbox{ if }  M=\mathbb{CP}^2 \\ 
 2n &    \mbox{ if }  M=\mathbb{S}^2 \times \mathbb{S}^2.  
 \end{cases}
\end{equation*}
If such a symplectic embedding  $\varphi$ exists, then  by (\ref{eq42}), $n_2 \le n_1$.   This  contradicts our assumption. 
 \end{proof}

Combining the results of  S. K. Donaldson \cite{Don},   G. Beiner \cite{GB} and Corollary \ref{corollary1}, we  reveal the infiniteness of ECH capacities for closed symplectic manifolds with rational symplectic forms. This result  is  also proved independently and contemporaneously by G. Beiner \cite{GB1} using a similar idea.
 \begin{thm}  \label{thm4}
 Let $(M, \Omega)$ be a 4-dimensional closed    symplectic manifold. Suppose the symplectic form is rational in the sense that     there exists $\tau>0$ such that $\tau [\Omega] \in H^2(M, \Z)$. Then  for any $k \in \Z_{\ge 1}$,  we have 
 $$c_k(M, \Omega )=\infty. $$
 \end{thm}
 \begin{proof}
  It suffices to show that $ c_1(M, \Omega )=\infty$ due to the monotonicity  of   the ECH capacities.  Also, by the conformality  of ECH capacities, i,.e.,  $ c_1(M, \tau \Omega )= \tau c_1(M,  \Omega )$,  it suffices to prove the case that $\tau=1$.
 
 By the result of  S. K. Donaldson  \cite{Don}, there exists a symplectic submanifold $\Sigma_n$ with $\operatorname{PD}[\Sigma_n] =n [\Omega]$,  provided that $n>0$ is sufficiently large.   Moreover, by Proposition 39 of \cite{Don}, $\Sigma_n$ is connected for sufficiently large $n$.  We refer readers to Theorem 1.3 of \cite{H4} for these properties. 

In Lemma \ref{lem23}, we show that  
$$c_k(M \setminus \Sigma_n, \Omega) = \frac{1}{2n} c_k(Y, \lambda), $$
where  $(Y, \lambda)$ is a prequantization over  $(\Sigma_n,  n \Omega \vert_{\Sigma_n})$ with Euler number $$e=-[\Sigma_n]^2 =-n^2 \int_{M}\Omega \wedge \Omega <0.$$ We have the following  obvious  symplectic embedding 
$$(M \setminus \Sigma_n, \Omega) \hookrightarrow  (M , \Omega). $$
Therefore, we have
$$ \frac{1}{2n} c_1(Y, \lambda)=  c_1(M \setminus \Sigma_n, \Omega) \le  c_1 (M , \Omega)$$
for   sufficiently large $n$.  

If  $[\Sigma_n]^2\le 2g(\Sigma_n) -2$,    by Corollary 4.7 of \cite{GB}, then  $c_1(Y, \lambda) =\infty $. Hence, $ c_1 (M , \Omega) =\infty$. It is worth to noting that  $c_1(Y, \lambda) =\infty $ in \cite{GB} is stating for $\Z/2\Z$-coefficient, whereas we are using $\Q$-coefficient.   The proof of  Corollary 4.7 of \cite{GB} uses Theorem 3.1 \cite{GB}. The proof of the latter only relies on the formal properties of ECH cobordism maps  (for symplectic cobordism with non-disconnected boundaries).  By fixing a homology orientation, the ECH cobordism maps also can be defined over $\Q$ (see Remark   1.13 of \cite{HT}). Thus, his argument still holds for the $\Q$-coefficient. 

If  $[\Sigma_n]^2>2g(\Sigma_n) -2$,   by Corollary \ref{corollary1}, we have 
\begin{equation*}
c_1(M \setminus \Sigma_n, \Omega)=
\begin{cases}
n \int_{M} \Omega\wedge \Omega, & [\Sigma_n]^2 \ge 2g(\Sigma_n)\\
2 n \int_{M} \Omega\wedge \Omega,   & [\Sigma_n]^2 = 2g(\Sigma_n)-1. 
\end{cases}
\end{equation*}
 Thus,  $c_1 (M , \Omega) \ge n \int_{M} \Omega\wedge \Omega$ for all sufficiently large $n$. Again, we have $ c_1 (M , \Omega) =\infty$. 
  
 \end{proof}

 \paragraph{Relation with  others  results}
As mentioned at the beginning, D. Farris \cite{Fa},  J. Nelson and M. Weiler \cite{NW}  figured out the $\Z/2\Z$-modules structure of  ECH of prequantization bundles for all $e \in \Z_{<0}$ and  $0\le \Gamma \le |e| -1$. Their method should also work for other coefficients. But they haven't provided  further computations on the $U$ maps.   In a similar spirit, Nelson and Weiler also compute the ECH of certain prequantization orbibundles \cite{NW1, NW2}. On the other hand, our method here  only works for $\Q$-coefficient and additional assumptions on $(e, \Gamma)$,  whereas cannot  be applied to $\Z/2\Z$-coefficient  and general  $(e, \Gamma) $ without  nontrivial  modification.

The Seiberg-Witten Floer homology  of the prequantization bundles  over $\mathbb{S}^2$   (as $U$-modules) is  computed by P. Kronheimer, T. Mrowka, P. Ozsv\'{a}th and Z. Szab\'{o} \cite{KMOS}.  The Heegaard Floer homology for the  general case is obtained by combining the results of    K. Park (Theorem 4.1.1 of  \cite{Park})  and   Ozsv\'{a}th-Szab\'{o}  (Theorem 5.6 of \cite{OS}). The (hat) Seiberg-Witten Floer homology, embedded contact homology,  (positive) Heegaard Floer homology  are equivalent as  $U$-modules.  These  profound results are respectively proved by Taubes  \cite{T1, T2, T3, T4, T5}, 
Kutluhan-Lee-Taubes \cite{KLT1, KLT2, KLT3, KLT4, KLT5}, and  Colin-Ghiggini-Honda \cite{CGH1, CGH2, CGH3}. 
Therefore, the result in Theorem \ref{thm1}  can be obtained alternatively by combining the computations of  Ozsv\'{a}th-Szab\'{o} \cite{OS},  Park \cite{Park}, and ``ECH=SWF=HF".  

Although the  
$U$-module structure can in principle be recovered from known Floer-theoretic equivalences, our construction produces a  basis of filtered ECH and an explicit formula for the 
$U$ map  in that basis. This additional filtered information is used to compute the ECH spectrum and ECH capacities.  The  ECH spectrum of prequantization  bundles for some special cases, especially $\Sigma =\mathbb{S}^2$,  is known to be finite and computed by Ferreira-Ramos \cite{FR}, Shibata \cite{Shi} and the author \cite{GHC}.  For  the case that  $\chi(\Sigma)\le e<0$,  G. Beiner \cite{GB} shows  that  $c_k(Y, \lambda) =\infty$.   As mentioned in the proof of Theorem \ref{thm4}, his method also works for $\Q$-coefficient. Thus, combining  Theorem \ref{thm3}, the ECH spectrum  of prequantization bundles over $\Q$ are known for all cases.

Theorem \ref{thm4}  is  also proved independently by G. Beiner \cite{GB1}. He observes that $c_1(Y, \lambda) \ge 2|e|$ for a prequantization bundle with Euler number $e$ (Lemma 3.2 of \cite{GB1}), and the rest of the argument is basically the same as the proof of Theorem \ref{thm4}. The precise value of  $c_1(Y, \lambda) $ is not necessary in his proof. Thus, his argument is broader because there is no need to make an assumption on the coefficient. Beiner's previous work  also provides some irrational closed symplectic 4-manifolds with infinite ECH capacities (Corollary 4.13  of \cite{GB}). It seems  to be plausible to conjecture that   every  closed symplectic 4-manifold has  infinite ECH capacities.

There is an elementary version of ECH capacities for a symplectic $4$-manifold $(X, \omega)$ introduced by Hutchings \cite{H6}, denoted by $c_k^{ \rm Alt}(X, \omega), k \in\Z_{\ge1}$.  These capacities satisfy various properties that are similar to the ECH capacities (Theorem 6 of \cite{H6}).  Moreover,  $c_k^{\rm Alt}(X, \omega) \le  c_k(X, \omega)$ for all $k \ge 1$  (Theorem 1.2  of \cite{H6}). We refer readers to the survey \cite{H7} for more details on the  elementary version of ECH capacities and their applications.  
In contrast to Theorem \ref{thm4},   Lemma 18  of \cite{H6} tells us that there are many examples,  including  $(\mathbb{CP}^2, \omega_{\rm FS}) $  and  $( \mathbb{S}^2 \times \mathbb{S}^2 ,  \omega_{\mathbb{S}^2} \oplus \omega_{\mathbb{S}^2})  $,  
have finite  $c_k^{\rm Alt}(X, \omega) $ for all $k \in \Z_{\ge 1}$.   In Theorem D of \cite{GB1},   G. Beiner  improves  Hutchings's argument and shows that alternative ECH capacities of closed symplectic 4-manifold with $b_2^{+} =1$ are finite.

\paragraph{Idea of proof}

In \cite{Fa, NW},   D. Farris  and  Nelson-Weiler show that after a Morse-Bott perturbation, the Reeb orbits with action less than $L_{\varepsilon} $   are fibers of $Y$ at critical points of a fixed  perfect Morse function $H$.  
Moreover, they show that the differential on the filered ECH  complex $\ECC^{L_{\varepsilon}}(Y, \lambda_{\varepsilon}, \Gamma)$ vanishes.  Thus, a natural basis of  $\ECH^{L_{\varepsilon}}(Y, \lambda_{\varepsilon}, \Gamma)$ is the ECH generators consisting  of the fibers at critial points of $Y$. However, to compute the $U$ map under this basis need to study the holomorphic curves in the symplectization of $Y$ which is  a very difficult problem.

The core idea of this paper is  inspired by the works of  V. Munoz and B. L. Wang \cite{MW},  where they study the ring structures of the Seiberg-Witten cohomology of $S^1 \times \Sigma.$ S. Zhou uses the similar argument to compute the ring structures of  perturbed  Seiberg-Witten cohomology/ periodic Floer homology  of $S^1 \times \Sigma$ \cite{Zhou}, where   periodic Floer homology  is a sister version of ECH, also introduced by Hutchings \cite{H1}. 

 We adapt the argument of  Munoz-Wang \cite{MW} and Zhou \cite{Zhou} to ECH and prequantization bundles setting.  We construct a basis of $\ECH^{L_{\varepsilon}}(Y, \lambda_{\varepsilon}, \Gamma)$ by using ECH cobordism maps instead. This basis can help us to reduce the computations of $U$ map   to computations of the Gromov-Taubes invariants/Seiberg-Witten invariants  of ruled surfaces, where the latter can be computed by the T. J. Li and A. Liu's wall crossing formula for Seiberg-Witten invariants \cite{LL} and its generalization by C. Okonek and A. Teleman \cite{OT, OT2}.  

Specifically,  our  strategy is as follows:
\begin{itemize}
\item
Let $E\to \Sigma$ and $E^{\vee} \to \Sigma$ be the associated line bundles of $Y$ such that $$c_1(E) =-[\omega_{\Sigma}] \mbox{ and }  c_1(E^{\vee}) =[\omega_{\Sigma}]. $$
We endow   $E, E^{\vee}$ with symplectic forms such that the unit disk bundles $X_-= \mathbb{D}E$ and $X_+=\mathbb{D}E^{\vee}$ are respectively strong  symplectic filling of $(Y, \lambda)$ and   strong  symplectic  cap of $(Y, \lambda)$.  Gluing $X_+$ and $X_-$ along their  common boundary,  we obtain a  symplectic ruled surface $\pi_X: X \to \Sigma$. 

\item
Let $\mathbb{A}(\Sigma) : =\mathbb{Q} [u]\otimes \Lambda^* H^1(\Sigma, \Q).  $ 
For $N > 2g-2$, by      I.G. MacDonald's work \cite{Mac}, we can choose suitable elements $\{z_{N, i_N}\} \subset \mathbb{A}(\Sigma)$ such that they forms a basis of $H^*(\Sym^N \Sigma, \Q)$ (see (\ref{eq17})).   If $N=0$, we set $H^*(\Sym^0 \Sigma, \Q)=\Q$ and the basis is just $1$. 
 We take $N=m|e|+\Gamma$.   Under  assumptions $e<\chi(\Sigma)$,  and
 $\Gamma =0$ or $\Gamma>2g-2$, $N>2g-2$ or $N=0$. We show that 
$$\{\Phi_m(x_{m, i_m})\}_{0\le m\le d} \subset  \ECH^{L_{\varepsilon}}(Y, \lambda_{\varepsilon}, \Gamma)$$
by studying  the properties of holomorphic currents in $E^{\vee}$, where $x_{m, i_m } =z_{m|e|+\Gamma, i_{m|e|+\Gamma}}$. 
Here $\Phi_m$ are  the ECH cobordism maps induces by $X_+$ and suitable choices of homology classes in  $H_2(X_+, Y, \Z)$.  Their roles here are similar to the PSS morphisms in Hamiltonian Floer homology.

\item
Set $V$ to be the vector space  spaned  by $\{\Phi_n(x_{n, i_n})\}_{0\le n\le d} $.  We define a pairing   by 
\begin{equation} \label{eq12}
\langle  \Phi_m(x_{m, j_m}),  \Phi_n(x_{n, i_n})   \rangle_{ech}: = \Psi_m(x_{m, j_m} \otimes  \Phi_n(x_{n, i_n}) )  \  \   \   \  (\mbox{Definition   \ref{definition1}})
\end{equation}
where $\Psi_m$ is the ECH cobordism maps induced by $X_-$ and suitable choices of homology classes in  $H_2(X_-, Y, \Z)$.  By the composition law of  ECH cobordism maps, the right hand side  of (\ref{eq12}) is nothing but  the Gromov-Taubes invariants of the ruled surface $X$.  

\item
Using the wall crossing  formula for the  Seiberg-Witten invariants (\cite{LL, OT, OT2}), we 
show that the pairing $\langle \cdot , \cdot\rangle_{ech}$ is  the direct sum of the intersection pairings  on $$\bigoplus_{0 \le m\le d} H^*(\Sym^{m|e|+\Gamma} \Sigma, \Q) \mbox{ (see Lemma \ref{lem3})}.$$   In particular,  the pairing   $\langle \cdot , \cdot\rangle_{ech}$ is nondegenerate. Then $\{\Phi_n(x_{n, i_n}) \}$ is linear independent. By dimension count, we know that   $\{\Phi_n(x_{n, i_n})\} $ forms a basis of $\ECH^{L'_{\varepsilon}}(Y, \lambda_{\varepsilon},\Gamma)$, where $L'_{\varepsilon} \approx \frac{2}{3}L_{\varepsilon}$.

\item
With the above preparations, computing the $U$ map is equivalent to solve   the following linear equations
$$U \Phi_n(x_{n, i_n}) =\sum_{m} \sum_{j_m} a_{m, j_m}\Phi_m(x_{m, j_m}).$$ 
Then we can express $a_{m, j_m}$ in terms of   inversion of the intersection pairings on $H^*(\Sym^{n|e|+\Gamma} \Sigma, \Q)$ and   $ \langle U   \Phi_n(x_{n, i_n}),  \Phi_m(x_{m, j_m}) \rangle_{ech}$. The latter again are  certain Gromov-Taubes invariants of $X$ which are  computable.  Eventually, we obtain a formula of the $U$ map under the basis  $\{\Phi_d (x_{m, i_m}) \}_{0\le m\le d}$ in Lemma \ref{lem7}.
\item
Finally, we combine  Nelson-Weiler's Morse-Bott argument \cite{NW} and the previous computations to show that $\ECH(Y, \lambda, \Gamma)$ consists of $2^{2g}$ $U$-towers.   The computation of the ECH spectrum are more or less a  byproduct  of  the formula for  the $U$ map.

\end{itemize}

\begin{remark} \label{remark5}
The above idea may give  an algorithm to compute the $U$ map for all $\Gamma, e$ and any coefficient.
However, the choice of the basis  $\{x_{n, i_n}\} \subset \mathbb{A}(\Sigma)$ plays a crucial role in the computations. If we choose them arbitrary, then 
the pairing (\ref{eq12}) may not be the direct sum of the intersection pairings  on $\oplus_{n \ge 0} H^*(\Sym^{n|e|+\Gamma} \Sigma, \Q)$. Instead, it should be very complicated.  Moreover, others choice of the basis   increases  dramatically  the level of difficulty to computing  $ \langle U  \Phi_n(x_{n, i_n}),  \Phi_m(x_{m, j_m}) \rangle_{ech}$. For  a given  tuple  $g, e, \Gamma $, one may follow the above route  to compute the $U$ map, however, it seems hard to get a closed formula.  

Such a choice of the basis is  the reason why we assume $e< \chi(\Sigma)  $,  and  $\Gamma=0$ or  $\Gamma>2g-2$, and use $\Q$-coefficient.

\end{remark}

\begin{ack}
The author would like to express his sincere gratitude to Shaoyang Zhou for generously sharing his manuscript \cite{Zhou} and for introducing the idea of  Munoz-Wang \cite{MW}, which greatly inspired this work.
\end{ack}

\section{Preliminaries}

\subsection{Embedded contact homology}
 In this section, we give a quick review of the embedded contact homology. For more details, please refer to  \cite{H4}.

Let $ Y$ be a closed, contact 3-dimensional manifold equipped with a nondegenerate contact
form $\lambda$. The contact structure of $Y$ is denoted by $\xi: =\ker \lambda$.
The \textbf{Reeb vector field} $R$ of $(Y, \lambda)$ is characterized by  conditions $\lambda(R)=1$ and $d\lambda(R, \cdot)=0$.
A \textbf{Reeb orbit} is a smooth map $\gamma: \mathbb{R}_{\tau} / T \mathbb{Z} \to Y $ satisfying the ODE $\partial_{\tau} \gamma =R \circ \gamma$ for some $T>0$. 
 The linearized Reeb flow for time $T$ defines a symplectic linear map
 \begin{equation*}
   P_{\gamma}: (\xi_{\gamma(0)}, d\lambda) \to (\xi_{\gamma(T)}, d\lambda).
 \end{equation*}
 The Reeb orbit $\gamma$ is called \textbf{nondegenerate} if $1$ is not an eigenvalue of $P_{\gamma}$.  A nondegenerate Reeb orbit $\gamma$ is called \textbf{elliptic} if the eigenvalues of $P_{\gamma}$ are on the unit circle, \textbf{positive hyperbolic} if the eigenvalues   are real positive
numbers, and \textbf{negative hyperbolic} if the eigenvalues are real negative numbers. Given $L \in \mathbb{R}$, the contact form $\lambda$ is called \textbf{$L$-nondegenerate} if all Reeb orbits with action less than $L$  are nondegenerate.  We  call $\lambda$  nondegenerate if all   Reeb orbits  are nondegenerate.

 An \textbf{orbit set}    $\alpha=\{(\alpha_i, m_i)\}$ is a finite set of Reeb orbits together with multiplicities, where $\alpha_i$ are distinct, nondegenerate, irreducible embedded Reeb orbits and $m_i$  are positive integers.  An orbit set $\alpha$ is called an \textbf{ECH generator} if $m_i=1$ whenever $\alpha_i$ is a hyperbolic orbit. Define the \textbf{action} of an orbit set $\alpha$ by
 \begin{equation*}
   \mathcal{A}_{\lambda}(\alpha):=\int_{\alpha} \lambda = \sum_{i} m_i \int_{\alpha_i} \lambda.
 \end{equation*}
 Note that  the action $   \int_{\gamma} \lambda$ is just the period $T$ in the definition of  Reeb orbits.

\textbf{Notation. }In the rest of the paper, we often write an orbit set using multiplicative notation $\alpha=\Pi_i \alpha_i^{m_i}$ instead.

\paragraph{The ECH index}
Fix $\Gamma \in H_1(Y, \mathbb{Z})$. Given orbit sets $\alpha=\Pi_i  \alpha_i^{m_i}$ and $\beta=\Pi_j  \beta_j^{n_j}$ on  $Y$ with $[\alpha]=[\beta] =\Gamma$,  let $H_2(Y, \alpha ,\beta)$ be the set of 2-chains $Z$ such that $\partial Z = \sum_i m_i \alpha_i - \sum_j n_j \beta_j$, modulo boundaries of 3-chains.   An element in  $H_2(Y, \alpha ,\beta)$  is called \textbf{a relative homology class}. Note that the set  $H_2(Y, \alpha ,\beta)$ is an affine space over $H_2(Y, \mathbb{Z})$.

Given $Z \in H_2(Y, \alpha ,\beta)$ and  trivializations    $\tau$ of $\xi\vert_{\alpha}$ and $\xi \vert_{\beta}$,  the ECH index is defined by
\begin{equation*}
I(\alpha, \beta, Z) := c_{\tau}(\xi \vert_Z) + Q_{\tau}(Z) + \sum_i  \sum\limits_{p=1}^{m_i} CZ_{\tau}(\alpha_i^{p})- \sum_j \sum\limits_{q=1}^{n_j} CZ_{\tau}(\beta_j^{q}),
\end{equation*}
where $ c_{\tau}(\xi \vert_Z) $ and $Q_{\tau}(Z)$ are respectively the relative Chern number and the relative self-intersection number  (see \cite{H4} and \cite{H2}), and $CZ_{\tau}$ is the Conley-Zehnder index.  The  ECH index $I$  only depends  on orbit sets $\alpha$, $\beta$ and the  relative homology class  $Z$.

\paragraph{$J$-Holomorphic currents}
An almost complex structure on $(\mathbb{R} \times Y, d(e^s \lambda))$ is called \textbf{admissible} if $J$ is $\mathbb{R}$-invariant,  $J(\partial_s) =R$, $J(\xi) =\xi$ and $J \vert_{\xi}$ is  $d \lambda$-compatible. We denote set of admissible almost complex structures by $\mathcal{J}(Y, \lambda)$.

 A \textbf{$J$-holomorphic current} from $\alpha$ to $\beta$ is a formal sum $\mathcal{C} =\sum_a d_a C_a$, where $C_a$ are distinct  irreducible simple holomorphic curves, the $d_a$ are positive integers,  $\mathcal{C}$ is asymptotic to $\alpha$ as a current as $s \to \infty $ and  is asymptotic to $\beta$ as a current as $s \to -\infty.$ Fix $Z \in H_2(Y, \alpha, \beta)$. Let $\mathcal{M}_Y^J(\alpha, \beta, Z)$ denote the moduli space of holomorphic currents with relative homology class $Z$.

  Let $C$ be $J$-holomorphic curve in $\mathbb{R} \times Y $ whose positive ends are asymptotic to $\alpha= \Pi \alpha_i^{m_i}$  and negative ends are asymptotic to $\beta= \Pi_j \beta_j^{n_j}$. For each $i$, let $k_i$ denotes the number of ends of $C$ at $\alpha_i$, and let $\{p_{ia}\}^{k_i}_{a=1}$ denote their multiplicities. Likewise,
for each $j$, let $l_j$ denote the number of ends of $C$ at $\beta_j$, and let $\{q_{jb}\}^{lj}_{b=1}$ denote their multiplicities.  The set of numbers $\{p_{ia}\}^{k_i}_{a=1}$ modulo order is  called the \textbf{partition} of $C$ at $\alpha_i$.
The \textbf{Fredholm index} of $C$ is defined by
\begin{eqnarray*}
\operatorname{ind} (C ):= -\chi(C) + 2 c_{\tau}(\xi \vert_{C} ) + \sum\limits_i \sum\limits_{a=1}^{k_i}CZ_{\tau} (\alpha^{p_{ia}}_i) -  \sum\limits_j \sum\limits_{b=1}^{l_j} CZ_{\tau} (\beta^{q_{jb}}_j).
\end{eqnarray*}
 By \cite{H1, H2} or  Section 3.4 of \cite{H4}, if $C$ is a simple holomorphic curve, then the ECH
index and Fredholm index satisfy the following \textbf{ECH index inequality}
\begin{equation} \label{eq20}
I(C) \ge \operatorname{ind}(C) +2\delta(C),
\end{equation}
where $\delta(C)$ is the count of singularities of $C$ with positive integer weights.
\paragraph{Embedded contact homology} Fix a homology class $\Gamma \in H_1(Y, \mathbb{Z})$. The chain group $\ECC(Y, \lambda, \Gamma)$ is a free $\F$-module generated by the ECH generators with homology class $\Gamma$.  Fix a generic  $J \in \mathcal{J}(Y, \lambda)$. The differential  is defined by
\begin{equation*}
\langle \partial \alpha, \beta \rangle: =\sum_{Z \in H_2(Y, \alpha, \beta), I(Z) =1} \# \left(\mathcal{M}_Y^{J}(\alpha, \beta, Z)/ \mathbb{R}\right).
\end{equation*}
Then $\ECH(Y, \lambda, \Gamma)$ is the homology of this chain complex  $(\ECC(Y, \lambda, \Gamma), \partial)$.

\paragraph{$U$ map and $H_1(Y)$-action}
There are two additional structures on $\ECH(Y, \lambda, \Gamma)$ called \textbf{$U$ map} and \textbf{$H_1(Y)$-action}. What follows  paraphrases  Section 1.a of \cite{T5}. Also see Section 2.5 \cite{HT1}.  

The $U$ maps is defined as follows. Fix a point $z\in Y$.  Let $\mathcal{M}_Y^J(\alpha, \beta)\langle z \rangle$ denote the moduli space  of holomorphic currents  passing through $(0, z) \in \mathbb{R} \times Y$. 

The $U$ map is defined by 
\begin{equation*}
\langle U \alpha, \beta \rangle: =\sum_{Z \in H_2(Y, \alpha, \beta), I(Z) =2} \# \left(\mathcal{M}_Y^{J}(\alpha, \beta, Z)\langle z\rangle\right).
\end{equation*}
Then $U$ is a chain map and descends to a degree $-2$ map
$$U: \ECH_*(Y, \lambda, \Gamma) \to \ECH_{*-2}(Y, \lambda, \Gamma) $$ on homology. Moreover, the $U$ map is independent of the choice of $z$.  

Let $\eta \in H_1(Y, \Z)/\rm{Torsion}$. Let $\gamma$ be a 1-cycle representing $\eta$.  Let   $\mathcal{M}_Y^J(\alpha, \beta)\langle \gamma \rangle$ denote the moduli space  of holomorphic currents  passing through $(0, \gamma) \in \mathbb{R} \times Y$. 
Define 
\begin{equation*}
\langle \eta \cdot  \alpha, \beta \rangle: =\sum_{Z \in H_2(Y, \alpha, \beta), I(Z) =1} \# \left(\mathcal{M}_Y^{J}(\alpha, \beta, Z)\langle \gamma \rangle\right).
\end{equation*}
$\eta \cdot $ is a chain map and its desecends to a degree $-1$ map $$\eta: \ECH_*(Y, \lambda, \Gamma) \to \ECH_{*-1}(Y, \lambda, \Gamma) $$ on homology. Moreover,  $\eta \cdot  $ is independent of the choice of  the $1$-cycle.

Let $\mathbb{A}(Y): = \mathbb{F}[U]\otimes \Lambda^*H_1(Y, \F) $. One also can show that the $U$ map and the $H_1(Y, \F)$-action  commute. Thus,  we have an $\mathbb{A}(Y)$-action on  $ \ECH_*(Y, \lambda, \Gamma) .$

\paragraph{Filtered ECH, ECH spectrum and ECH capacities}
Given $L >0$, define $\ECC^L(Y, \lambda, \Gamma)$ to be a submodule generated by the ECH generators with $\mathcal{A}_{\lambda}<L$. Note that the differential $\partial$ decreases the action. Therefore,   $\ECC^L(Y, \lambda, \Gamma)$ is a subcomplex and its homology is well defined, denoted by $\ECH^L(Y, \lambda, \Gamma)$. The group  $\ECH^L(Y, \lambda, \Gamma)$ is called \textbf{filtered ECH}.  For $L<L'$, the inclusion induces a homomorphism
\begin{equation*}
i^{L, L'} : \ECH^L(Y, \lambda, \Gamma) \to \ECH^{L'}(Y, \lambda, \Gamma).
\end{equation*}
We write $i^L: =i^{L, \infty}$.

Suppose that $c_1(\xi)$ is torsion. Then $I(\alpha, \beta, Z)$ is independent of the choice of $Z$ for
null-homologous orbit sets $\alpha $ and $ \beta$ (see Proposition 1.6  of \cite{H1}). Hence, we write it as  $I(\alpha, \beta)$. For any orbit set $\alpha$ with $[\alpha]=0$, define its \textbf{grading} by
\begin{equation} 
\operatorname{gr}(\alpha):= I(\alpha, \emptyset).
\end{equation}
This gives a $\mathbb{Z}$ grading on $\ECH(Y, \lambda, 0)$. Note that  the  $U$ map is a degree $-2$ map with respect to  this grading.
\begin{remark}
In the case that $Y$ is a prequantization bundle, by Lemma 3.11 of \cite{NW}, $c_1(\xi)$ is torsion. Therefore,     the grading (\ref{eq10}) on $\ECH_*(Y, \lambda, 0)$ is well defined.
\end{remark}

There is  a canonical element $[\emptyset] \in  \ECH(Y, \lambda, 0)$ which is represented by the empty orbit set.  The class $[\emptyset] $  is  called  the \textbf{contact invariant.} We remark that $[\emptyset] \ne 0$ if $(Y, \lambda)$ admits a symplectic filling.  So the  contact invariant of the prequantization bundle $\pi: Y \to \Sigma$ is nonzero.

Assume that $\lambda$ is nondegenerate. For $k \in \mathbb{Z}_{\ge 1}$, the \textbf{$k$-th value on ECH spectrum}
is defined by
\begin{equation*}
c_k(Y, \lambda) : = \inf\{L \in \mathbb{R}\vert     \exists   \sigma  \in \ECH^L(Y, \lambda, 0) \mbox{ such that } U^{k}   (\sigma)= [\emptyset]\}.
\end{equation*}
If $\lambda$ is degenerate, define $c_k(Y, \lambda) : =  \lim_{n \to \infty} c_k(Y, f_n\lambda), $ where $f_n:Y\to \mathbb{R}_{>0}$  is a sequence of  smooth  functions such that $f_n\lambda$ is nondegenerate and $f_n$  converges to $1$ in $C^0$ topology.

Let $(X, \omega)$ be an exact filling of a contact manifold $(Y, \lambda)$. The \textbf{$k$-th ECH capacity of $(X, \omega)$ } is defined by 
\begin{equation*}
c_k(X, \omega) :=c_k(Y, \lambda). 
\end{equation*}
For a general symplectic manifold $(X, \omega)$, the \textbf{$k$-th ECH capacity of $(X, \omega)$ } is 
\begin{equation*}
c_k(X, \omega) :=\sup\{ c_k(X', \omega') :  (X', \omega') \hookrightarrow (X, \omega),  \mbox{ $(X', \omega')$ is a Liouville domain} \}.
\end{equation*}
 
\paragraph{ECH cobordism maps}
Now we review  the ECH cobordism maps induced by symplectic cobordisms. To begin with, let us  recap  some concepts.

 A \textbf{strong symplectic cobordism} from $(Y_+, \lambda_+)$ to $(Y_-, \lambda_-)$
is a compact symplectic manifold $(X,\omega) $ such that $\partial X =Y_+ \bigsqcup(-Y_-)$ and $\omega(V_{\pm}, \cdot) =\lambda_{\pm}$ along $Y_{\pm}$, where $V_{\pm}$ are Liouville vector fields defined near $Y_{\pm}$.  If $Y_+ =\emptyset$, then $(X, \omega)$ is called a \textbf{strong symplectic cap} of $(Y_-, \lambda_-)$.  If $Y_- =\emptyset$, then $(X, \omega)$ is called a \textbf{strong symplectic filling} of $(Y_+, \lambda_+)$.

Using $V_{\pm}$, we can identify collar neighborhoods of $Y_{\pm}$ symplectically as $((-\delta, 0] \times Y_+, d(e^s \lambda_+) )$ and  $([0,\delta) \times Y_-, d(e^s \lambda_-) )$. We glue $(X,\omega)$ with cylindrical ends $([0, \infty) \times Y_+, d(e^s \lambda_+))$ and
$(-\infty, 0]\times Y_-, d(e^s\lambda_-))$ along $Y_{\pm}$ to obtain a symplectic manifold $(\widehat{X},\hat{\omega})$, which we call
the \textbf{symplectic completion} $(X,\omega)$.  

An almost complex structure $J$ on $\widehat{X}$ is called \textbf{cobordism admissible} if $J \vert_{\mathbb{R}_{\pm} \times Y_{\pm}}$ are admissible and $J$ is compactible with $\hat{\omega}.$ 
Given a cobordism admissible almost complex structure, $J$-holomorphic currents  and their ECH index, Fredholm index   are  defined similarily in the cobordism setting (see \cite{H2}). Furthermore, the ECH index inequality (\ref{eq20}) still holds. 

Let  $\mathcal{M}^J_X(\alpha_+, \alpha_-)$ denote the moduli space of holomorphic currents in $\widehat{X}$. 
A \textbf{broken holomorphic current}  is a chain of holomorphic currents $\{\mathcal{C}_{N_+}, ...,\mathcal{C}_0,..., \mathcal{C}_{-N_-}\}$ such that 
\begin{itemize}
\item
$\mathcal{C}_k \in \mathcal{M}_{Y_+}^{J_+}(\alpha_+^{k}, \alpha_+^{k-1})$ for $1\le k\le N_+$ and $\alpha_+= \alpha_+^{N_+}$.
\item 
$\mathcal{C}_0 \in \mathcal{M}^J_X(\alpha_+^{0}, \alpha_-^{0}) $
\item
$\mathcal{C}_k \in \mathcal{M}_{Y_-}^{J_-}(\alpha_-^{k+1}, \alpha_-^{k})$ for $-N_-\le k\le -1$ and $\alpha_-= \alpha_-^{-N_-} $.
\item
For $k\ne 0$, $\mathcal{C}_k$ is not $\mathbb{R}$-invariant (as current). 
\end{itemize}
Let $\overline{\mathcal{M}^J_X}(\alpha_+, \alpha_-)$ denote the moduli space of broken holomorphic currents from $\alpha_+$ to $\alpha_-$. 

Fix a relative class $A \in H_2(X, \partial X, \mathbb{Z})$ such that $\partial A =\Gamma_+ -\Gamma_-$.   Let $\rho: H_2(X,\partial X, \mathbb{Z}) \to \mathbb{R}$ be  a homomorphism defined by
 \begin{equation}\label{eq21}
\rho(A) : =\int_A \omega-\int_{\Gamma_+} \lambda_+ + \int_{\Gamma_-} \lambda_-
\end{equation}
One can show that $\rho$ is well defined by Stokes' theorem.

The ECH cobordism maps  package   for strong symplectic cobordisms is summarized in the following theorem. 
\begin{theorem}[Cf. Hutchings-Taubes \cite{HT}] \label{thm2}
Let $(Y_{\pm}, \lambda_{\pm})$ be connected   contact $3$-manifolds.  Let $(X, \omega)$ be a strong symplectic cobordism from $(Y_+, \lambda_+)$ to $(Y_-, \lambda_-)$. Let $\mathbb{A}(X):=\F[U]\otimes \Lambda^*H_1(X, \F)$, where $\F$ is a field.  Then there are homomorphisms
\begin{equation*}
\begin{split}
&\ECH^L(X, \omega, A): \mathbb{A}(X)\otimes \ECH^L(Y_+, \lambda_+, \Gamma_+) \to \ECH^{L+ \rho(A)}(Y_-, \lambda_-, \Gamma_-),\\
&\ECH(X, \omega, A): \mathbb{A}(X)\otimes \ECH(Y_+, \lambda_+, \Gamma_+) \to \ECH(Y_-, \lambda_-, \Gamma_-)\\
\end{split}
\end{equation*}
satisfying  the following properties:
\begin{enumerate}
\item (Composition law)
Let $(X_+, \omega_+)$ be a strong  symplectic cobordism from $(Y_+, \lambda_+)$ to $(Y, \lambda)$ and $(X_-, \omega_-)$ be a strong  symplectic cobordism from $(Y, \lambda)$ to $(Y_-, \lambda_-)$.  Gluing  $(X_+, \omega_+)$  and $(X_-, \omega_-)$ along $(Y, \lambda)$, we obtain a strong symplectic cobordism from $(Y_+, \lambda_+)$ to $(Y_-, \lambda_-)$. Let $\eta_{\pm} \in \mathbb{A}(X_{\pm})$. Then  
\begin{equation*} 
\begin{split}
        &\ECH(X_-, \omega_-, A_-) \big(\eta_- \otimes \ECH(X_+, \omega_+, A_+ )(\eta_+ \otimes - ) \big)\\
        =&\sum_{A|_{X_\pm}=A_\pm}  \ECH(X, \omega, A)( i_{X_-}(\eta_-) i_{X_+}(\eta_+) \otimes -). 
        \end{split}
\end{equation*}
Here $i_{X_{\pm}}: X_{\pm} \to X$ are  the inclusions.

\item (Commute with $\mathbb{A}(Y)$-action)
Let $\eta_{\pm} \in \mathbb{A}(Y_{\pm})$.  Then 
\begin{equation*} 
        \eta_- \cdot  \ECH(X, \omega, A)( \eta\otimes \sigma ) =  \ECH(X, \omega, A)((i_{Y_-})_*(\eta_-)\eta \otimes \sigma)
\end{equation*}
and 
\begin{equation*} 
         \ECH(X, \omega, A) (\eta \otimes   ( \eta_+\cdot \sigma))  =  \ECH(X, \omega, A)( (i_{Y_+})_*(\eta_+)\eta) \otimes \sigma)
\end{equation*}
where $i_{Y_{\pm}} : Y_{\pm} \to X$ are  the inclusions.

\item (Gromov-Taubes invariants)
If $X$ is closed, then 
$$\ECH(X, \omega, A)(\eta) =\Gr(X, \omega, A, \eta), $$
where $\Gr(X, \omega, A,  -): \mathbb{A}(X) \to \F$ is the Gromov-Taubes  invariants defined by C. H. Taubes \cite{Ta2}.

\item (Preservation of the contact invariants)
\begin{equation} \label{eq37}
       \ECH(X, \omega, 0)([\emptyset])= \pm [\emptyset].
\end{equation}

\item (Filtration)
For $L<L'$, we have the following diagram:
\begin{equation*}	 
\begin{CD}
				\mathbb{A}(X) \otimes \ECH^L(Y_+, \lambda_+, \Gamma_+)  @>\ECH^L(X, \omega, A)>>\ECH^{L+\rho(A)}(Y_-, \lambda_-, \Gamma_-) \\
				@VVi^{L, L'}_{Y_+}  V @VV i_{Y_-}^{L+\rho(A), L'+\rho(A)} V\\
\mathbb{A}(X) \otimes \ECH^{L'}(Y_+, \lambda_+, \Gamma_+)) 	 @> \ECH^{L'}(X, \omega, A)>>    \ECH^{L'+\rho(A)}(Y_-, \lambda_-, \Gamma_-) 
			\end{CD}
\end{equation*}

\item (Direct limit)  
 $$\lim_{L \to \infty} \ECH^L(X, \omega, A)=\ECH(X, \omega, A). $$

\item (Holomorphic curves axiom)
Let $J$ be a cobordism admissible almost complex structure. Then for $L \in \R$, there is a chain  map 
$$ \ECC^L(X, \omega, A)_J : \mathbb{A}(X)\otimes \ECC^L(Y_+, \lambda_+, \Gamma_+) \to  \ECC^{L+\rho(A)}(Y_-, \lambda_-, \Gamma_-)$$ inducing the ECH cobordism map $ \ECH^L(X, \omega, A)$  and satisfying the following property: If 
\begin{equation*} 
       \langle \ECC^L(X, \omega, A)_J (\eta \otimes \alpha_+), \alpha_- \rangle \ne 0,
\end{equation*}
then  there is  a broken holomorphic current $\mathcal{C} \in \overline{\mathcal{M}^J_X}(\alpha_+, \alpha_-)$ passing through  a fixed  cycle representing  $\eta$ and   $I(\mathcal{C}) =\deg \eta$.

\end{enumerate}

\end{theorem}

Heuristicaly, the ECH cobordism maps should be defined by counting   holomorphic currents passing through points and $1$-cycles in $(\widehat{X}, \hat{\omega}) $. However, due to the technical  issues mentioned in   Section 5.5 of \cite{H4}, such an approach has not yet been realized. Hutchings and Taubes \cite{HT} define the cobordism maps by using ``ECH =SWF''. Here ``SWF'' is short for the \textbf{Seiberg-Witten Floer homology} which is defined by P.  Kronheimer and T.  Mrowka \cite{KM}.  

Specifically,  one has  a  canonical $\mathbb{A}(Y)$-module isomorphism
$$\mathcal{T}_Y: \ECH_*(Y, \lambda, \Gamma) \to \widehat{HM}^{-*}(Y, \mathfrak{s}_{\Gamma})$$
 between Seiberg-Witten Floer homology and ECH  \cite{T1, T2, T3, T4, T5}.  Here $\mathfrak{s}_{\Gamma}$ is the spin-c structure determined by the contact structure and $\Gamma$.  A filtered version of $\mathcal{T}_Y^L$ is also well defined (Section 3.3 of \cite{HT}). Then the cobordism maps are  defined by 
\begin{equation*}
\begin{split}
& \ECH(X, \omega, A)= \mathcal{T}_{Y_-}^{-1} \circ \widehat{HM}(X, \mathfrak{s}_{A}) \circ \mathcal{T}_{Y_+},\\
&\ECH(X, \omega, A)= (\mathcal{T}^{L+\rho(A)}_{Y_-})^{-1} \circ \widehat{HM}_L(X, \omega, \mathfrak{s}_{A}) \circ \mathcal{T}^L_{Y_+}.
\end{split}
 \end{equation*}
where $\widehat{HM}(X,  \mathfrak{s}_{A})$ is the Seiberg-Witten cobordism map and $\mathfrak{s}_{A}$ is the spin-c structure determined by the symplectic form and $A \in H_2(X, \partial X, \Z)$.   $\widehat{HM}_L(X, \omega, \mathfrak{s}_{A})$ is a filetred version of  Seiberg-Witten cobordism map which defined by   is the Seiberg-Witten equations perturbed by large scaling of $\hat{\omega}$ (see (4.15) and Corollary 5.3  of \cite{HT}). Then the first three  properties holds because the isomorphisms $\mathcal{T}_Y$  intertwines the respective $\mathbb{A}(Y)$-actions (Theorem 1.1 of \cite{T5}) and corresponding properties of Seiberg-Witten cobordism maps.  Property (\ref{eq37}) is proved by M. Echeverria for Seiberg-Witten cobordism maps.  Thus, the same is true for ECH cobordism maps as $\mathcal{T}_Y$ preserves both side's contact invariants. Hutchings and Taubes prove the 5-7  properties in the case that the cobordism is an exact symplectic cobordism \cite{HT}. The analysis in \cite{HT} could be slightly modified to  prove the properties for strong symplectic cobordisms (see Hutchings's unpubished works \cite{H5}). 

 \paragraph{Homology orientations}
Since we are working with $\Q$-coefficient  rather than $\Z/2\Z$, signs matter. To define $ \widehat{HM}(X,  \mathfrak{s}_{A}) $ and $\widehat{HM}_L(X,  \omega, \mathfrak{s}_{A})$, we need to fix   a suitable choice of homology orientation (see Definition 3.4.1 of \cite{KM}) of $X$.  
See Remark 1.4 and  Remark 1.13 of \cite{HT} for the discussion on the coefficients.

Specifically, we want to fix a homology orientation such that the sign in (\ref{eq37}) is positive, in the situations  where we apply Theorem \ref{thm2}.

We briefly recall the orientation convention for the Seiberg-Witten cobordism maps. Let $X$ be a cobordism  from $Y_+$ to $Y_-$. Given two monopoles $\mathfrak{c}_{\pm}$ on $Y_{\pm}$, let $\mathfrak{M}(\mathfrak{c}_+, \mathfrak{c}_-) $ denote the moduli space of solutions to 4-dimensional   Seiberg-Witten equations which are asymptotic to  $\mathfrak{c}_{\pm}$. Associated to the determinant line of  linearized Seiberg--Witten equations,   two possible orientations form an orientation set which we denote by $\Lambda(\mathfrak{c}_+, \mathfrak{c}_-)$. Namely, a choice of element in  $\Lambda(\mathfrak{c}_+, \mathfrak{c}_-)$ determines the orientation  of $\mathfrak{M}(\mathfrak{c}_+, \mathfrak{c}_-) $. There is a canonical identification
$$\Lambda(\mathfrak{c}_+, \mathfrak{c}_-) \cong \Lambda(\mathfrak{c}_+)  \otimes_{\Z/2\Z} \Lambda(X)\otimes_{\Z/2\Z}\Lambda(\mathfrak{c}_-)$$
where \(\Lambda(\mathfrak{c}_{\pm})\)  are  the orientation set associated to   monopoles $\mathfrak{c}_{\pm}$ and $\Lambda(X)$ is the set of homology orientations of $X$ (Definition 3.4.1 of \cite{KM}). Thus, after fixing an orientation $o(\mathfrak{c})\in\Lambda(\mathfrak{c})$ and  a  homology orientation $o(X)\in\Lambda(X)$ determines the signs with which nondegenerate solutions are counted in the cobordism map.

If  $(X, \omega)$ is a closed symplectic manifold, then the symplectic form induces   a canonical homology  orientation $o(X, \omega) \in \Lambda(X)$  (see Page 460-461 of \cite{Ta4}).  We pick such a homology  orientation  throughout.

Now consider the following situation.  Let $(X_+, \omega_+)$ be  a strong symplectic cap of $(Y, \lambda)$ and $(X_-, \omega_-)$ a strong symplectic  filling  of $(Y, \lambda)$.   We claim that there is a choice of homology orientation $o(X_{\pm}) \in \Lambda(X_{\pm})$ such that 
$$\ECH(X_+, \omega_+, 0)(1) =[\emptyset] \mbox{ and } \ECH(X_-, \omega_-, 0)([\emptyset]) =1. $$

The choice is as follows. Gluing   $(X_+, \omega_+)$ and   $(X_-, \omega_-)$ along $(Y, \lambda)$, we obtain a closed symplectic manifold $(X, \omega)$.    
 Let $\mathfrak{c}_{\emptyset}$ denote the monopole corresponding to the empty ECH generator. That is $$
\Phi^r(\emptyset)=\mathfrak{c}_{\emptyset},$$
where $\Phi^r$ is the map in Theorem 1.1 of \cite{T2}. 
   By Theorem 1.1 of \cite{KM1},  we have a perfect orientation 
 $$o(X_-,\omega_-)\in
\Lambda(\mathfrak{c}_{\emptyset})\otimes_{\Z/2\Z}\Lambda(X_-)$$
such that $\ECH(X_-, \omega_-, 0)([\emptyset]) =1$. 
We choose the homology orientations  $o(X_{\pm}) \in \Lambda(X_{\pm})$ by requiring
\begin{equation} \label{eq39}
  o(\mathfrak{c}_{\emptyset})\otimes o(X_-) =o(X_-,\omega_-)  \mbox{ and } o(X_+)  o(X_-) =o(X, \omega). 
\end{equation}

With the above  choice, we claim that   $ \ECH(X_+, \omega_+, 0)(1) =[\emptyset]$.  Assume $ \ECH(X_+, \omega_+, 0)(1) =\tau[\emptyset]$, where $\tau=1$ or $\tau =-1$.  By the  composition law, we have 
$$\tau = \ECH(X_-, \omega_-, 0) (\ECH(X_+, \omega_+, 0)(1)) =\sum_{A \vert_{X_{\pm}} =0} \Gr(X, \omega, A). $$
Note that  $A \vert_{X_{\pm}} =0$ implies that $\int_A \omega =0$.  Thus, there is no holomorphic currents with homology class $A$ except the empty current.  Thus, the right hand of the above equation is  $\Gr(X, \omega, 0) =1$.  Hence, $\tau =1$. 

We pick the homology orientation satisfying (\ref{eq39}) throughout.

\subsection{Cohomology rings of symmetric products  of surfaces }
 
In this section, we recap the cohomology ring of the symmetric product of  a surface $\Sym^n\Sigma$, computed by Macdonald \cite{Mac}.  For our purpose, we assume that $n >2g-2$ throughout, unless otherwise stated.

Let $\mathbb{A}(\Sigma):=\Q[u]\otimes \Lambda^*H^1(\Sigma, \Q),$ where $u$ is a formal variable with degree $2$.    Let
$$ \{\alpha_i\}_{i=1}^g,   \qquad   \{\beta_i\}_{i=1}^g
$$
be symplectic generators of $H_1(\Sigma;\Z)$, with
$\alpha_i\cap \beta_i=1$.  We identify them with elements of
$H^1(\Sigma;\Q)$ by Poincar\'e duality.

Define
$$
        \theta:=\sum_{i=1}^g\alpha_i  \beta_i \in \Lambda^2 H^1(\Sigma, \Q) \subset \mathbb{A}(\Sigma).
$$
For $0\le r\le g$, define the primitive part
$$
        \Lambda_0^r
        :=\ker\Bigl\{
        \theta^{g-r+1}:
        \Lambda^rH^1(\Sigma, \Q)\longrightarrow
        \Lambda^{2g-r+2}H^1(\Sigma,\Q)
        \Bigr\}.
$$
Let
       $ \{e_r^c\,\}_{c=1}^{N_r}
$
be a basis of $\Lambda_0^r$, where $N_r:=\dim \Lambda_0^r$.  Then
$$
        \mathcal{E}
        :=\Bigl\{\theta^b e_r^c\;\Big|\;
        0\le r\le g,\;1\le c\le N_r,\;0\le b\le g-r
        \Bigr\}
$$
is a basis of $\Lambda^*H^1(\Sigma, \Q)$. Therefore, any $x \in \mathbb{A}(\Sigma)$ is  a  linear combination of $u^a\theta^b e_r^c$, where $ \theta^b e_r^c \in \mathcal{E}$.  The degree of elements  $u^a\theta^b e_r^c$  is defined to be 
\begin{equation} \label{eq38}
\deg(u^a\theta^b e_r^c) =2a + 2b + r. 
\end{equation} 

   Define a polynomial 
\begin{equation} \label{eq1}
\mathfrak{R}_n(u, \theta):=u^{n-2g+1}\prod_{i=1}^g(\alpha_i\beta_i-u)  =\sum_{k=0}^{g} (-1)^{g-k}  \frac{\theta^k}{k!}u^{n-g-k+1}.
\end{equation}
According to Macdonald's result ((6.3) in \cite{Mac}), we know that  
$$H^*(\Sym^n \Sigma, \Q) \cong \mathbb{A}(\Sigma)/\langle \mathfrak{R}_n(u, \theta) \rangle$$
provided that $n>2g-2$.   Then we have a natural projection    $$\Pi:   \mathbb{A}(\Sigma) \to H^*(\Sym^n \Sigma, \Q).$$ 
 $\Pi$ is a ring homomorphism, i.e., $\Pi(x_1x_2) =\Pi(x_1) \wedge \Pi(x_2)$. 
According to Macdonald's result \cite{Mac}, we have 
\begin{itemize}
\item
$\Pi$ maps  
\begin{equation}\label{eq17}
\begin{split}
        \mathcal{B}_n
        :=&\{ \mathcal{E}, u\mathcal{E}, ...,u^{n-g}\mathcal{E}\}\\
        =&\Bigl\{u^a\theta^b e_r^c\;\Big|\;
        0\le r\le g,\;1\le c\le N_r,\;0\le b\le g-r,
        \;0\le a\le n-g
        \Bigr\}.
        \end{split}
\end{equation}
 to a basis of  $H^*(\Sym^n \Sigma, \Q)$.  In particular,
$$ \dim H^*(\Sym^n\Sigma, \Q)=2^{2g}(n-g+1). $$

\item
For $x \in \mathcal{B}_n$, we have 
$$\deg \Pi(x) =\deg x. $$
The degree on the left-hand side of the above equation corresponds to the degree of the cohomology $H^*(\Sym^n \Sigma, \Q)$, whereas the degree on the right-hand side is defined by (\ref{eq38}).
\item
$\ker \Pi=\{P \mathfrak{R}_n(u, \theta)\vert P \in \mathbb{A}(\Sigma)\}$.
\end{itemize}

Let $q^n: H^{*}(\Sym^n \Sigma, \Q) \otimes H^{*}(\Sym^n \Sigma,\Q)   \to \mathbb{Q}$ denote the \textbf{intersection pairing} as follows. Let  $x_1, x_2 \in  \mathcal{B}_n$. If $\degop x_1 + \degop x_2  \ne 2n$,   set $q^n(\Pi(x_1), \Pi(x_2)) :=0.$  If $\degop x_1+ \degop x_2 =2n$,    define  $q^n(\Pi(x_1), \Pi(x_2)) $ to be the number satisfying 
\begin{equation}
\Pi(x_1) \wedge \Pi(x_2) =q^n(\Pi(x_1), \Pi(x_2)) \frac{1}{g!}\Pi(u^{n-g}\theta^g). 
\end{equation}
  Finally, we extend $q^n$  bilinearly over $\mathbb Q$.

In the case that  $\degop x_1 + \degop x_2 =2n$,  then $x_1x_2 - q^n(\Pi(x_1), \Pi(x_2)) \frac{1}{g!} u^{n-g}\theta^g \in \ker \Pi$.
Therefore, 
\begin{equation} \label{eq10}
x_1x_2=q^n(\Pi(x_1), \Pi(x_2))u^{n-g} \theta^g/g! +P\mathfrak{R}_n(u, \theta) 
\end{equation}
 for some homogenous element $P \in \mathbb{A}(\Sigma) $  with  $\degop P = \degop x_1+ \degop x_2- 2(n-g+1) =2g-2$.

\begin{remark} \label{remark4}
By a direct computation, we have 
$$\frac{1}{k!} \theta^k =\sum_{1\le i_1<i_2 <....<i_k\le g} \alpha_{i_1}\beta_{i_1}...\alpha_{i_k}\beta_{i_k}.$$
Let $\theta^{[k]}$ denote the right hand side of the above equation.  Even though $\frac{1}{k!}$ is not well  in $\Z$ or $\Z/2\Z$, the  $\theta^{[k]}$  is  still well defined in $\Z$ or $\Z/2\Z$.  Thus, polynomial $\mathfrak{R}_n(u, \theta)$ is  well defined over    $\Z$ or $\Z/2\Z$.  

However,    simply replacing  $\theta^b$ by $\theta^{[b]}$  in $\mathcal{B}_n$ cannot supply a basis of  $H^{*}(\Sym^n \Sigma, \Z)$ or $H^{*}(\Sym^n \Sigma, \Z/2\Z)$.  For example, taking $g=2$ and $\F=\Z/2\Z$, then $\theta^{[1]}\theta^{[1]} =  0$.   Then $\theta^{[1]}  \in \Lambda^2_0$. Thus, any  basis $\{e_2^c\} $ of $\Lambda^2_0$ together with $\theta^{[1]} $ cannot be linear independent. 

 \end{remark}

\section{Holomorphic currents in  line bundles}
Let $\pi_{E^{\vee}}: E^{\vee}\to \Sigma$ be the dual line bundle of $E$. Hence, $c_1(E^{\vee}) =[ \omega_{\Sigma}]$.  Let 
$$\mathbb{D}E:=\{ v \in E \vert |v|_h \le 1\} \mbox{ and } \mathbb{D}E^{\vee}:=\{ v \in E^{\vee} \vert |v|_{h^{\vee}} \le 1\} $$
be the unit disk subbundles. Let $Y=\{\rho =1\}$ and $Y^{\vee} =\{r=1\}$. Note that $Y^{\vee}  \cong -Y$. The minus sign here meaning reversing the fiber orientation.   Similar to the case of $E$, we have a global angular form $\alpha_{E^{\vee}}  \in  \Omega^1(E^{\vee}\setminus \Sigma, \R)  $  such that $d\alpha_{E^{\vee}}  =-\pi_{E^{\vee}}^*\omega_{\Sigma}$. The global angular form is locally defined by 
\begin{equation}\label{eq33}
\alpha_{E^{\vee}}   : = \frac{1}{2\pi}\left( d\theta^{\vee} -iA_{\nabla}^{\vee}\right)  , 
\end{equation}
where $d\theta^{\vee}$ is the angular form of the dual fiber.  Under the  fiber orientation reversing identification $Y^{\vee} $ and $Y$, we have   $d\theta^{\vee} = -d \theta$ and $A_{\nabla}^{\vee} =-A_{\nabla}$, where $d\theta$  is the angular form of the  fiber of $E$.  Thus, $\alpha_{E^{\vee}}=-\alpha_E$.

Define two $2$-forms  $\Omega $ on $\mathbb{D}E$ and $\Omega^{\vee}$ on $\mathbb{D}E^{\vee}$  respectively  by 
\begin{equation*}
\Omega :=\pi_{E}^*\omega_{\Sigma} + d(\rho^2 \alpha_{E}) \mbox{ and }  \Omega^{\vee}:=3\pi_{E^{\vee}}^*\omega_{\Sigma} + d(r^2 \alpha_{E^{\vee}}),
\end{equation*}
where $\rho, r$ are  the radius coordinates  of $E$ and  $E^{\vee}$ defined by the metric $h$. Extend $\Omega$ and $\Omega^{\vee}$ over the zero section $\Sigma$  smoothly. One can show that both of them are symplectic.

 By a direct computation, the Liouville vector fields $Z, Z^{\vee}$ on $\mathbb{D}E  \setminus \Sigma$ and  $\mathbb{D}E^{\vee} \setminus \Sigma$ are respectively 
$$Z=\frac{1+\rho^2}{2 \rho^2}\rho \partial_{\rho} \mbox{ and }Z^{\vee}=\frac{r^2-3}{2r^2}\,r\partial_r.$$
Therefore, the induced contact form on $Y=\{\rho =1\}$ and $Y^{\vee} =\{r=1\}$ are respectively 
 $$ \lambda= \Omega(Z, \cdot )  \mbox{ and }   \Omega^{\vee}(Z^{\vee}, \cdot )  =-2\alpha_{E^{\vee}} =\lambda. $$
 Note that $Z$ is outward pointed and   $Z^{\vee}$ is inward pointed. Thus,   $(\mathbb{D}E, \Omega)$ is a strong symplectic filling of $(Y, \lambda)$ while  $(\mathbb{D}E^{\vee}, \Omega^{\vee})$ is a strong symplectic cap of $(Y, \lambda).$

Denote $(X_+, \omega_+) :=(\mathbb{D}E^{\vee}, \Omega^{\vee})$ and $(X_-, \omega_-):=(\mathbb{D}E, \Omega)$.  Let $\Sigma_+$ and $\Sigma_- $ denote the zero sections of $\mathbb{D}E^{\vee}$ and $\mathbb{D}E$ respectively.

\subsection{Morse-Bott perturbation and  Reeb orbits}
  Fix    a perfect Morse function 
        $H:\Sigma\longrightarrow \R$
with $|H|_{C^0}\le 1/10$.  Let $\varepsilon>0$ be a small positive number. Define a perturbed contact form on $Y$ by
\[
        \lambda_{\varepsilon} :=(1+\varepsilon \pi^*H )\lambda. 
\]
By Lemma 3.1 of \cite{NW},  there exists $L_{\varepsilon}>0$ such that the Reeb orbits  wtih action smaller than $L_{\varepsilon}$ are covers of the $S^1$-fibers over critical points of $H$. We work with the Reeb orbits with $\mathcal{A}_{\lambda_{\varepsilon}}<L_{\varepsilon}$ throughout. 

Let $\alpha$ be an orbit set with $\mathcal{A}_{\lambda_{\varepsilon}}(\alpha)<L_{\varepsilon}$. Then 
$$
        \alpha
        =e_+^{m_+}e_-^{m_-}\prod_{i=1}^{2g}h_i^{m_i},
$$
where
\begin{itemize}[leftmargin=2.5em]
    \item $e_+$  is the fiber which lies over the maximum of $H$,
    \item $e_-$  is the fiber which lies over the minimum of $H$,
    \item $h_i$   is the fiber which lies over the $i$-th index-one critical point of $H$.
\end{itemize}
Let 
$$
         M_{\alpha}:=m_+ +m_-+\sum_{i=1}^{2g}m_i
$$
to be the total multiplicity of $\alpha$.
Fix $\Gamma \in \mathbb{Z}/|e|\mathbb{Z}$. Note that the $S^1$-fiber of $Y$ has homology class $1 \mod |e|$. Therefore,  if $[\alpha]=\Gamma$, then we have a nonnegative integer $d(\alpha)$ such that 
\[
        M_{\alpha}=d(\alpha)|e| +\Gamma.
\]
The  action  of $\alpha$ is 
$$
        \mathcal{A}_{\lambda_{\varepsilon}}(\alpha) = \int_{\alpha}{\lambda_{\varepsilon}}  =  2\bigl(M_{\alpha} + \varepsilon\pi^*H(\alpha)\bigr),
$$
where $\pi^*H(\alpha)$ is short for $\sum_{i} m_i H(\pi(\alpha_i))$.

We perturb   symplectic manifolds $(X_+, \omega_+)$ and $(X_-, \omega_-)$   so that they are respectively  strong  symplectic cap and  strong  symplectic filling of $(Y, \lambda_{\varepsilon}). $ The construction is as follows. 

By using the Liouville vector field, a collar neighborhood  of $(Y, \lambda_{\varepsilon})$  in $X_+$ is symplecticmorphic  to 
$$\big( [0, \delta)_s\times Y, d(e^s \lambda)\big).$$ 
Fix a non-increasing  cutoff function $\phi(s)$ such that  $\phi(s) =1$ when $s \le \frac{1}{4}\delta$ and  $\phi(s)=0$ when $s \ge \frac{1}{2}\delta$.   Then we define 
$$
\omega^{\varepsilon}_+=
\begin{cases}
\omega_+, & X_+ \setminus  \big( [0, \delta)_s\times Y, d(e^s \lambda)\big), \\
d\big(e^s ( (1+ \phi(s)\varepsilon  \pi^*H)\lambda)  \big), &  [0, \delta)_s\times Y, d(e^s \lambda). \\
\end{cases}
$$
The perturbed symplectic filling $(X_-, \omega^{\varepsilon}_-)$ is constructed similarly (see (3.4) of \cite{GHC} for example).

\subsection{Holomorphic currents in the cap-side}

The main purpose of this section is study the properties of the holomorphic currents in the symplectic completion of   $(X_+, \omega_+^{\varepsilon}).$ The argument is essentially parallel to   the results  for the holomorphic currents in the filling side (Section 3.2   of \cite{GHC}).  We deduce a constrain on the degree of  holomorphic currents.  This simple constrain forces the image of the cobordism map $\ECH(X_+, \omega_+, A)(x)$ lies in a certain truncation of  $\ECH(Y, \lambda_{\varepsilon}, \Gamma) $ (Lemma \ref{lem2}).

The orbit set  $\alpha$ is assumed to satisfies $\mathcal{A}_{\lambda_{\varepsilon}} (\alpha)<L_{\varepsilon}$ throughout.  As a result, $\alpha$ consists of $S^1$-fibers over critical points of $H$.

To begin with, we first illustrate the form of the relative homology classes in $H_2(X_+, \emptyset, \alpha)$.      
For each  orbit set $\alpha
        =e_+^{m_+}e_-^{m_-}\prod_{i=1}^{2g}h_i^{m_i}$, we define a relative homology class $Z^{\vee}_{\alpha} \in H_2(X_+, \emptyset, \alpha)$ which is represented by 
        $$m_+ F_{\pi(e_{+})}^{\vee}  +  \sum_{i=1}^{2g} m_i F_{\pi(h_i)}^{\vee} + m_- F_{\pi(e_{-})}^{\vee},$$
        where $F^{\vee}_p $ denote the disk fiber of $\mathbb{D}E^{\vee}$ over $p \in \Sigma$.  Because $H_2(X_+, \emptyset, \alpha)$ is an affine space of $H_2(X_+, \Z) \cong \Z[\Sigma_+]$, every element in $H_2(X_+, \emptyset, \alpha)$ can be written as $Z_{\alpha}^{\vee}+k[\Sigma_+]$ for some $k \in \Z$.  
     
By Poincar\'e duality, $H_2(X_+, Y, \Z) \cong H^2(X_+, \Z) \cong H^2(\Sigma_+, \Z)  \cong \Z$.   The generator is represented by the disk fiber  $F^{\vee}$ of $\mathbb{D}E^{\vee}$. 
Fix $A \in H_2(X_+, Y, \Z)$ such that $\partial A =\Gamma \in \mathbb{Z}/|e|\Z$. Given an orbit set $\alpha$ with total multiplicity $M_{\alpha}=d(\alpha)|e| + \Gamma$, let  
$$Z_{\alpha, A}^\vee\in H_2(X_+, \emptyset, \alpha)$$
be a relative homology class  such that $[Z_{\alpha, A}^\vee] =A$.  
Write $Z_{\alpha, A}^{\vee}= Z_{\alpha}^{\vee} +k[\Sigma_+]$.  Then
$$A\cdot \Sigma_+=(Z_{\alpha}^{\vee} +k[\Sigma_+]) \cdot \Sigma_+ =M_{\alpha} + k|e|.$$
Therefore, $k=\frac{(A\cdot \Sigma_+ -M_{\alpha})}{|e|} .$ Hence,  $Z_{\alpha, A}^\vee$ is uniquely determinede by $\alpha$ and $A$. 
Let $d(A)$ be the constant such that
$$A \cdot \Sigma_+ =d|e| +\Gamma. $$
Then $$Z_{\alpha, A}^{\vee} = Z_{\alpha}^{\vee} + (d(A)-d(\alpha))[\Sigma_+].$$

We begin with computing  the ECH index of relative homology classes in $X_+$.   
\begin{lemma}  \label{lem13}
With the above conventions, the ECH index of relative homology  $Z_\alpha^\vee+k[\Sigma_+]$ is 
$$
        I_{\mathrm{cap}}\bigl(Z_\alpha^\vee+k[\Sigma_+]\bigr)
        =M_{\alpha}-m_+ +m_-+2kM_{\alpha}+k\chi(\Sigma)+|e| k(k+1).
$$
\end{lemma}

\begin{proof}
Recall the ECH index formula
$$
        I(Z)=c_\tau(Z)+Q_\tau(Z) -CZ^{ech}_{\tau}(\alpha).
$$
For $p \in \operatorname{Crit} H$, fix  a   trivialization   on $T_p\Sigma$. Such a trivialzation   can be lifted to
a trivialization $\tau$ of $\xi \vert_{\gamma_p(t)}$.  Under this trivialization, by Lemma 3.9 of  \cite{NW}, we have
\begin{equation} \label{eq5}
CZ_{\tau}(\gamma_p^k) ={\rm{ind}}_p H -1, 
\end{equation}
where  ${\rm{ind}}_p H$ is the Morse index at $p$.

First consider the  relative homology class $Z_\alpha^\vee$.  We regard the connection $\nabla$ as a map $A_{\nabla}^{\vee}  : TE^{\vee} \to E^{\vee}$. This gives a spliting 
$$TE^{\vee} =TE^{\vee}_{hor} \oplus TE^{\vee}_{vert},$$
where $TE^{\vee}_{hor} =\ker A_{\nabla}^{\vee} $ and  $TE^{\vee}_{vert} =\pi_{E^{\vee}}^*E^{\vee}.  $ The trivialization of $T_p\Sigma$ also can be lifted to  $T_pE^{\vee}_{hor}$.  Thus,
$c_{\tau}(T \mathbb{D}E^{\vee}_{hor} \vert_{F_p^{\vee}} ) =0 $. 
The horizational bundle  $T \mathbb{D}E^{\vee}_{hor}$ also can be regarded as a normal bundle of $F_p^{\vee}$. Since  $F_p^{\vee}$ is embedded, we have 
$$Q_{\tau}(F_p^{\vee}) =c_{\tau}(T\mathbb{D}E^{\vee}_{hor} \vert_{F_p^{\vee}} ) =0. $$
If $p, q$ are distinct critical points, then $ F_p^{\vee} $ and $F_q^{\vee}$ are disjoint. Hence, $Q_{\tau}(F_p^{\vee}, F_q^{\vee})  =0 $. 
Because $Q_{\tau}$ is  quadratic,  we have 
\begin{equation} \label{eq6}
Q_{\tau}(Z_{\alpha})  =\sum_{p} m_p^2 Q_{\tau}( F_p^{\vee})   + 2 \sum_{ p \ne q } m_p m_q Q_{\tau}(F_p^{\vee}, F_q^{\vee}) =0.
\end{equation} 

Along $\pi^{*}_{E^{\vee}}E^{\vee} \vert_{F_p}=E_p$, the trivialization is given by a section $ \psi = - \partial_{\theta^{\vee}} \wedge Z^{\vee} = \frac{3-r^2}{2r^2} \partial_{\theta^{\vee}} \wedge r\partial_r$ near $\partial \mathbb{D} E_p$,  where $r$ is the radius coordinate.  Such a section   $\psi$ can be extended to the whole $ \mathbb{D} E_p$ by setting  $\psi =\partial_{\theta^{\vee}} \wedge r\partial_r$.  Thus, 
$c_{\tau}(T\mathbb{D}E^{\vee}_{vert} \vert_{F_p^{\vee}} )  = \#\psi^{-1}(0)=1.$
Sicne relative Chern number is additivity, we have 
\begin{equation} \label{eq7}
c_{\tau}(T\mathbb{D}E^{\vee}\vert_{Z_{\alpha}} )   =\sum_{p} m_p  c_{\tau}(T\mathbb{D}E^{\vee}_{hor} \vert_{F_p^{\vee}} )   +  m_p c_{\tau}(T\mathbb{D}E^{\vee}_{vert} \vert_{F_p^{\vee}} )  =M_{\alpha}. 
\end{equation} 
Combining (\ref{eq5}), (\ref{eq6}), (\ref{eq7}) together, we have 
$$I_{cap}(Z^{\vee}_{\alpha}) = M_{\alpha}-m_+ +m_.$$

Now add $k$ copies of the section $\Sigma_+$.  By definition, 
\begin{equation*}  
I_{cap}(Z^{\vee}_{\alpha} + k[\Sigma_+]) =I_{cap}(Z^{\vee}_{\alpha})  + 2 k Z^{\vee}_{\alpha} \cdot [\Sigma_+] +    kc_1(T\mathbb{D}E^{\vee} \vert_{\Sigma_+})+k^2[\Sigma_+]^2.
\end{equation*} 
The intersection-term between
$Z_\alpha^\vee$ and $k[\Sigma_+]$  is 
        $2kM_{\alpha},$
because each fiber disk intersects $\Sigma_+$ once and the total multiplicity
is $M_{\alpha}$.

By Adjunction formula, we have 
\begin{equation}\label{eq8}
\begin{split}
        c_1(T\mathbb{D}E^{\vee}|_{\Sigma_+})
       = \langle c_1(T\Sigma), \Sigma \rangle +[\Sigma_+]^2  =\chi(\Sigma)+|e|. 
\end{split}
\end{equation}
 
Combining the reference contribution, the cross-term, and the section
contribution gives
\[
        I_{\mathrm{cap}}\bigl(Z_\alpha^\vee+k[\Sigma_+]\bigr)
        =M_{\alpha}-m_+ +m_-+2kM_{\alpha}+k\chi(\Sigma)+|e| k(k+1).
\]
\end{proof}

\begin{remark} \label{remark1}
Assume $\Gamma=0$. By Proposition 1.3  of \cite{NW}, we have 
$$\operatorname{gr}(\alpha)= I(\alpha, \emptyset) =d(\alpha) \chi(\Sigma) +M_{\alpha} d(\alpha) +m_+- m_-. $$
Fix $A \in H_2(X_+, Y,\Z)$ with $\partial A =0$ and $A \cdot \Sigma_+ =d(A)|e|$.  Then $Z_{\alpha, A}^{\vee} = Z_{\alpha}^{\vee} + (d(A)-d(\alpha))[\Sigma_+]$. Combining these results with  Lemma \ref{lem13}, we have 
\begin{equation*}
I_{cap}(Z_{\alpha, A}^{\vee}) =d(A)(d(A)+1)|e| + d(A) \chi(\Sigma) -\operatorname{gr}(\alpha).
\end{equation*}
\end{remark}

Using the computations in  Lemma \ref{lem13}, we obtain a formula for the Fredholm index of a holomorphic curves in $\widehat{X}_+. $
\begin{lemma}
Let $C $ be a $J$ holomorphic curve with relative homology class $Z_{\alpha} +k [\Sigma_+]$.   Then 
$$\operatorname{ind} C= 2g(C)-2 + h(C)+2e_-(C) + 2M_{\alpha} + 2k|e| +2k \chi(\Sigma),$$
where $e_-(C)$ denote the number of ends at covers of $e_-$ and $h(C)$ the number of ends at hyperbolic orbits. 
\end{lemma}
\begin{proof}
Let   $e_+(C)$ denote the number of ends at covers of $e_+$.  
By definition of the Fredholm index and (\ref{eq5}), (\ref{eq7}),  (\ref{eq8}), we have 
\[
\begin{aligned}
    \operatorname{ind} C=&2g(C)-2 +h(C) +e_+(C) +e_-(C) + 2c_{\tau}(Z_{\alpha} +k [\Sigma_+]) -e_+(C) + e_-(C)\\
=&2g(C)-2 +h(C) +2e_-(C) + 2M_{\alpha} +  2k \chi(\Sigma) + 2k|e|.
\end{aligned}
\]
\end{proof}

Next, we compute the energy of the holomorphic currents.

\begin{lemma} \label{lem8}
Let $\mathcal{C}=\{\mathcal{C}_0, \mathcal{C}_{-1}, ..., \mathcal{C}_{-N}\} \in \overline{\mathcal{M}_{X_+}^J}(\emptyset, \alpha)$ be a broken  holomorphic current with total relative homology class $Z^{\vee}_{\alpha} + k[\Sigma_+]$. Then the total  energy  of $\mathcal{C}$ is 
$$\begin{aligned}
    E_{\omega_+^{\varepsilon}}(\mathcal{C})  =  E_{\omega_+^{\varepsilon}}(\mathcal{C}_0)  + \sum_{i=-1}^{-N} E_{d\lambda_{\varepsilon}} (\mathcal{C}_i)=3(M_{\alpha} +k |e|) - \mathcal{A}_{\lambda_{\varepsilon}}(\alpha). 
\end{aligned}$$
Moreover,  we have  $\mathcal{A}_{\lambda_{\varepsilon}}(\alpha) \le 3(M_{\alpha} +k |e|) =3 [\mathcal{C}] \cdot [\Sigma_+] $.
\end{lemma}
\begin{proof}
It suffices to prove the case that $\mathcal{C}$ is unbroken.  Then its energy is 
\begin{equation*}
\begin{split}
   E_{\omega_+^{\varepsilon}}(\mathcal{C}) &= \int_{\mathcal{C} \cap X_+} \omega_+^{\varepsilon}  + \int_{\mathcal{C} \cap  \mathbb{R}_{\le 0} \times Y} d\lambda_{\varepsilon}\\
&= \int_{\mathcal{C} \cap X_+} \omega_+^{\varepsilon}  +  \int_{\mathcal{C} \cap \partial X_+} \lambda_{\varepsilon} -\mathcal{A}_{\lambda_{\varepsilon}}(\alpha).\\
        \end{split}
\end{equation*}
Note that  $\int_{\mathcal{C} \cap X_+} \omega_+^{\varepsilon}  +  \int_{\mathcal{C} \cap \partial X_+} \lambda_{\varepsilon} $ is   the homomorphism $\rho(A)$ (\ref{eq21}), where $A $ denote the  homology class $[\mathcal{C}] \in H_2(X_+, Y, \Z). $ Becasue the exact perturbations on $\omega_+^{\varepsilon} $ and $\lambda_{\varepsilon}$  canncels each other, we have 
 \begin{equation*}
\begin{split}
\rho(A) =\int_A \omega_+ + \int_{\partial A} \lambda. 
        \end{split}
\end{equation*}
Then 
 \begin{equation*}
\begin{split}
\rho(A) &=\int_{Z_{\alpha} + k[\Sigma_+]} (3\pi^*_{E^{\vee}} \omega_{\Sigma} - d(r^2\alpha_{\nabla})) + 2M_{\alpha}\\
&= 3k|e| + M_{\alpha} +2M_{\alpha} =3(k|e| + M_{\alpha}).   
        \end{split}
\end{equation*}

Because the energy of  a holomorphic curve is nonnegative, we have $$3(M_{\alpha} +k |e|) - \mathcal{A}_{\lambda_{\varepsilon}}(\alpha) \ge 0. $$
\end{proof}

\begin{remark}
The   computations in Lemma \ref{lem8} show that $\rho(A) =3 A \cdot [\Sigma_+]$. 
\end{remark}

Combining the index and energy computations,  in the next two lemmas, we deduce a  contraint of a relative homology class provided that the class contains a holomorphic current.   

\begin{lemma} \label{lem6}
Assume $J$ is a generic cobordism admissible almost complex structure. Let $C \in \mathcal{M}_{X_+}^J(\emptyset, \alpha)$ be a simple holomorphic curve with relative homology class 
$[C]=Z_\alpha^\vee+k[\Sigma_+] $.  Then $k \ge 0$. 
\end{lemma}
\begin{proof}
If $\alpha$ is the empty set, i.e., $C$ is closed, then  the  homology class of $C$ is $[C] =k[\Sigma_+]$. The energy of $C$ is 
$$\int_C \omega_+^{\varepsilon} =3k |e|\ge0. $$
Hence, $k \ge 0$. 

Assume $\alpha$ is nonempty.  Lemma \ref{lem8} implies that $3(M_{\alpha} +k |e|)\ge \mathcal{A}_{\lambda_{\varepsilon}}  (\alpha)>0 $.  Therefore, 
\begin{equation} \label{eq9}
M_{\alpha} \ge 1-k|e|. 
\end{equation}
We  divide the situation into the  following three cases. 
\begin{itemize}
\item
 If $k \le -2$, then 
\begin{equation*}
\begin{split}
 I_{\mathrm{cap}}\bigl(Z_\alpha^\vee+k[\Sigma_+]\bigr)
       &=M_{\alpha}-m_+ +m_- + kM_{\alpha} + k(M_{\alpha} + k|e|   + |e|  + \chi(\Sigma))  \\
       &\le  -m_+ +m_- + -M_{\alpha}  + 2k  \le -4< 0. 
        \end{split}
\end{equation*}
Here we have used the observation that $M_{\alpha} + k|e|   + |e|  + \chi(\Sigma)>1$.  This follows from (\ref{eq9}) and $e< \chi(\Sigma)$.

\item
If $k=-1$ and $\chi(\Sigma)=2$, then
 \begin{equation*}
\begin{split}
 I_{\mathrm{cap}}\bigl(Z_\alpha^\vee+k[\Sigma_+]\bigr)
       &=-M_{\alpha}-m_+ +m_- -2\le -2<0. 
        \end{split}
\end{equation*}
\item
If $k=-1$ and $\chi(\Sigma)\le 0$, then
\begin{equation*}
\begin{split}
&\operatorname{ind} C-  I_{\mathrm{cap}}\bigl(Z_\alpha^\vee-[\Sigma_+]\bigr) \\
=&  2g(C)-2 + h(C)+2e_-(C) + 2M_{\alpha}- 2|e| -2 \chi(\Sigma)+m_+-m_-+M_{\alpha}+ \chi(\Sigma)\\
\ge &2(M_{\alpha}-|e| -1) +h(C) + m_+ +2e_-(C)      -\chi(\Sigma)\\
  \ge&  h(C) + m_+ +2e_-(C)  \ge 1. 
        \end{split}
\end{equation*}
Note that we have used (\ref{eq9}) in the last two lines  of the above inequalities.  
\end{itemize}
 In all  cases,  it  voliates the ECH inequality (\ref{eq20}).
\end{proof}

\begin{lemma}  \label{lem5}
Assume $J$ is a generic cobordism admissible almost complex structure.  Let $\mathcal{C}=\{\mathcal{C}_0, \mathcal{C}_{-1}, ..., \mathcal{C}_{-N}\} \in \overline{\mathcal{M}_{X_+}^J}(\emptyset, \alpha)$ be a broken  holomorphic current with total relative homology class $Z^{\vee}_{\alpha} + k[\Sigma_+]$.
Then $k \ge 0$
\end{lemma}
\begin{proof}
Let $\mathcal{C}=\{\mathcal{C}_0, \mathcal{C}_{-1}, ..., \mathcal{C}_{-N}\}$ be broken holomorphic current with relative homology class $Z^{\vee}_{\alpha} + k[\Sigma_+]$, where $\mathcal{C}_0 \in \mathcal{M}^J_X(\emptyset, \alpha_0)$ and  $\mathcal{C}_{-i} \in \mathcal{M}^J_Y(\alpha_{i-1}, \alpha_i)$ and $\alpha_N= \alpha$.  By Lemma \ref{lem8}, we have 
\begin{equation*}
\begin{split}
E_{\omega_+^{\varepsilon}}(\mathcal{C}) =&3(M_{\alpha} + k|e|) -\mathcal{A}_{\lambda_{\varepsilon}}(\alpha)\\
=&E_{\omega_+^{\varepsilon}}(\mathcal{C}_0) + \sum_{i=1}^{N} \int_{\mathcal{C}_{-i}} d\lambda_{\varepsilon}\\
=& 3(M_{\alpha_0} + k_{\mathcal{C}_0}|e|)   -\mathcal{A}_{\lambda_{\varepsilon}}(\alpha_0) + \sum_{i=1}^{N} \mathcal{A}_{\lambda_{\varepsilon}}(\alpha_{i-1})  -\mathcal{A}_{\lambda_{\varepsilon}}(\alpha_{i}). 
\end{split}
\end{equation*}
Hence, we have 
\begin{equation*}
 k=k_{\mathcal{C}_0 } +\frac{M_{\alpha_0} - M_{\alpha}}{|e|} =k_{\mathcal{C}_0 } +\sum_{i=1}^N\frac{M_{\alpha_{i-1}} - M_{\alpha_{i}}}{|e|} 
\end{equation*}
By Proposition 4.4 of \cite{NW}, $M_{\alpha_{i-1}} - M_{\alpha_{i}} \ge 0$ for all $1\le i\le N$.  Thus, it suffices to show that $k_{\mathcal{C}_0 }  \ge 0$. 

By definition, we write $\mathcal{C}_0=\sum_a d_a C_a$, where $C_a$ are distinct simple holomorphic curves and $d_a $ are positive integers. Let $Z_{\alpha_a}^{\vee} + k_a [\Sigma_+]$ denote the relative homology class of $C_a$. By definition,  $k_{\mathcal{C}_0}=\sum_a d_a k_a$. Thus, Lemma \ref{lem6}  implies $k_{\mathcal{C}_0} \ge 0$.

\end{proof}

\section{Gromov-Taubes invariants of Ruled surfaces }

Recall that  $(X_+, \omega^{\varepsilon}_+)$ is a strong  symplectic cap of $(Y, \lambda_{\varepsilon})$ and
$(X_-, \omega^{\varepsilon}_-)$ is a strong  symplectic filling of
$(Y, \lambda_{\varepsilon})$, we  glue them along their common boundary to obtain a
closed symplectic manifold $(X,\omega^{\varepsilon})$.  The disk fibers of $\mathbb{D}E $ and $\mathbb{D}E^{\vee} $ are glued to be a sphere. Thus, $\pi_X :X \to \Sigma$ is a sphere bundle over $\Sigma$ (ruled surface).   Topologically, we have 
$$X=\mathbb{P}(\underline{\mathbb{C}}\oplus E) \to \Sigma,$$
  where $\underline{\mathbb{C}}$ denote  the trivial line bundle over $\Sigma$.   

The gold of this section is to use the wall-crossing formula in  \cite{OT, OT2} to compute the Gromov-Taubes invariants of $X$.  Let us clarify some facts about the topology of $X$ first.

Let $\Sigma_+ \subset  X$ and  $\Sigma_- \subset X $ denote respectively the zero sections of $\mathbb{D}E^{\vee} $ and $\mathbb{D}E$  in $X$.  By the construction, $\Sigma_+$ and $\Sigma_-$ are disjoint and $[\Sigma_{\pm}]^2=\pm|e| ,$ $[\Sigma_{\pm}] \cdot S^2=1$, and both of them are sections of $\pi_X : X \to \Sigma$. 
 It is well known that first homology class of a ruled surface is the first homology of the base, and the second homology of a ruled surface is generated by the $S^2$-fiber and a section. Thus, we have   the following lemma. 
\begin{lemma} \label{lem9}
We have $H_1(X,\Z) \cong H_1(\Sigma,\Z)$ and $H_2(X, \mathbb{Z}) \cong \mathbb{Z} [S^2] \oplus \mathbb{Z}[\Sigma_+] $. Moreover, one has
$[\Sigma_{\pm}]^2=\pm|e| , [\Sigma_{\pm}] \cdot [S^2]=1,$
and 
$$[\Sigma_+] =[\Sigma_-] +|e|[S^2]. $$
\end{lemma}
\begin{proof}
See Proposition 7.62  of \cite{Wen}. 
\end{proof}

By Lemma \ref{lem9}, we identify $\mathbb{A}(X)$ with $\mathbb{A}(\Sigma):=\Q[u] \otimes \Lambda^*H^1(\Sigma, \Q)$ throughout. The following theorem is a reinterpretation of Corollary 4.7 of \cite{OT2} for Gromov-Taubes invariants. 

\begin{theorem}[Gromov-Taubes invariants  of Ruled surfaces \cite{OT2}] \label{thm0} 
Let $x =u^a \theta^b e_c^r \in \mathbb{A}(\Sigma)$, where $0 \le b\le g$ and $a \ge 0$.  If $x =u^a \theta^b $,  $2a+2b = I(A)$ and $A \cdot [S^2] \ge 0$, then 
$$\Gr(X, \omega^{\varepsilon}, A, x) = \frac{g!}{(g-b)!} (1+ A \cdot [S^2])^{g-b};$$
otherwise, $\Gr(X, \omega^{\varepsilon}, A, x) =0$. 
\end{theorem}
 \begin{proof}
Fix  a spin-c structure $\mathfrak{s}$,  a Riemannian metric $g$, and a $i\mathbb{R}$-value  self-dual 2 form $\mu$. The (perturbed) Seiberg-Witten equations is 
 \begin{equation} \label{eq40}
 \begin{cases}
 D_A\Psi = 0 &\\
 F_A^+ =\frac{1}{4}q(\psi) + \mu,
 \end{cases}
 \end{equation}
 where $\psi$ is a section of spin-c bundle $S_+$,  $A$ is connection of $\det S_+$, $F_A$ is the curvature of $A$, $D_A$ is the Dirac operator associated with $A$, and $q$ is a certain quadric operator on $\Gamma(S_+)$. 
 Roughly speaking, the Seiberg-Witten invariant is defined by counting the number of solutions to the above equations.  We refer    the details on  Seiberg-Witten equations to \cite{Sal, KM}.

As mentioned before, the orientation of the moduli space of solutions to Seiberg-Witten equations is determined by the    a choice of homology orientation in  $\Lambda(X)$. 
  Now   $(X, \omega^{\varepsilon}) $ is symplectic. The symplectic form induces a canonical  one $o(X, \omega) \in \Lambda(X)$  (see Page 460-461 of \cite{Ta4}).   
  
 Fix a spin-c structure $\mathfrak{s}_A$  with  first Chern class 
 \begin{equation*}
 \begin{split}
 c_1(\mathfrak{s}_A)=&2\operatorname{PD}(A)+ c_1(K_X^{-1}) \\
 =&2\operatorname{PD}(A)+ 2\operatorname{PD}([\Sigma_+]) + (2-2g - |e|)\operatorname{PD}[S^2]. 
 \end{split}
 \end{equation*}
 Fix a metric $g$ on $X$. Then we have a  self-dual  harmonic 2-form $\omega_g$       
such that  it  is positively proportional to  $[\omega^{\varepsilon}] $. 
Define the discriminant  
\begin{equation}\label{eq36}
  \Delta(g, \mu) =\int_X  i\mu\wedge \omega_g -2\pi [\omega_g] \cdot c_1 (\mathfrak{s}_A).
  \end{equation}
   Let 
  \begin{equation*}
  \SW_X^{ \operatorname{sign} \Delta(g, \mu) }(\mathfrak{s}, -):  \mathbb{A}(X) \to \mathbb{Q}
  \end{equation*}
  denote the Seiberg-Witten invariant defined by  counting solutions to (\ref{eq40}).   $\SW_X^{+} $ is the one in defined by the data in  ``Taubes' chamber''.  Specifically, according to Taubes's    ``Gr=SW"  theorem \cite{Ta1, Ta2, Ta3, Ta4} (Theorem 2 of \cite{Ta4}),        we have 
$$\Gr(X, \omega^{\varepsilon},  A, x)  = \SW_X^{+}(\mathfrak{s}_A, x).  $$

 Thus, it suffices to compute $\SW_X^{+}$. 
To this end, we apply the general wall crossing formula of  C. Okonek and A. Teleman \cite{OT, OT2}. They  define  a class  $\Theta_{\mathfrak{s}_A}  \in \Lambda^2H_1(X, \Z)$ by 
\begin{equation*} 
\begin{split}
\Theta_{\mathfrak{s}_A}(a\cup b) : =\frac{1}{2} \langle c_1(\mathfrak{s}_A) \cup a\cup b, X\rangle. 
\end{split}
\end{equation*} 
By identifications $H_1(X, \Z)  \cong H^1(X, \Z) \cong H^1(\Sigma, \Z)$, we have 
\begin{equation*} 
\begin{split}
\Theta_{\mathfrak{s}_A}  (a\cup b)  &=\frac{1}{2} \langle   c_1(\mathfrak{s}_A ),  [S^2] \rangle\langle a\cup b, \Sigma \rangle \\
&= (1+ A \cdot[S^2]) \langle a\cup b, \Sigma \rangle.
\end{split}
\end{equation*} 
Then $\Theta_{\mathfrak{s}_A} =( 1+ A \cdot[S^2]) \theta. $

If  $I(A)  \ne \degop( u^a\theta^be^r_c) $, then $\SW_X^+(\mathfrak{s}_A, u^a\theta^be^r_c  )  =0$ by its definition.

Consider the case that $A \cdot [S]^2 \ge 0$, i.e.,  $\langle c_1(\mathfrak{s}_A),  [S^2]  \rangle \ge 2$,  and   $I(A)  = \degop( u^a\theta^be^r_c) $. By Corollary 4.7 of \cite{OT2},  we have  
\begin{equation*} 
\begin{split}
\SW_X^+(\mathfrak{s}_A, u^a\theta^be^r_c  ) &= \int_{T_X}   \exp(\Theta_{\mathfrak{s}_A})   \wedge \theta^be^r_c  =  \int_{T_X}   \exp\left((1+A \cdot [S^2]) \theta\right)  \wedge \theta^be^r_c \\
\end{split}
\end{equation*}
where $T_X :=H^1(X, \R)/H^1(X, \Z) \cong H^1(\Sigma, \R)/H^1(\Sigma, \Z) =\mathbb{T}^{2g}$ is the Jacobian torus with symplectic form  $\theta$.
Note that   if $e^r_c    \ne 1$, then $ \int_{T_X}   \theta^m e^r_c \ne 0$  only  if  $$2m =2g -\degop e^r_c =2g-r. $$
If $r =1$, then the right hand side is odd, and hence no such $m$ exists.  If $r \ge 2$, then $m=g-\frac{1}{2}r\ge g+1-r$.  By definition of  $e^r_c$, $\theta^m e^r_c =0$. Thus,   
$$ \int_{T_X}   \theta^m e^r_c = 0$$
for any $m \in \Z_{\ge 0}$.  Thus,  $\SW_X^+(\mathfrak{s}_A, u^a\theta^be^r_c  ) =0 $ if $e^r_c    \ne 1$. 

In the case that $e^r_c= 1$ and $I(A)  = \degop( u^a\theta^b) =2a+2b $, we have 
 \begin{equation*} 
\begin{split}
\SW_X^+(\mathfrak{s}_A, u^a\theta^b ) =&  \int_{T_X}   \exp((1+A \cdot [S^2]) \theta)  \wedge \theta^b \\   
=& \int_{\mathbb{T}^{2g}} \frac{(1+A\cdot [S^2])^{g-b}}{(g-b)!}  \theta^g=\frac{g!}{(g-b)!} (1+ A \cdot [S]^2)^{g-b}. 
\end{split} 
\end{equation*}

Finally, according to  Corollary 4.7 of \cite{OT2}, if $ A \cdot[S]^2 \le -1$ (i.e.,  $\langle c_1(\mathfrak{s}_A),  [S^2]  \rangle\le 0$),  then  $$\SW_X^+(\mathfrak{s}_A, u^a\theta^be^r_c  )  =0.$$
  
 \end{proof}
 
\begin{remark}\label{remark3}
Even though we have $``Gr=SW''$ theorem, the Gromov-Taubes invariants may, in general, depend on the choice of symplectic form, since different symplectic forms  may  determine different homology orientations and different ``Taubes chambers.''

By our construction,  for different $\varepsilon_1 \ne \varepsilon_2$,  we have  
$$\omega^{\varepsilon_1} =\omega^{\varepsilon_2} + d\mathfrak{a}_{\varepsilon_1, \varepsilon_2},$$
where $\mathfrak{a}_{\varepsilon_1, \varepsilon_2}$ is a certain  small 1-form supporting in a small neck neighborhood of $Y$.    In particular,  $[\omega^{\varepsilon_1}] =[\omega^{\varepsilon_2}]$.
For sufficiently small $\varepsilon_1, \varepsilon_2$,  $\omega^{\varepsilon_2} + t d\mathfrak{a}_{\varepsilon_1, \varepsilon_2}$ is symplectic for $t \in [0,1]$. Theorefore, $\omega^{\varepsilon_1} $ is deformation equivalent to $\omega^{\varepsilon_2}$. Then 
they determine the same homology orientation and the same discriminant  (\ref{eq36}).  Thus, we are using the same  homology orientation implicitly for all sufficiently small $\varepsilon>0$ in Theorem \ref{thm0}. As a result,   the Gromov-Taubes invariants  defined using by   $ \omega^{\varepsilon_1} $ and $\omega^{\varepsilon_2}$ agree, and both are equal to   $SW_X^+$.  This is reflected in the conclusion of  Theorem \ref{thm0}. 

  To simplify the notation, we drop the $\omega^{\varepsilon}$ in $\Gr(X, \omega^{\varepsilon},  A, x) $. 
\end{remark}

 \begin{lemma} \label{lem1}
Let $\mathfrak{R}_n(u, \theta)$ be the polynomial (\ref{eq1}), where $n >2g-2$.   If $P \in \mathbb{A}(\Sigma)$ is a homogeneous element with $\degop P =2g-2$, then $$\Gr (X, nS^2, P\mathfrak{R}_n(u, \theta))=0. $$
\end{lemma}
\begin{proof}
For $P \in \mathbb{A}(\Sigma) =\Q[u] \otimes \Lambda^*H^1(\Sigma, \Q)$, we can write $P$ as a linear combination of $u^a\theta^be_r^c$ because $\theta^b e_r^c$ is a basis of $\Lambda^*H^1(\Sigma, \Q). $

By Theorem \ref{thm0}, it suffices to consider the case that  $P=u^a\theta^b$ for some $a, b \ge 0$ such that $a+b =g-1$.  Since 
$$\mathfrak{R}_n(u, \theta)  =\sum_{k=0}^{g} (-1)^{g-k}  \frac{\theta^k}{k!}u^{n-g-k+1}$$
has degree $2(n-g+1)$,      $\degop P\mathfrak{R}_n(u, \theta) =2n =I(n[S^2]) $.  

By  Theorem \ref{thm0}, we have 
\begin{equation}
\begin{split}
&\Gr(X,  nS^2, P\mathfrak{R}_n(u, \theta)) =\Gr(X, nS^2, \sum^{g-b}_{k=0} (-1)^{g-k} \frac{1}{k!}\theta^{k+b} u^{n-g-k+1+a})\\
=& (-1)^gg!\sum_{k=0}^{g-b} (-1)^{-k} \frac{1}{k!} \frac{1}{(g-b-k)!} (1)^{g-b-k}=0.
\end{split}
\end{equation}

\end{proof}

\section{Computing the $U$ maps}
In this section, we construct a basis of the filtered ECH using the ECH cobordism maps (Theorem \ref{thm2}), and then compute the $U$ map under this basis by using Theorem \ref{thm0}.

\subsection{A basis from ECH cobordism maps}
Because $H_1(X_{\pm}, \Z) \cong H_1(\Sigma, \Z)$ and $H_1(Y, \Q) \cong H_1(\Sigma, \Q)$, we identify $ \mathbb{A}(X_{\pm}), \mathbb{A}(Y)$ with $ \mathbb{A}(\Sigma)$.

\begin{lemma} \label{lem4}
Let $A_{\pm} \in H_2(X_{\pm}, Y, \Z)$ such that $\partial A_{\pm} =\Gamma$. Let $A \in H_2(X, \Z)$ such that $A \vert_{X_{\pm}} =A_{\pm}$. Then 
$$A = \left(A_- \cdot [\Sigma_-] \right) [S^2] + \frac{A _+\cdot [\Sigma_+] -A_- \cdot [\Sigma_-] }{|e|}[\Sigma_+].$$

The ECH index of $A$ is 
\begin{equation*}
\begin{split}
I(A) = A_+\cdot \Sigma_+  + A_-  \cdot   \Sigma_- + \frac{A_+\cdot \Sigma_+ -A_- \cdot \Sigma_-}{|e|} \chi(\Sigma) +  \frac{(A_+\cdot \Sigma_+)^2 -(A_-\cdot \Sigma_-)^2}{|e|}
\end{split}
\end{equation*}
\end{lemma}

\begin{proof}
Write $A =a [S^2] + b[\Sigma_+]$.   By $A \vert_{X_{\pm}} =A_{\pm}$, we have 
\begin{equation*}
\begin{split}
&A \cdot [\Sigma_+] = A_+ \cdot [\Sigma_+] = a+ b|e|\\
&A \cdot [\Sigma_-] = A_- \cdot [\Sigma_-] = a.\\
\end{split}
\end{equation*}
Hence,  $b =(A \cdot [\Sigma_+] -A_- \cdot [\Sigma_-] )/|e|.  $

By the fact that  $c_1(TX) 
 =  2\operatorname{PD}([\Sigma_+]) + (2-2g- |e|)\operatorname{PD}[S^2] $ and Lemma \ref{lem9}, then
\begin{equation*}
\begin{split}
I(A) =&\langle c_1(TX), A \rangle + A \cdot A\\
=&2a  + 2b|e|+ (2-2g-|e|)b  +2ab + b^2|e|\\
=&a + b\chi(\Sigma) +(a+ b|e|) + 2ab + b^2|e|.  
\end{split}
\end{equation*} 
Substituting the previous calculations regarding $a, b$ into the above equation  yields the desired result on ECH index.
\end{proof}

For $d \ge 0$,  take    relative classes  $A_{\pm}(d) \in H_2(X_{\pm}, Y, \Z)$   satisfying
\begin{equation}
        A_{\pm}(d)\cdot [\Sigma_{\pm}]=d|e| +\Gamma.
\end{equation}
Given $d_1, d_2$, by Lemma \ref{lem4}, we have a unique class $A_{d_1, d_2} \in H_2(X, \Z)$ such that $A \vert_{X_+} =A_{+}(d_1)$ and $A \vert_{X_-} =A_-(d_2).$

For simplicity, we write
\begin{equation}
        \tilde{\Phi}^{\varepsilon}_d:=\ECH(X_+, \omega^{\varepsilon}_+, A_+(d)) \mbox{ and }
        \Psi^{\varepsilon}_d:=\ECH(X_-, \omega^{\varepsilon}_-, A_-(d)).
 \end{equation} 
We compute the composition of     $\Psi^{\varepsilon}_{d_2}$ and  $\tilde{\Phi}^{\varepsilon}_{d_1}$ in the sequential lemmas.

 \begin{lemma} \label{lem17}
Suppose $\Gamma=0$. For $x=u^a \theta^b e_r^c  \in \mathbb{A}(\Sigma)$, then  $\tilde{\Phi}^{\varepsilon}_0(x) =0$ unless $x=1$. Moreover,  $\tilde{\Phi}^{\varepsilon}_0(1)= [\emptyset] $. 
\end{lemma}
\begin{proof}
By Lemma \ref{lem8},   the filtration and direct limit properties in Theorem \ref{thm2}, we have $\tilde{\Phi}^{\varepsilon}_0(x) \in \operatorname{Im}\{ i^{\delta}: \ECH^{\delta}(Y, \lambda_{\varepsilon}, 0) \to \ECH(Y, \lambda_{\varepsilon}, 0)\}$, where 
$0< \delta \le 1$.  Note that $\ECH^{\delta}(Y, \lambda_{\varepsilon}, 0) =\Q [\emptyset].$ If $\tilde{\Phi}^{\varepsilon}_0(x) \ne 0 $,  then 
$$0=\operatorname{gr}([\emptyset])=\operatorname{gr}(\tilde{\Phi}^{\varepsilon}_0(x) ) =I(A_+(0)) -\degop x = 0 -\degop x. $$
Thus,  $\tilde{\Phi}^{\varepsilon}_0(x) \ne 0 $ only when $x=1$.

Note that $A_+(0) =0$ when $\Gamma=0$. Because  $\ECH(X_+, \omega^{\varepsilon}_+, 0)$ preserves the contact invariants, 
 we have   $\tilde{\Phi}^{\varepsilon}_0(1)=   [\emptyset] $.  The sign here is positive due to the choice of homology orientation (\ref{eq39}). 
\end{proof}

\begin{lemma} \label{lem14} 
Set $\mathcal{B}_0=\{1\}$.  Suppose that $\Gamma=0$.  For $d \ge 0$ and  $x  \in \mathcal{B}_{d|e|}$, we have  
\begin{equation*}
\Psi^{\varepsilon}_0(1 \otimes  \tilde{\Phi}^{\varepsilon}_d(x))=
\begin{cases}
 \frac{g!}{(g-b)!}2^{g-b}  & d=1 \mbox{ and $x = u^{|e| -g +1 -b} \theta^b$},\\
1 & d=0 \mbox{ and $x =1$},\\
0 & else. 
\end{cases}
\end{equation*}
\end{lemma}
\begin{proof}
  By Lemma \ref{lem17}, $\tilde{\Phi}^{\varepsilon}_0(1) =[\emptyset]$. Note that $A_-(0) =0$ when $\Gamma=0$. Because  $\ECH(X_-, \omega^{\varepsilon}_-, 0)$ preserves the contact invariants, 
 we have  $\Psi^{\varepsilon}_0(1 \otimes  \tilde{\Phi}^{\varepsilon}_0(1)) =1. $   The  positive sign here again  due to the choice of homology orientation (\ref{eq39}).

Consider $d \ge 1$. Let $x =u^a \theta^b e_r^c \in \mathcal{B}_{d|e|}$.  By composition law of ECH cobordism maps and Theorem \ref{thm0}, we have 
$$\Psi^{\varepsilon} _0(1 \otimes \tilde{\Phi}^{\varepsilon}_d(x))=\Gr(X, A_{d, 0}, x).  $$  
By definition, $\Gr(X, A_{d, 0}, x) \ne 0$ only if $I(A_{d, 0}) =\degop x. $

Note that $\deg x \le 2d|e|$ for  $x \in \mathcal{B}_{d|e|}$. If $d \ge 2$, then  by Lemma \ref{lem4}, we obtain 
\begin{equation*}
\begin{split}
I(A_{d, 0}) =d^2|e| + d(|e| + \chi(\Sigma)) \ge 2d|e| +2(|e| + \chi(\Sigma)) > 2 d|e| \ge \degop x. 
\end{split}
\end{equation*}
Here we have used the assumption that $e<\chi(\Sigma)$. 
Thus, $\Psi^{\varepsilon} _0(1 \otimes  \tilde{\Phi}^{\varepsilon}_d(x))=0$ when $d\ge 2$.

Now consider $d=1$.  By Theorem \ref{thm0},  $\Gr(X, A_{1, 0}, x) \ne 0$ only if  $e_r^c=1$ and $I(A_{1, 0}) =\deg x$. Thus, it suffices to consider $x =u^{|e| -g +1 -b} \theta^b$, where $0 \le b \le g$.  By Theorem \ref{thm0}, we have 
$$\Psi^{\varepsilon} _0(1 \otimes  \tilde{\Phi}^{\varepsilon}_1(u^{|e| -g +1 -b} \theta^b))= \frac{g!}{(g-b)!}2^{g-b}. $$  
\end{proof}

\begin{definition} \label{definition2}
Set $\mathcal{B}_0=\{1\}$. For $d \ge 0$ and  $x \in \mathcal{B}_{d|e|+\Gamma}$, define 
 \begin{equation*}
  \Phi_d^{\varepsilon}(x):=
 \begin{cases}
\tilde{\Phi}^{\varepsilon}_1(x) - \Psi_0^{\varepsilon}(1 \otimes \tilde{\Phi}^{\varepsilon}_1(x)) [\emptyset] & d=1 \mbox{ and }\Gamma=0\\
 \tilde{\Phi}^{\varepsilon}_d(x) &  \mbox{ else}. 
 \end{cases}
 \end{equation*}
\end{definition}
Later, we will show that elements in Definition \ref{definition2} forms a basis of the filtered ECH.

The next lemma is an analogy of Proposition 3.4 in \cite{MW}. We show that   the pairing of elements  in Definition \ref{definition2} is diagonal and the pairing matrix is  simply the direct sum of the intersection pairing matrix of $H^*(\Sym^n \Sigma, \Q)$.  
\begin{lemma}\label{lem3}
Assume $e< \chi(\Sigma).$  Suppose that $\Gamma=0$ or $\Gamma>2g-2$.  Let
        $x_i
        \in  \mathcal{B}_{d_i |e| +\Gamma}, i=1,2, $ where $d_i \ge 0$.   One has
$$
        \Psi^{\varepsilon}_{d_2}(x_2 \otimes  \Phi^{\varepsilon}_{d_1}(x_1))
        =0
$$
unless
$$
        d_1=d_2.
$$

 If $d_1=d_2=m$, then 
 \begin{equation*}
  \Psi^{\varepsilon}_{m}(x_2 \otimes  \Phi^{\varepsilon}_{m}(x_1)) =q^{m|e|+\Gamma} (\Pi(x_2), \Pi(x_1)),
 \end{equation*}
 where $q^{m|e|+\Gamma}$ is the intersection  pairing  on $H^*(\Sym^{m|e|+\Gamma}\Sigma, \Q).$   
  \end{lemma}

\begin{proof}

By the composition law of ECH cobordism maps (Theorem \ref{thm2}), we have 
$$
        \Psi^{\varepsilon}_{d_2}(x_2 \otimes \tilde{\Phi}^{\varepsilon}_{d_1}(x_1))
        =\Gr(X, A_{d_1, d_2}, x_2x_1), 
$$
 where $A_{d_1, d_2} \in H_2(X, \Z)$ is the homology class such that $A_{d_1, d_2} \vert_{X_+} =A_+(d_1)$ and $A_{d_1, d_2} \vert_{X_-} =A_-(d_2)$. 
By Lemma \ref{lem4}, we have 
$$A_{d_1, d_2}= (d_2|e|+\Gamma) [S^2] + (d_1-d_2)[\Sigma_+]$$  
\begin{equation*}
\begin{split}
\mbox{ and } I(A_{d_1, d_2})= (d_1-d_2) \chi(\Sigma) + (d_1+d_2) |e| + (d_1^2 -d_2^2)|e| + 2\Gamma(d_1-d_2+1). 
\end{split}
\end{equation*}
By Theorem \ref{thm0}, the Gromov-Taubes invariant $\Gr(X, A, x_2x_1)$  is nonzero only if
\begin{equation*}
        d_1\ge d_2,
        \qquad
        x_2x_1\ne 0\in  \mathbb{A}(\Sigma),
        \qquad
        \degop x_1+\degop x_2=I(A).
\end{equation*}
We assume these conditions hold; otherwise, there is nothing to compute.

If $\Gamma=0$ and $d_2 =0$, this case has already been  proved in the Lemma \ref{lem14}.    
In the rest of the cases, we   always have  $d_i|e|+\Gamma > 2g-2$.    Then we  have a basis $\mathcal{B}_{d_i|e|+\Gamma}$  defined  in (\ref{eq17}).  Take $x_i =u^{a_i}\theta^{b_i}e_{r_i}^{c_i} \in  \mathcal{B}_{d_i |e| +\Gamma}$.  
By definition,  $a_i\le d_i|e| +\Gamma-g$.  If $x_2x_1\ne 0 \in \mathbb{A}(\Sigma)$, then
$$
        \degop x_1+\degop x_2\le 2(d_1+d_2)|e| +4\Gamma-2g,
$$
because the total degree of $\theta^{b_1}e_{r_1}^{c_1}\theta^{b_2}e_{r_2}^{c_2}$ is at most $2g$; otherwise $x_2x_1=0$.

Assume $d_1-d_2 \ge 1$. In particular, $d_1 \ge 1$. Then $$(d_1+d_2)|e|+\chi(\Sigma) \ge |e|+ \chi(\Sigma) >0.$$ Thus, 
\[
\begin{aligned}
        I(A)
        &=(d_1+d_2)|e|
        +(d_1-d_2)\bigl((d_1+d_2)|e|+\chi(\Sigma)\bigr)  +2\Gamma(d_1-d_2+1)\\
        &\ge (d_1+d_2)|e|
        +(d_1+d_2) |e| + 2- 2g +4\Gamma\\
&> 2(d_1+d_2)|e| +4\Gamma -2g \ge \degop x_1+\degop x_2. 
\end{aligned}
\]
Thus the degree condition cannot hold.  Hence  $ \Psi^{\varepsilon}_{d_2}(x_2 \otimes \tilde{\Phi}^{\varepsilon}_{d_1}(x_1))
        $ vanishes  when 
$d_1 \ne d_2$.

Now assume $d_1= d_2 =m$.  By Lemma \ref{lem4}, $A =(m|e|+ \Gamma) [S^2]$  and  $I(A) =2(m|e|+ \Gamma)$.   Set $N=m|e|+ \Gamma>2g-2$.

If  $\degop x_1 + \degop x_2 \ne 2N$, then $\Gr(X, A, x_2x_1) =0$ by the degree reason. Also,  $q^{N}(\Pi(x_1), \Pi(x_2)) =0$ by definition.  Thus, they agree. 

Assume  $\degop x_1 + \degop x_2 = 2N$.  Since $N>2g-2$, recall that we have (\ref{eq10}), i.e., 
$$x_2x_1 = q^N(\Pi(x_2), \Pi(x_1)) \frac{u^{N-g}\theta^g}{g!} +P \mathfrak{R}_N(u, \theta),$$
where $P \in \mathbb{A}(\Sigma)$ is a homogenous element  with degree $2g-2$.  By Lemma \ref{lem1} and Theorem \ref{thm0}, 
$$\Gr(X, A, x_2x_1) = q^N(\Pi(x_2), \Pi(x_1)) \Gr(X, A, \frac{u^{N-g}\theta^g}{g!})=q^N(\Pi(x_2), \Pi(x_1)) .$$
\end{proof}

\begin{lemma} \label{lem24}
Assume $e<\chi(\Sigma)$.  Suppose  $\Gamma =0$ or $\Gamma>2g-2$. Let $\{x_{m, i_m}\}=\mathcal B_{m|e|+\Gamma}$  be the basis  described  in  (\ref{eq17}).  
For $d \ge 0$, the set of elements 
$$
        \bigl\{\Phi^{\varepsilon}_m(x_{m, i_m})\bigr\}_{  0 \le m \le d}
$$
are linear independent.  In particular, $\Phi^{\varepsilon}_m(x_{m, i_m})$ are nonzero. 
\end{lemma}
\begin{proof}
 Suppose
$$
        \sum_{0 \le m \le d}  \sum_{i_m} a_{m, i_m}\Phi^{\varepsilon}_m(x_{m, i_m})=0.
$$
Pair with the   filling-side ECH cobordism  maps $\Psi^{\varepsilon}_n(x_{n, j_n})$. By Lemma \ref{lem3}, we have 
\begin{equation*}
\begin{split}
       0&= \sum_{0  \le m \le d}  \sum_{i_m} a_{m, i_m}\Psi^{\varepsilon}_n(x_{n, j_n} \otimes  \Phi^{\varepsilon}_m (x_{m, i_m}))\\
&=  \sum_{i_n} a_{n, i_n}\Psi^{\varepsilon}_n(x_{n, j_n} \otimes \Phi^{\varepsilon}_n (x_{n, i_n}))\\
&=  \sum_{i_n} a_{n, i_n}q^{n|e| +\Gamma}(\Pi(x_{n, j_n}), \Pi(x_{n, i_n})),
\end{split}
\end{equation*}
where $q^{n|e| +\Gamma}$ is the intersection pairing of $H^*(\Sym^{n|e| +\Gamma}\Sigma, \Q)$.  Thus, it is nondegenerate. As a result, $a_{n, i_n}=0$ for all $i_n$ and $n$.    In other words, $\{\Phi^{\varepsilon}_m(x_{m, i_m})\}_{0 \le m \le d }$ are linear independent.

\end{proof}

Recall that $L_{\varepsilon}$ is the number such that the Reeb orbits with action $\mathcal{A}_{\lambda_{\varepsilon}}<L_{\varepsilon}$ consists of  covers of fibers over critical points of $H$.  We choose $L_{\varepsilon}$ to be irrational throughout.  
\begin{lemma}\label{lem2}
Assume $e<\chi(\Sigma)$.  Suppose  $\Gamma =0$ or $\Gamma>2g-2$. Let $\{x_{m, i_m}\}=\mathcal B_{m|e|+\Gamma}$  be the basis  described  in  (\ref{eq17}). .

 Fix  $d_{\varepsilon}: = \min\{\lfloor \frac{(L_{\varepsilon}  -3\Gamma) }{3|e|}\rfloor,  \lfloor  \frac{1}{2}(1/\varepsilon -2\Gamma/|e|) -1/2\rfloor \}$.    For $0\le d\le d_{\varepsilon}$,   then the inclusion morphisms 
 $$i^{L_{\varepsilon, d}} : \ECH^{L_{\varepsilon, d}}(Y, \lambda_{\varepsilon},  \Gamma ) \to \ECH(Y, \lambda_{\varepsilon},  \Gamma ) $$
 are injective,  where  $L_{\varepsilon, d}:=\max\{2(1+\varepsilon) (d|e| +\Gamma), 1\}. $
  Moreover, there exists  elements $\{\sigma_{m, i_m}\}_{0\le m\le d}  \subset \ECH^{L_{\varepsilon, d}}(Y, \lambda_{\varepsilon},  \Gamma )$ such that  
$$
         \Phi^{\varepsilon}_m(x_{m, i_m}) = i^{L_{\varepsilon, d}}(\sigma_{m, i_m})
$$
and they 
forms a basis of  $\ECH^{L_{\varepsilon, d}}(Y, \lambda_{\varepsilon},  \Gamma )$.  \end{lemma}

\begin{proof}

If $\Gamma=0$ and $d=0$, then  $\ECH^{L_{\varepsilon, 0}}(Y, \lambda_{\varepsilon},  0)  = \Q [\emptyset]$. Then $[\emptyset]$ is a basis.   By Lemma \ref{lem14}, $\Phi_0^{\varepsilon}(1) =[\emptyset] =i^{L_{\varepsilon, 0}}([\emptyset]).$

Now consider the general case that $L_{\varepsilon, d}>1$.  Let $0 \le d \le d_{\varepsilon}$.     By the filtration and direct limit   properties in Theorem \ref{thm2}, we have 
 \begin{equation*}
 \begin{split}
 &\tilde{\Phi}^{\varepsilon}_m(x_{m, i_m})  = i^{\delta + \rho(A_+(m))} \circ \ECH^{\delta}(X_+, \omega_+^{\varepsilon}, A_+(m))(x_{m, i_m}),\\
 \mbox{ and } & \ECH^{\delta}(X_+, \omega_+^{\varepsilon}, A_+(m))(x_{m, i_m}) \in \ECH^{\rho(A_+(m)) + \delta}(Y, \lambda_{\varepsilon}, \Gamma), 
 \end{split}
 \end{equation*}
 where $\delta>0$ is an arbitrary  small constant. 
  By Lemma \ref{lem8},   we have  
  $$ \rho(A_+(m)) = 3(m|e|+\Gamma). $$ 
  Then  $\rho(A_+(m)) + \delta \le \rho(A_+(d)) + \delta \le \rho(A_+(d_{\varepsilon})) + \delta<L_{\varepsilon}$. Such  a $\delta$ exists because $L_{\varepsilon}$ is irrational.

Fix a generic cobordism almost complex structure $J$ on the symplectic completion of  $(X_+, \omega_+^{\varepsilon})$.  By Theorem \ref{thm2}, we have a chain map  
$$\phi^{\varepsilon}_m : \mathbb{A}(\Sigma) \to \ECC^{\delta + \rho(A_+(m))}(Y, \lambda_{\varepsilon}, \Gamma) \subset  \ECC^{L_{\varepsilon}}(Y, \lambda_{\varepsilon}, \Gamma)$$
   inducing  $\ECH^{\delta}(X_+, \omega_+^{\varepsilon}, A_+(m))$ and satisfying the holomorphic curve axioms. Write  $$\phi^{\varepsilon}_m(x_{m, i_m}) =\sum_{\alpha}  \langle   \phi^{\varepsilon}_m(x_{m, i_m}) , \alpha \rangle \alpha \in \ECC^{L_{\varepsilon}}(Y, \lambda_{\varepsilon},  \Gamma).  $$
If $ \langle   \phi^{\varepsilon}_m(x_{m, i_m}) , \alpha \rangle   \ne 0$,  by the holomorphic curve axioms, we have a broken holomorphic current $\mathcal{C} \in \overline{\mathcal{M}^J_{X_+}}(\emptyset, \alpha)$ with relative homology class $Z^{\vee}_{\alpha, A_+(m)}  =Z^{\vee}_{\alpha} + (m-d(\alpha))[\Sigma_+]$. By Lemma \ref{lem5}, $d(\alpha) \le m$.  Then
\begin{equation*}
\begin{split}
\mathcal{A}_{\lambda_{\varepsilon}}(\alpha) = &2 (1+O(\varepsilon)) M_{\alpha} = 2(1+O(\varepsilon))(d(\alpha) |e| +\Gamma) \\
\le & 2(1+O(\varepsilon))(m |e| +\Gamma)\\
<& L_{\varepsilon, m}\le L_{\varepsilon, d}
\end{split}
\end{equation*}
where $L_{\varepsilon, d}=2(1+\varepsilon) (d|e|+ \Gamma). $ Note that here the size of $O(1)$ is bounded  by $|H|_{C^0}$. Since we choose $|H|_{C^0} \le 1/10$,     we have the above strict  inequality.   

By the above discussion,   we have  a well-defined  class  
$$ [\phi^{\varepsilon}_m(x_{m, i_m})] \in  \ECH^{L_{\varepsilon, d}}(Y, \lambda_{\varepsilon},   \Gamma ) \mbox{ and } \tilde{\Phi}^{\varepsilon}_m(x_{m, i_m}) = i^{L_{\varepsilon, d}}  (  [\phi^{\varepsilon}_m(x_{m, i_m})]). $$
Define 
\begin{equation*}
\sigma_{m, i_m}:=
\begin{cases}
[\phi^{\varepsilon}_1(x_{1, i_1})] - \Psi_0^{\varepsilon}(1 \otimes \tilde{\Phi}^{\varepsilon}_1(x_{1, i_1})) [\emptyset] & m=1 \mbox{ and }\Gamma=0\\
[\phi^{\varepsilon}_m(x_{m, i_m})] & else. 
\end{cases}
\end{equation*}
Then ${\Phi}^{\varepsilon}_m(x_{m, i_m}) = i^{L_{\varepsilon, d}}  (\sigma_{m, i_m}).$

  Suppose 
$$\sum_{m=0}^d \sum_{i_m} a_{m, i_m}\sigma_{m, i_m} =0.$$
Applying $i^{L_{\varepsilon, d}}$ to the above equation, we obtain
$$\sum_{m=0}^d \sum_{i_m} a_{m, i_m}\Phi^{\varepsilon}_m(x_{m, i_m})=0.$$
By Lemma \ref{lem24},  $a_{m, i_m} =0$. Therefore,  $\{\sigma_{m, i_m}\}_{0\le m\le d} $ are linear independent.  As a result,  
\begin{equation*}
\begin{split}
\dim \ECH^{L_{\varepsilon, d} }(Y, \lambda_{\varepsilon},   \Gamma)  \ge&  \dim \operatorname{Span}_{\Q} \{\sigma_m(x_{m, i_m})\}_{0\le m\le d} \\
=& \delta(\Gamma) + \sum_{m=1}^d\dim H^*(\Sym^{m|e| + \Gamma}\Sigma, \Q) \\
= &\delta(\Gamma) + \sum_{m=1}^d2^{2g}(m|e|+\Gamma -g+ 1),
\end{split}
\end{equation*}
where $\delta(\Gamma) :=2^{2g}(\Gamma-g+1)$ if $\Gamma>2g-2$, and $\delta(\Gamma) :=1$ if $\Gamma=0$.

On the other hand,  if $\alpha $  is an  ECH generator such that $d(\alpha)  \ge d+1$ and $[\alpha]=\Gamma$, then 
\begin{equation*}
\begin{split}
\mathcal{A}_{\lambda_{\varepsilon}}(\alpha) >& 2(1-\varepsilon) ( (d+1)|e|+\Gamma)\\
=&2(1+\varepsilon)( d|e| +\Gamma)  +2|e| - 2\varepsilon((2d+1)|e|+2\Gamma)\\
\ge & 2(1+\varepsilon)( d|e| +\Gamma)   =L_{\varepsilon, d}.
\end{split}
\end{equation*}
In other words,  $\ECC^{L_{\varepsilon, d}}(Y, \lambda_{\varepsilon},  \Gamma)$ is generated by the ECH generators with  $d(\alpha) \le d$.   It is directly to show that there 
are $2^{2g}(m|e|+\Gamma-g+1)$ 
 ECH generators   with   $d(\alpha)=m$ and $[\alpha]=\Gamma$, provided that $m|e|+\Gamma >2g-2$. Thus, the dimension of the filtered ECH is 
$$\dim \ECH^{L_{\varepsilon, d} }(Y, \lambda_{\varepsilon},   \Gamma)  \le \dim \ECC^{L_{\varepsilon, d}}(Y, \lambda_{\varepsilon},  \Gamma)=\delta(\Gamma) + \sum_{m=1}^d2^{2g}(m|e| +\Gamma  -g+ 1),$$

The dimension count implies that   $\{\sigma_{m, i_m}\}_{0\le m\le d} $   is a basis of $  \ECH^{L_{\varepsilon, d}}(Y, \lambda_{\varepsilon},   \Gamma)$.  The injectivity of $i^{L_{\varepsilon, d}}$ follows from Lemma \ref{lem24} and   that $\{\sigma_{m, i_m}\}_{0\le m\le d} $  is a basis.

  \end{proof}

\begin{remark}
Up to now, we didn't use the core result  $\partial =0$ in \cite{Fa, NW}.  Therefore, under the assumptions $|e| >2g-2$, $\Gamma=0$ or $\Gamma>2g-2$,  Lemma \ref{lem2} actually  gives a way to compute the $\ECH^{L_{\varepsilon, d} }(Y, \lambda_{\varepsilon},   \Gamma) $ which is different from Farris \cite{Fa} and  Nelson-Weiler's \cite{NW} approaches.  The dimension count in  Lemma \ref{lem2},    $  \ECC^{L_{\varepsilon, d}}(Y, \lambda_{\varepsilon},   \Gamma) =  \ECH^{L_{\varepsilon, d}}(Y, \lambda_{\varepsilon},   \Gamma)$,  indeed implies that $\partial =0$. 
\end{remark}

 \begin{remark}
 Due to  Lemma \ref{lem2}, in the rest of the paper, we implicitly identify  $\ECH^{L_{\varepsilon, d} }(Y, \lambda_{\varepsilon},   \Gamma)$ as a subspace of  $\ECH(Y, \lambda_{\varepsilon},   \Gamma)$ and regard  $\{\Phi^{\varepsilon}_m(x_{m, i_m})\}_{0\le m\le d}$ as a basis  of $\ECH^{L_{\varepsilon, d} }(Y, \lambda_{\varepsilon},   \Gamma).$
 \end{remark}

  \begin{definition} \label{definition1}
Fix  $0\le d\le d_{\varepsilon}$. 
Define a bilinear map  $$\langle \cdot ,  \cdot \rangle_{ech}:  \ECH^{L_{\varepsilon, d}}(Y, \lambda_{\varepsilon}, \Gamma) \times     \ECH^{L_{\varepsilon, d}}(Y, \lambda_{\varepsilon}, \Gamma)  \to \Q $$ by 
 \begin{equation*}
   \langle  \Phi_m^{\varepsilon}(x_{m, i_m}) ,   \Phi_n^{\varepsilon}(x_{n, i_n}) \rangle_{ech}: =  \Psi_m^{\varepsilon} (x_{m, i_m} \otimes  \Phi_n^{\varepsilon}(x_{n, i_n})).  
 \end{equation*}
 \end{definition}
By Lemma \ref{lem14} and Lemma \ref{lem3},  we have 
 $$\langle \cdot,  \cdot \rangle_{ech} =\bigoplus_{m=0}^d  q^{m|e|+\Gamma}. $$
Here if $m= \Gamma=0$, $q^0$ is understood as a bilinear map on $\Q$ such that  $q^0(1, 1) =1$.

\subsection{Computing the $U$ maps}
We assume $d \le d_{\varepsilon}$ throughout; unless otherwise stated. 
With the preparations in the last section,  now we can  compute the $U$ maps
\begin{equation*}
U:  \ECH_*^{L_{\varepsilon, d}}(Y, \lambda_{\varepsilon},  \Gamma) \to \ECH_{*-2}^{L_{\varepsilon, d}}(Y, \lambda_{\varepsilon},  \Gamma)
\end{equation*}
under the basis  $\{\Phi^{\varepsilon}_m(x_{m, i_m})\}_{0 \le m\le d} $  given by  Lemma \ref{lem2}.

By the $ \mathbb{A}(Y) $-commuting  property of the ECH cobordism maps, we have 
\begin{equation}
\begin{split}
U\Phi_d^{\varepsilon}(u^{k }\theta^b e_r^c)  = \Phi_d^{\varepsilon}(u^{k+1}\theta^b e_r^c), 0\le k \le d|e|+\Gamma -g, \Gamma\ne 0 \mbox{ or } d\ne 1.
\end{split}
\end{equation}
So the $U$  map looks simple under this basis. The only problem  is the case that $k=d|e|+\Gamma-g$ because  $u^{d|e| +\Gamma -g+1}\theta^b e_r^c$ is no longer  inside $\mathcal{B}_{d|e| +\Gamma -g}$, and  $\Phi_d^{\varepsilon}(u^{d|e|  +\Gamma-g+1}\theta^b e_r^c)$ is not in the basis given by Lemma \ref{lem2}. Thus, it suffices to compute that the case that $k = d|e|+\Gamma -g$.

We introduce the following notations:  Let $y =u^k \theta^be_r^c \in \mathcal{B}_{d|e|+\Gamma}$, where $0\le k\le d|e|+\Gamma- g$.  Define 
\[y//u^{d|e|+\Gamma-g} =
\begin{cases}
 0 &  0\le k \le d|e|+\Gamma-g-1, \\
\theta^be_r^c &  k =d|e|+\Gamma-g.
\end{cases}
\]
Define 
\begin{equation} \label{eq22}
\begin{split}
Q_d(y): =& u y - (-1)^g\mathfrak{R}_{d|e|+\Gamma}(u, \theta) \left(y//u^{d|e|+\Gamma-g}\right)\\
\end{split}
\end{equation}
It is worth to knowing that   $Q_d(y) \in \operatorname{Span}_{\Q} \mathcal{B}_{d|e|+\Gamma}$  and  $\Pi(Q_d(y) ) =u \wedge \Pi(y)$.

The  computation of the $U$ map under the basis $\{\Phi^{\varepsilon}_d(x)  \vert x \in \mathcal{B}_{d|e|+\Gamma}\}$ is given by the following lemma.

\begin{lemma}  \label{lem7}
Assume $e<\chi(\Sigma)$.  Suppose $\Gamma= 0$ or $\Gamma>2g-2$. 
Let $z =\theta^b e_r^c\in\mathcal E$, where $\mathcal E$ is the chosen basis of
$\Lambda^*H^1(\Sigma;\mathbb Q)$. Then the following formulas hold.

\begin{enumerate} [label=\textbf{C.\arabic*}]
\item \label{C0}
 Assume $0\le k \le d|e| +\Gamma-g -1$. If      $\Gamma \ne 0$ or $d\ne 1$,   then 
\begin{equation} \label{eq41}
U\Phi_d^{\varepsilon}(u^{k}z)  = \Phi_d^{\varepsilon}(u^{k+1}z). 
\end{equation}

If $\Gamma=0 $ and  $d =1$, then 
$$U \Phi_1^{\varepsilon}(u^{k}z) ={\Phi}_1^{\varepsilon}(u^{k+1}z)   + \Psi_0^{\varepsilon}(1 \otimes \tilde{\Phi}^{\varepsilon}_1(u^{k+1}z) ) [\emptyset] .$$

\item\label{C1}
Suppose that either
$$
\Gamma>2g-2,\qquad d\ge1,
$$
or
$$
\Gamma=0,\qquad d\ge2,
$$
and that
$$
(d,e,z)\ne (2, 1-2g,\theta^g).
$$
Then
\begin{equation} \label{eq27}
U\Phi_d^\varepsilon
\left(
u^{d|e|+\Gamma-g}z
\right)
=
\Phi_{d-1}^\varepsilon(z)\
+
\sum_{k=1}^{g}
\frac{(-1)^{k+1}}{k!}
\Phi_d^\varepsilon
\left(
u^{d|e|+\Gamma-g-k+1}\theta^kz
\right).
\end{equation}
Let $y=u^kz \in \mathcal{B}_{d|e| +\Gamma}, 
0\le k\le d|e|+\Gamma-g$, then (\ref{eq41}) and (\ref{eq27}) can be unified to 
\begin{equation*} 
U\Phi_d^\varepsilon\left(y\right)
=\Phi_{d-1}^\varepsilon(y//u^{d|e|+\Gamma -g})\
+
\Phi_d^\varepsilon(Q_d(y)). 
\end{equation*}

\item  \label{C2} Suppose that
$$
\Gamma=0 \mbox{ and } d=1.$$
Then
\begin{equation*}
U\Phi_1^\varepsilon
\left(
u^{|e|-g}z
\right)
=
2^g\Phi_0^\varepsilon(z) + \sum_{k=1}^{g}
\frac{(-1)^{k+1}}{k!}
\Phi_1^\varepsilon
\left(
u^{|e|-g-k+1}\theta^kz
\right).
\end{equation*}
Note that  $\Phi_0^\varepsilon(1) =[\emptyset] $; otherwise $\Phi_0^\varepsilon(z)=0$. (See Lemma \ref{lem17}). 

\item   \label{C3} Suppose that
\[
\Gamma=0,\qquad d=2,\qquad  e=1-2g.
\]
Then
\begin{equation*}
U\Phi_2^\varepsilon
\left(
u^{2|e|-g}z
\right)
= 
\Phi_1^\varepsilon(z)+
\sum_{k=1}^{g}
\frac{(-1)^{k+1}}{k!}
\Phi_2^\varepsilon
\left(
u^{2|e|-g-k+1}\theta^kz
\right)+
g!\delta_{z,\theta^g}[\emptyset],
\end{equation*}
where $\delta_{z,\theta^g}=1$ if $z=\theta^g$, and
$\delta_{z,\theta^g}=0$ otherwise. In particular, 
\begin{equation*}
U\Phi_2^\varepsilon(u^{2|e|-g}\theta^g)
= 
\Phi_1^\varepsilon(\theta^g)+
g![\emptyset] = \tilde{\Phi}_1^\varepsilon(\theta^g),
\end{equation*}

\end{enumerate}
\end{lemma}

\begin{proof}
Consider the situation in  \ref{C0}. If $\Gamma \ne 0 $ or $d \ne 1$, then  $\Phi_d^{\varepsilon}(u^{k}z)   = \tilde{\Phi}_d^{\varepsilon}(u^{k}z)$. The conclusion follows directly the fact that the cobordism maps commute with the $U$ maps. 

If $\Gamma =0 $ and $ d= 1$,  by Definition \ref{definition2}, we have  
\begin{equation*}
\begin{split}
U \Phi_1^{\varepsilon}(u^{k}z)  =&U\left( \tilde{\Phi}_1^{\varepsilon}(u^{k}z)  - \Psi_0^{\varepsilon}(1 \otimes \tilde{\Phi}^{\varepsilon}_1(u^{k}z)) [\emptyset]  \right)\\
=& \tilde{\Phi}_1^{\varepsilon}(u^{k+1}z) \\
=&{\Phi}_1^{\varepsilon}(u^{k+1}z)   + \Psi_0^{\varepsilon}(1 \otimes \tilde{\Phi}^{\varepsilon}_1(u^{k+1}z)) [\emptyset] .  
\end{split}
\end{equation*}

 Let
       $ \{x_{n,i_n}:  1\le i_n \le \dim H^*(\Sym^{n|e| + \Gamma}\Sigma, \Q)\}$
be the  elements of $\mathcal{B}_{n|e| +\Gamma }$, where  $n \ge 0$.  Now we consider the cases \ref{C1}, \ref{C2} and \ref{C3}, i.e., the top degree case that 
$$
        y=u^{d|e| +\Gamma -g} z =u^{d|e| +\Gamma -g}\theta^b e_r^c.
$$
Then 
$\Phi^{\varepsilon}_d(y) \in \ECH^{L_{\varepsilon, d}}(Y, \lambda_{\varepsilon},  \Gamma) $.
Because the $U$ map decreases the action,  $U\Phi^{\varepsilon}_d(y) \in \ECH^{L_{\varepsilon, d}}(Y, \lambda_{\varepsilon},  \Gamma)$. 
By Lemma \ref{lem2}, we can write
\[
        U \Phi^{\varepsilon}_d(y)
        =\sum_{n=0}^{d}\sum_{i_n} a_{n, i_n}\Phi^{\varepsilon}_n(x_{n, i_n}),  \mbox{ where } a_{n, i_n} \in \Q.  
\]
Pair with $\Phi^{\varepsilon}_m(x_{m, j_m})$.  By Lemma \ref{lem3}, only the
$n=m$ summand survives:
\[
\begin{aligned}
      &\langle  \Phi^{\varepsilon}_m(x_{m, j_m}),  U \Phi^{\varepsilon}_d(y) \rangle_{ech}=  \Psi^{\varepsilon}_m(x_{m, j_m} \otimes U \Phi^{\varepsilon}_d(y))\\
        =&\sum_{n=0}^{d}\sum_{i_n} a_{n, i_n}\,
         \langle  \Phi^{\varepsilon}_m(x_{m, j_m}), \Phi^{\varepsilon}_n(x_{n, i_n}) \rangle_{ech}  \\
        =&\sum_{i_m} a_{m, i_m} q^{m|e| +\Gamma }( \Pi(x_{m,j_m}),  \Pi(x_{m, i_m})),
\end{aligned}
\]
where $q^{m|e| +\Gamma}$  denote the intersection pairing of $H^*(\Sym^{m|e| +\Gamma}\Sigma, \Q)$. Denote   $$q^{m|e|+\Gamma }_{i_m, j_m} :=q^{m|e| +\Gamma}(\Pi(x_{m, i_m}), \Pi(x_{m, j_m})).$$
Then 
$$
        a_{m, i_m}
        =\sum_{ j_m}(q^{m|e| +\Gamma })^{-1}_{i_m, j_m}\,
        \Psi^{\varepsilon}_m(x_{m, j_m} \otimes U  \Phi^{\varepsilon}_d(y)).
$$
Note that $ \Psi^{\varepsilon}_m(x_{m, j_m} \otimes U  \Phi^{\varepsilon}_d(y))=  \Psi^{\varepsilon}_m(x_{m, j_m} \otimes U  \tilde{\Phi}^{\varepsilon}_d(y))$ becasue $U[\emptyset] =0$. Using the composition law  of ECH cobordism maps, this becomes
\begin{equation} \label{eq43}
        a_{m, i_m}
        =\sum_{ j_m}(q^{m|e|+\Gamma})^{-1}_{i_m, j_m} 
        \Gr\bigl(X,A_{d,m},x_{m, j_m}u^{d|e|+\Gamma  -g+1}\theta^b e_r^c\bigr),
\end{equation}
where $A_{d, m} \in H_2(X, \Z) $ such that $A_{d, m} \vert_{X_+} =A_+(d)$ and 
$A_{d, m} \vert_{X_-} =A_-(m)$.  By Lemma \ref{lem4}, its  ECH index is 
\begin{equation*}
\begin{split}
        I(A_{d,m})&=(d-m+1)(d+m)|e|+(d-m)\chi(\Sigma)   +2\Gamma(d-m+1)\\
        &=(d-m)( (d+m) |e| + \chi(\Sigma)) + (d+m)|e|+  2\Gamma(d-m+1). 
        \end{split}
\end{equation*}

The Gromov invariants   $\Gr\bigl(X,A_{d,m},x_{m, j_m}u^{d|e|  +\Gamma -g+1}\theta^b e_r^c\bigr) \ne 0$ only if  
\begin{equation}\label{eq11}
I(A_{d,m}) =\degop x_{m, j_m}u^{d|e| +\Gamma  -g+1}\theta^b e_r^c.  
\end{equation}

We now compute all the coefficients appearing in (\ref{eq43}). We distinguish two cases according to whether
$m\neq 0$ or $\Gamma\ne 0$, or $m=\Gamma=0$.
In the first case, since $0\le m\le d$, we further divide the three   exhaustive subcases: $d-m\ge 2$, $d-m=1$, and $d=m$.

\paragraph{Case 1: $m\ne 0$ or $\Gamma \ne 0$} We  first consider the case  $m\ne 0$ or $\Gamma \ne 0$.   Then $m|e| + \Gamma >2g-2$ and  $x_{m, j_m} \in \mathcal{B}_{m|e| +\Gamma}$ is of the form $ u^{a'} \theta^{b'}e_{r'}^{c'}$, where $a' \le m|e|+\Gamma -g$.  Observes  that 
\begin{equation} \label{eq2}
\begin{split}
\degop x_{m, j_m}u^{d|e| +\Gamma  -g+1}\theta^b e_r^c \le &2(d|e|  +\Gamma -g+1) +2 (m|e|  +\Gamma  -g) +2g\\
 \le& 2(d+m)|e|  +4\Gamma  +2  -2g. 
\end{split}
\end{equation}

 Under the assumption $m\ne 0$ or $\Gamma \ne 0$, we divide the situation into the following three subcases. 
\begin{itemize}
\item
If $d-m \ge 2$, then  
\begin{equation*}
\begin{split}
I(A_{d,m})& \ge 2(d+m) |e| + \chi(\Sigma)  +\chi(\Sigma) +(d+m)|e|  + 2\Gamma(d-m+1)  \\
& \ge  2(d+m) |e| + \chi(\Sigma)   + ( |e| + \chi(\Sigma)  ) + 2\Gamma(d-m+1)\\
&> 2(d+m) |e| + \chi(\Sigma)  + 6\Gamma \\
&\ge \degop x_{m,j_m}u^{d|e| +\Gamma -g+1}\theta^b e_r^c . 
        \end{split}
\end{equation*}
In the last second line,  we have used the assumption that $|e|+\chi(\Sigma)>0$.
Thus,  $$\Gr\bigl(X,A_{d,m},x_{m,j_m}u^{d|e| +\Gamma-g+1}\theta^b e_r^c\bigr) =0$$ because (\ref{eq11}) does not hold. 

\item
If $d-m=1$, then 
\begin{equation*}
\begin{split}
I(A_{d,d-1})  = 2(2d-1)|e| + \chi(\Sigma)+4\Gamma  =2 ((d-1 )|e| +\Gamma) +  2(d|e| +\Gamma  -g+1). 
        \end{split}
\end{equation*}
When $x_{d-1, j_{d-1}}$ satisfies (\ref{eq11}), we have 
\begin{equation} \label{eq3}
\degop x_{d-1, j_{d-1}}\theta^b e_r^c =2 ((d-1)|e| +\Gamma) . 
\end{equation}
Assume (\ref{eq3}) holds. By the definition of the basis  $\mathcal{B}_{(d-1)|e| +\Gamma}$,  $x_{d-1,j_{d-1}} $ is of the form $u^{(d-1)|e|+\Gamma  -g-i} \theta^{b_1} e^{c_1}_{r_1} $.  Then
$$x_{d-1,j_{d-1}}\theta^b e_r^c = u^{(d-1)|e|+\Gamma   -g-i} \theta^{b_1} e^{c_1}_{r_1} \theta^b e_r^c.$$
The  degree of  $ \theta^{b_1} e^{c_1}_{r_1} \theta^b e_r^c$ is at most $2g$. Thus,  (\ref{eq3}) forces $i=0$ and $\degop  \theta^{b_1} e^{c_1}_{r_1} \theta^b e_r^c =2g $. Since $\dim \Lambda^{2g}H^1(\Sigma, \Q) =1$, we   have  $ \theta^{b_1} e^{c_1}_{r_1} \theta^b e_r^c = c_0 \theta^g$ for some constant $c_0$.  Then  
$$x_{d-1,j_{d-1}}\theta^b e_r^c = c_0 u^{(d-1)|e| +\Gamma -g} \theta^g. $$

By the definition of the intersection pairing, we have $c_0=\frac{1}{g!}q^{(d-1)|e| +\Gamma }(\Pi(x_{d-1,j_{d-1}}), \Pi(\theta^b e_r^c )).$

By Theorem \ref{thm0}, we obtain
\begin{equation}
\begin{split}
 &\Gr\bigl(X,A_{d,d-1},x_{d-1,j_{d-1}}u^{d|e|+\Gamma-g  +1}\theta^b e_r^c\bigr) \\
=&q^{(d-1)|e| +\Gamma}(\Pi(x_{d-1,j_{d-1}}), \Pi( \theta^b e_r^c)) \Gr(X,A_{d,d-1}, \frac{u^{(2d-1)|e| +2\Gamma   -2g +1} \theta^g}{g!} ) \\ =& q^{(d-1)|e| +\Gamma }(\Pi(x_{d-1,j_{d-1}}), \Pi(\theta^b e_r^c)). 
 \end{split}
\end{equation}

If the degree condition (\ref{eq3}) does not hold, then the above equality still hold because both of the Gromov-Taubes invariants and intersection pairing $q^{(d-1)|e| +\Gamma }(\Pi(x_{d-1,j_{d-1}}), \Pi(\theta^b e_r^c))$ vanish. 

Taking into account all  the above disscusion, we have 
\begin{equation}\label{eq25}
\begin{split}
&\sum_{i_{d-1}} a_{d-1, i_{d-1}}\Phi_{d-1}^{\varepsilon}(x_{d-1, i_{d-1}})\\
=&
\sum_{i_{d-1}}\sum_{ j_{d-1}}(q^{(d-1)|e| +\Gamma})^{-1}_{i_{d-1}, j_{d-1}} 
        \Gr\bigl(X,A_{d,d-1},x_{d-1, j_{d-1}}u^{d|e|+\Gamma-g+1}\theta^b e_r^c\bigr)  \Phi_{d-1}^{\varepsilon}(x_{d-1, i_{d-1}})\\
=&\sum_{i_{d-1}}\sum_{ j_{d-1}}(q^{(d-1)|e|+\Gamma})^{-1}_{i_{d-1}, j_{d-1}} 
        q^{(d-1)|e| +\Gamma }(\Pi(x_{d-1,j_{d-1}}), \Pi(\theta^b e_r^c)) \Phi_{d-1}^{\varepsilon}(x_{d-1, i_{d-1}})\\
=& \Phi_{d-1}^{\varepsilon}(\theta^b e_r^c)= \Phi_{d-1}^{\varepsilon}(y//u^{d|e|+\Gamma- g}). 
 \end{split}
\end{equation}

\item 
Finally, we consider the case $d=m$.  In this case,  $I(A_{d, d}) = 2(d|e|+\Gamma) $. 
Recall that  
\begin{equation} \label{eq18}
\begin{split}
 uy =
  (-1)^g\mathfrak{R}_{d|e|+\Gamma}(u, \theta) \big(y//u^{d|e|+\Gamma-g} \big) + Q_d(y),
 \end{split}
\end{equation}
where 
$$Q_d(y): = \left(\sum_{k=1}^g (-1)^{k+1} \frac{1}{k!} \theta^k u^{d|e|+\Gamma -g -k +1} \right) (y//u^{d|e|+\Gamma-g}).$$
Note that $Q_d(y) \in \operatorname{Span}_{\Q} \mathcal{B}_{d|e|+\Gamma }$. Equation (\ref{eq18}) implies that $\Pi(Q_d(y)) = \Pi(uy)$.

Then $x_{d, j_d} uy = x_{d, j_d}  \big(y//u^{d|e|+\Gamma-g} \big) (-1)^g\mathfrak{R}_{d|e|+\Gamma} (u, \theta)+x_{d, j_d}Q_d(y).$ If the degree condition (\ref{eq11}) holds, i.e., 
\begin{equation} \label{eq34}
\degop x_{d, j_{d}} +  2 + \degop y =2(d|e|+\Gamma),  
\end{equation}
then 
\begin{equation*}
\begin{split}
&\degop   x_{d, j_d}  \big(y//u^{d|e|+\Gamma-g} \big)  =2(d|e| +\Gamma-1) -2(d|e|+\Gamma-g)= 2g-2,\\
\mbox{ and }&\degop x_{d, j_d} + \degop Q_{d}(y) =2(d|e|+\Gamma).
\end{split}
\end{equation*}
 By Lemma \ref{lem1} and Lemma \ref{lem3}, we have 
\begin{equation*}
\begin{split}
\Gr(X, A_{d, d}, x_{d, j_d}uy)= &\Gr(X, A_{d, d}, x_{d, j_d}Q_{d}(y))\\
=&\Psi^{\varepsilon}_d(x_{d, j_d} \otimes   \tilde{\Phi}^{\varepsilon}_d(Q_d(y)))  \\
=&\langle \Phi^{\varepsilon}_d(x_{d, j_d}),  \Phi^{\varepsilon}_d(Q_d(y))\rangle_{ech} =q^{d|e|+\Gamma}(\Pi(x_{d, j_d}), \Pi(Q_{d}(y))).
\end{split}
\end{equation*}
In the last second step of the above equation, we remark that  $\Psi^{\varepsilon}_d(x_{d, j_d} \otimes   \tilde{\Phi}^{\varepsilon}_d(Q_d(y)))  
= \Psi^{\varepsilon}_d(x_{d, j_d} \otimes   {\Phi}^{\varepsilon}_d(Q_d(y)))   + c \Psi^{\varepsilon}_d(x_{d, j_d} \otimes    [\emptyset])   $  and the term  $\Psi^{\varepsilon}_d(x_{d, j_d} \otimes    [\emptyset]) \ne 0$ only if $d=\Gamma =0$.
However, $d=\Gamma =0$ is ruled out by the assumption $\Gamma \ne 0$ or $m \ne 0$.

If the degree condition (\ref{eq34}) does not hold, then 
$$\Gr(X, A_{d, d}, x_{d, j_d}uy)=0 = q^{d|e|+\Gamma}(\Pi(x_{d, j_d}), \Pi(Q_{d}(y))).$$

In sum, we have  
\begin{equation} \label{eq26}
\begin{split}
&\sum_{i_{d}} a_{d, i_{d}}\Phi_{d}^{\varepsilon}(x_{d, i_{d}})\\
=&
\sum_{i_{d}}\sum_{ j_{d}}(q^{d|e|+\Gamma})^{-1}_{i_{d}, j_{d}} 
        \Gr\bigl(X,A_{d,d},x_{d, j_d}u^{d|e|+\Gamma-g+1}\theta^b e_r^c\bigr)  \Phi_{d}^{\varepsilon}(x_{d, i_{d}})\\
=&\sum_{i_{d}}\sum_{ j_{d}}(q^{d|e|+\Gamma})^{-1}_{i_{d}, j_{d}} 
        q^{d|e|+\Gamma }(\Pi(x_{d,j_{d}}), \Pi(Q_d(y))) \Phi_{d}^{\varepsilon}(x_{d, i_{d}})\\
=&\Phi_{d}^{\varepsilon} \left( \sum_{i_{d}}\sum_{ j_{d}}(q^{d|e|+\Gamma})^{-1}_{i_{d}, j_{d}} 
        q^{d|e|+\Gamma }(\Pi(x_{d,j_{d}}), \Pi(Q_d(y)))  x_{d, i_{d}} \right) \\
=& \Phi_{d}^{\varepsilon}(Q_d(y)). \\
 \end{split}
\end{equation}
\end{itemize}

Combining (\ref{eq25}) and (\ref{eq26}),  we obtain the conclusion (\ref{eq27}) for the cases  $(\Gamma >2g-2, d \ge1)$.

\paragraph{Case 2: $\Gamma=0$ and $m=0$} Now consider the exceptional  case that $\Gamma=0$ and $m=0$.  Then $x_{0, 1}=1$ and the observation (\ref{eq2}) does not hold. Instead, we have  
\begin{equation*}  
\degop u^{d|e|    -g+1}\theta^b e_r^c \le 2d|e| +2.
\end{equation*}
If $d\ge 2$, then 
\begin{equation*}
\begin{split}
I(A_{d, 0})=&d^2|e| +d|e| +d\chi(\Sigma) \\
\ge& 2d|e| + d(|e| + \chi(\Sigma)) \ge 2d|e| +2d\\
>& 2d|e| +2 \ge \degop u^{d|e|    -g+1}\theta^b e_r^c 
\end{split}
\end{equation*}
provided that $|e| \ge 2g .$   Then $\Psi^{\varepsilon}_0(1  \otimes U  \Phi^{\varepsilon}_d(y)) =\Gr(X, A_{d, 0}, uy)=0$ because (\ref{eq11}) does not hold.   Combining this with (\ref{eq25}),  (\ref{eq26}), the formula (\ref{eq27}) still hold  for the case $(\Gamma=0, d\ge 2)$ provided that   $|e| \ge 2g$. 

If $e=1-2g$, then   $I(A_{d, 0}) =d^2|e|+d$. The condition  $I(A_{d, 0}) =\deg u^{d|e|    -g+1}\theta^b e_r^c$ implies that 
$$d(d-2)|e| +d- 2 +2g- \deg \theta^b e_r^c=0.$$
Thus, (\ref{eq11})
 holds if and only if $\theta^be_r^c =\theta^g$ and $d=2$. In this case, by Theorem \ref{thm0}, we have  $$\Gr(X, A_{2, 0}, u^{2|e|-g+1} \theta^g) =g!.$$ 

If $d=1$, then (\ref{eq11}) is equivalent to 
$$2|e| + \chi(\Sigma) =2|e| + 2-2g+\degop \theta^be_r^c .$$
This forces $\theta^be_r^c=1$.  By Theorem \ref{thm0}, $\Gr(X, A_{1, 0}, u^{|e|-g+1}) =2^g$.

Thus,  the contribution to the $\Phi^{\varepsilon}_0(1)$ term is  
\begin{equation}  \label{eq24}
a_{0, 1} \Phi^{\varepsilon}_0(1)=
\begin{cases}
g! \Phi^{\varepsilon}_0(1) =g![\emptyset]& \theta^be_r^c =\theta^g,  e=1-2g,   \mbox{ and } d=2\\
2^{g}\Phi^{\varepsilon}_0(1)=2^g[\emptyset] & \theta^be_r^c =1 \mbox{ and } d=1\\
0 =2^g\Phi^{\varepsilon}_0(\theta^be_r^c) (\mbox{Lemma \ref{lem17}})  & \theta^be_r^c  \ne 1 \mbox{ and } d=1\\
0 & else.
\end{cases}
\end{equation}

Combining (\ref{eq25}), (\ref{eq26}) and (\ref{eq24}), we obtain the results for \ref{C1}, \ref{C2} and \ref{C3}. 

\end{proof}

The conclusions in Lemma \ref{lem7} can be simplified into a simple form when $|e| +\Gamma \ge 2g$.  This simple form is very useful when we prove Theorem \ref{thm1} and compute the ECH spectrum. To illustrate this,     we define  a  polynomial 
\begin{equation}  \label{eq31}
P_m (u, \theta):= u^{m|e| +\Gamma -g} + \sum_{k=1}^g (-1)^{k}  \frac{1}{k!} \theta^{k} u^{m|e| +\Gamma -g -k}.
\end{equation}
Note that the above polynomial  is well defined when   $m|e|+\Gamma \ge 2g$.  When it is well defined, we have $P_m (u, \theta) \in \operatorname{Span}_{\Q}\mathcal{B}_{m|e|+\Gamma}$ and $$ uP_m (u, \theta) =(-1)^g\mathfrak{R}_{m|e|+\Gamma}(u, \theta) .$$

\begin{lemma} \label{lem18}
 Assume $\Gamma=0$ or $\Gamma>2g-2$.   Suppose that  $|e| +\Gamma \ge 2g$. Then we have 
  \begin{equation} \label{eq16}
U  \Phi^{\varepsilon}_d(P_d (u, \theta) z) = \Phi^{\varepsilon}_{d-1}(z)  
\end{equation}
for $z \in \mathcal{E}$ and $d \ge 1$.     If $\Gamma=0$ and $e=1-2g$, then (\ref{eq16}) still holds for $d \ge3$. 
  \end{lemma}
  \begin{proof}
  Suppose that we are under the conditions (\ref{C1}) that (\ref{eq27}) holds. Then $\Phi_d^{\varepsilon}(x) = \tilde{\Phi}_d^{\varepsilon}(x)$.  The conclusion follows from  (\ref{eq27}) 
 and the fact that ECH cobordism maps commute the $\mathbb{A}(\Sigma)$-action.

Suppose that we are under the conditions (\ref{C2}) and $|e|\ge 2g$.   Now   
$$\Phi_1^{\varepsilon}(u^{|e| -g-k} \theta ^k z ) = \tilde{\Phi}_1^{\varepsilon} (u^{|e| -g-k} \theta ^k z ) -  \Psi_0(1 \otimes \tilde{\Phi}_1^{\varepsilon} (u^{|e| -g-k} \theta ^k z ) [\emptyset],$$
where $z\in \mathcal{E}$.  Then for $1 \le k\le |e|-g$, we have 
\begin{equation*}
\begin{split}
U  \Phi_1^{\varepsilon}(u^{|e| -g-k} \theta ^k z )   =  &U \tilde{\Phi}_1^{\varepsilon} (u^{|e| -g-k} \theta ^k z )= \tilde{\Phi}_1^{\varepsilon} (u^{|e| -g-k+1} \theta ^k z ) \\
=&
\Phi_1^{\varepsilon}(u^{|e| -g-k+1} \theta ^k z)  +  \Psi_0(1 \otimes \tilde{\Phi}_1^{\varepsilon} (u^{|e| -g-k+1} \theta ^k z ) [\emptyset].
\end{split}
\end{equation*}

By Lemma \ref{lem7}, we obtain
\begin{equation*}
\begin{split}
U  \Phi^{\varepsilon}_1(P_1 (u, \theta) z) =& U \Phi^{\varepsilon}_1(u^{|e|-g}z) 
+  \sum_{k=1}^g \frac{(-1)^k}{k!}U \Phi^{\varepsilon}_1(u^{|e|-g-k} \theta^kz) \\
=& 2^g\Phi_0^\varepsilon(z) + \sum_{k=1}^{g}
\frac{(-1)^{k+1}}{k!}
\Phi_1^\varepsilon
\left(
u^{|e|-g-k+1}\theta^kz
\right) +  \sum_{k=1}^g \frac{(-1)^k}{k!} \Phi^{\varepsilon}_1(u^{|e|-g-k+1} \theta^k z)  \\
+&  \sum_{k=1}^g   \frac{(-1)^k}{k!}  \Psi_0(1 \otimes \tilde{\Phi}_1^{\varepsilon} (u^{|e| -g-k+1} \theta ^k z ) [\emptyset]\\
=&2^g\Phi_0^\varepsilon(z)  + \sum_{k=1}^g   \frac{(-1)^k}{k!}  \Psi_0(1 \otimes \tilde{\Phi}_1^{\varepsilon} (u^{|e| -g-k+1} \theta ^k z ) [\emptyset]. 
\end{split}
\end{equation*}

By Lemma \ref{lem14}, if  $z \ne 1$, then $\Phi_0^\varepsilon(z) =0 $ and   $ \Psi_0(1 \otimes \tilde{\Phi}_1^{\varepsilon} (u^{|e| -g-k+1} \theta ^k z ) =0$ for  $1 \le k \le|e|- g$.   Then (\ref{eq16}) holds. 
If $z=1$, then 
\begin{equation*}
\begin{split}
&2^g\Phi_0^\varepsilon(1)  + \sum_{k=1}^g   \frac{(-1)^k}{k!}  \Psi_0(1 \otimes \tilde{\Phi}_1^{\varepsilon} (u^{|e| -g-k+1} \theta ^k  ) [\emptyset]\\
=&2^g\Phi_0^\varepsilon(1)   + \left( \sum_{k=1}^g   \frac{(-1)^k}{k!}  \frac{g!2^{g-k}}{(g-k)!}\right) \\
=& [\emptyset] = \Phi_0^\varepsilon(1). 
\end{split}
\end{equation*}

Note that the assumption   $|e| +\Gamma \ge 2g$ rules  out the  case \ref{C3}.  In sum, we finish the proof. 

\end{proof}

\section{Proof of Theorem \ref{thm1}}

Write $V^{\varepsilon}_d:= \ECH^{L_{\varepsilon, d}}(Y, \lambda_{\varepsilon}, \Gamma)$.   Then Lemma \ref{lem2} implies that  $V^{\varepsilon}_d =\operatorname{Span}_{\Q}\{\Phi_m(x_{m,i_m})\}_{0\le m \le d}$.  Define  
\begin{equation*}
\pi_d:  V^{\varepsilon}_d  \to H^*(\Sym^{d|e|+\Gamma} \Sigma, \Q) 
\end{equation*}
by sending  $\Phi^{\varepsilon}_m(x_{m,i_m})$ to $0$ for $0\le m\le d-1$ and 
$$\pi_d( \Phi^{\varepsilon}_d(x_{d,i_d})) =\Pi(x_{d, i_d} ), $$
where $\Pi: \mathbb{A}(\Sigma) \to H^*(\Sym^{d|e|+\Gamma} \Sigma, \Q) $ is the projection.  Here we set $ H^*(\Sym^0\Sigma, \Q)  =\Q$ with $u=0$ action.

The following lemma is an observation regarding the relation between $U$ map and $\wedge u$ on  cohomology rings of symmetric product of $\Sigma$. 

\begin{lemma} \label{lem10}
Suppose $\Gamma= 0$ or $\Gamma>2g-2$.  We have the following  diagram for $1 \le d\le d_{\varepsilon}$: 
\begin{equation} \label{eq13}
\begin{tikzcd}
0   \arrow[r, " "]  & V^{\varepsilon}_{d-1}  \arrow[r, " "]  \arrow[d, "U"] & V^{\varepsilon}_d \arrow[r, "\pi_d"] \arrow[d, "U"'] & H^*(\Sym^{d|e|+\Gamma} \Sigma, \Q)  \arrow[r, " "] \arrow[d, "u\wedge "] &  0  \\
 0  \arrow[r, " "']                     & V^{\varepsilon}_{d-1} \arrow[r, " "']                    & V^{\varepsilon}_d  \arrow[r, "\pi_d"]                     & H^*(\Sym^{d|e|+\Gamma} \Sigma, \Q) \arrow[r, " "']                    & 0
\end{tikzcd}
\end{equation}
Moreover,   $U$ is a nilpotent. 
\end{lemma}
\begin{proof}
It suffices to check $ \pi_d \circ U =u \wedge  \pi_d$ for $\Phi_d^{\varepsilon}(u^{k}\theta^be^c_r)$, where $u^{k}\theta^be^c_r \in \mathcal{B}_{d|e|+\Gamma}$. 
 
By Lemma \ref{lem7}, for $y =u^{k}\theta^be^c_r $, we have 
\begin{equation} \label{eq14}
\begin{split}
&\pi_d \circ U \Phi_d^{\varepsilon}(y) =\pi_d(\Phi_d^{\varepsilon}(Q_d(y))) =\Pi(Q_d(y)). \end{split}
\end{equation}
On the other hand,  we have 
\begin{equation} \label{eq15}
\begin{split}
u\wedge \pi_d(\Phi_d^{\varepsilon}(y)) &= u\wedge \Pi(y)   =\Pi((-1)^g\mathfrak{R}_{d|e|+\Gamma}(u, \theta) \big(y//u^{d|e|+\Gamma-g} \big) + Q_d(y)) = \Pi(Q_d(y)) .
\end{split}
\end{equation}
Comparing (\ref{eq14}) and (\ref{eq15}),  we have $\pi_d \circ U \Phi_d^{\varepsilon}(y)  = u\wedge \pi_d(\Phi_d^{\varepsilon}(y))$.  

By the diagram, we know that $\pi_d \circ U^{d|e|+\Gamma+1}\sigma = \wedge^{d|e|+\Gamma+1} u \wedge \pi_d(\sigma) =0$, where $\sigma \in V^{\varepsilon}_{d}$. Hence,  $$U^{d|e|+\Gamma+1}\sigma\in V^{\varepsilon}_{d-1}. $$ 
Let $N=\sum_{m=0}^d(m|e|+\Gamma+1)  =\frac{d(d+1)}{2}|e| +(d+1)(\Gamma+1)$. Then $U^N \sigma \in V_{-1}=0 $, i,e.,  $U$ is a nilpotent.  

\end{proof}

By Lemma \ref{lem10} and Jordan decomposition,  we know that $V_d^{\varepsilon}$ consists of finite $U$-towers, i.e., 
\begin{equation*}
V_d^{\varepsilon} \cong \oplus_{j=1}^r \Q[U^{-1}]/U^{-l_{d, j}}\Q[U^{-1}],
\end{equation*}
where $l_{d, j} \ge 0$ and $r =\dim \coker \{ U: V_d^{\varepsilon} \to V_d^{\varepsilon}  \}. $  Our next task is to compute $   \dim \coker \{ U: V_d^{\varepsilon} \to V_d^{\varepsilon}  \}$.

Applying  Snake Lemma to the diagram (\ref{eq13}), we have 
\begin{equation} \label{eq19}
0\to \ker U_{d-1} \to \ker U_d \to \ker \wedge u  \xrightarrow{\delta_d}  \coker U_{d-1} \to \coker U_d  \xrightarrow{\pi_d}    \coker \wedge u \to 0.
\end{equation}
Here $U_m$ denote the restriction  $U: V_m^{\varepsilon} \to V_m^{\varepsilon} $.

 We first figure out the dimension of $ \ker \wedge u$ and  $\coker \wedge u$ in the following lemma.
 \begin{lemma}  \label{lem11}
Assume $d|e|+\Gamma \ge 2g$.
Then  $  \{ \Pi(P_d (u, \theta) z) : z \in \mathcal{E}\} $ forms a  basis of $\ker \wedge u$.  In particular,
$$\dim \ker \wedge u =\dim \coker \wedge u =2^{2g}. $$
 \end{lemma}
 \begin{proof}
 First, note that $uP_d (u, \theta) =(-1)^g\mathfrak{R}_{d|e|+\Gamma}(u, \theta)$ when $d|e|+\Gamma \ge 2g$.    Thus,
 $$u\wedge \Pi(P_d (u, \theta) z) = \Pi(uP_d (u, \theta)z) =\Pi((-1)^g\mathfrak{R}_{d|e|+\Gamma} (u, \theta)z ) =0. $$
 Thud, $ \{ \Pi(P_d(u, \theta) z) : z \in \mathcal{E}\}  \subset \ker \wedge u$.

Suppose  $\sum_{z \in \mathcal{E}} a_z  \Pi(P_d(u, \theta) z) =0$. For $z' \in \mathcal{E}$, then 
 \begin{equation*}
 \begin{split}
0=&\Pi(z') \wedge (\sum_{z \in \mathcal{E}} a_z  \Pi(P_d (u, \theta) z)) = \sum_{z \in \mathcal{E}} a_z  \Pi(P_d(u, \theta) z'z)  \\
=&  \sum_{z \in \mathcal{E}} a_z  \Pi(u^{d|e|+\Gamma -g}  z'z) +  \sum_{z \in \mathcal{E}} \sum_{k=1}^g (-1)^{k}  \frac{1}{k!}   \Pi(u^{d|e| +\Gamma -g -k} \theta^{k}zz') .\\
\end{split}
\end{equation*}
Note that coefficients of the term  $\Pi(u^{d|e|+\Gamma-g} \frac{\theta^g}{g!})$ is
$$0= \sum_{z \in \mathcal{E}, \degop z + \degop z'=2g} a_z    \langle z, z' \rangle_{\Lambda^*H^1(\Sigma, \Q)}. 
$$
Here $\langle \cdot,  \cdot \rangle_{\Lambda^*H^1(\Sigma, \Q)} $ is the intersection pairing  of $\Lambda^*H^1(\Sigma, \Q)$, and thus it is nondegenerate.  Then $a_z =0$ for all $z \in \mathcal{E}$. Therefore,  $  \{ \Pi(P_d(u, \theta) z) : z \in \mathcal{E}\} $ is linear independent. 

 We express $y \in H^*(\Sym \Sigma, \Q)$ in terms of  the basis $\{\Pi(x_{d, i}), x_{d, i} \in \mathcal{B}_{d|e|+\Gamma}\}=\{\mathcal{E}, u\mathcal{E},..., u^{d|e|+\Gamma-g} \mathcal{E}\}$, i.e., 
  \begin{equation*}
 \begin{split}
y = \sum_{i, j} (q^{d|e|+\Gamma})^{-1}_{i,j}q^{d|e|+\Gamma}(  \Pi(x_{d, j}), y ) \Pi(x_{d, i}).
 \end{split}
\end{equation*}
If $y \perp \operatorname{Im} \wedge u$ in the sene that $q^{d|e|+\Gamma}(y', y)=0$ for any $y' \in \operatorname{Im} \wedge u$,  then $q^{d|e|+\Gamma}( \Pi(x_{d, j}), y ) =0 $ if $x_{d, j} \in u^k\mathcal{E}$ for $k \ge 1$. As a result, we have 
  \begin{equation*}
 \begin{split}
y =  \sum_{x_{d, i} \in \mathcal{E}}  q^{d|e|+\Gamma}( \Pi(x_{d, i}), y ) (\Pi(x_{d, i}))^*,\\
 \end{split}
\end{equation*}
where  $ (\Pi(x_{d, i}))^* = \sum_{j} (q^{d|e|+\Gamma})^{-1}_{i,j} \Pi(x_{d, j}). $  Consequently, we have $$\dim (\operatorname{Im} \wedge u)^{\perp} \le \dim \Lambda^*H^1(\Sigma, \Q) =2^{2g}.$$

Because the intersection pairing is nondegenerate, we identify   (noncanonically) $ \coker \wedge u$ with $\operatorname{Im} \wedge u$. Thus, 
$$\dim \coker \wedge u =\dim (\operatorname{Im} \wedge u)^{\perp} \le 2^{2g}.$$ 
On the other hand,   we have  $$\dim \coker \wedge u = \dim \ker \wedge u  \ge \dim  \operatorname{Span}  \{ \Pi(P_d z) : z \in \mathcal{E}\} =2^{2g}. $$

 \end{proof}

\begin{lemma} \label{lem16}
Assume that $e < \chi(\Sigma)$ and that either
$\Gamma=0$ or $\Gamma>2g-2$. If $|e|+\Gamma \ge 2g$,  every $d\ge 1 $, we have 
$$
V_{d-1}^{\varepsilon}\subset UV_d^{\varepsilon}.
$$
If $\Gamma=0$ and $e=1-2g$, then the same conclusion holds for $d \ge 2$. 
\end{lemma}

\begin{proof}
For each \(m\ge 0\), set
\[
B_m^{\varepsilon}
:=
\operatorname{Span}_{\mathbb Q}
\left\{
\Phi_m^{\varepsilon}(u^kz)
\ \middle|\
u^kz\in\mathcal B_{m|e|+\Gamma},
\ z\in\mathcal E 
\right\}.
\]
If $\Gamma=0$, set $B_0^{\varepsilon}=\Q[\emptyset]$. By definition,
$$
V_d^{\varepsilon}
=
\bigoplus_{m=0}^{d}B_m^{\varepsilon}.
$$

 We divide the proof into two cases. 
\paragraph{Case 1:  $|e|+\Gamma\ge 2g$.}
  Then  the polynomial  $P_m(u,\theta)$ (\ref{eq31})  is well defined for every $m\ge1$.
By Lemma \ref{lem18}, we have 
$$
U\Phi_m^{\varepsilon}\bigl(P_m(u,\theta)z\bigr)
=
\Phi_{m-1}^{\varepsilon}(z),
\qquad
m\ge1,\quad z\in\mathcal E.
$$
For $m \ne 2$ or $\Gamma \ne 0$, applying $U^k$ to both sides gives
$$
\Phi_{m-1}^{\varepsilon}(u^kz)=U^k\Phi_{m-1}^{\varepsilon}(z)
=
U^{k+1}
\Phi_m^{\varepsilon}\bigl(P_m(u,\theta)z\bigr)
$$
whenever \(u^kz\in\mathcal B_{(m-1)|e|+\Gamma}\), i.e., $0\le k\le(m-1)|e|+\Gamma- g$.

If $\Gamma=0 $ and  $m=2$, note that $\Phi_1^{\varepsilon}$  no longer commutes  with $U$.  By Lemma \ref{lem14}, we have $\Phi_{1}^{\varepsilon}(z)= \tilde{\Phi}_{1}^{\varepsilon}(z)$.   By \ref{C0} case in Lemma \ref{lem7},  $0 \le k \le  |e|-g$,  we have  
\begin{equation*}
\begin{split}
U^{k+1}
\Phi_2^{\varepsilon}\bigl(P_2(u,\theta)z\bigr)=&U^k\Phi_{1}^{\varepsilon}(z)\\
=&   {\Phi}_{1}^{\varepsilon}(u^kz)  +  \Psi_0^{\varepsilon}(1 \otimes \tilde{\Phi}^{\varepsilon}
_1(u^kz)) [\emptyset] \\
=&{\Phi}_{1}^{\varepsilon}(u^kz)  +  \Psi_0^{\varepsilon}(1 \otimes \tilde{\Phi}^{\varepsilon}
_1(u^kz)) U{\Phi}_{1}^{\varepsilon}(P_1(u, \theta)).
\end{split}
\end{equation*}
Thus,  
$${\Phi}_{1}^{\varepsilon}(u^kz) = U\left( U^{k} \Phi_2^{\varepsilon}\bigl(P_2(u,\theta)z\bigr) - \Psi_0^{\varepsilon}(1 \otimes \tilde{\Phi}^{\varepsilon}
_1(u^kz)) {\Phi}_{1}^{\varepsilon}(P_1(u, \theta)) \right)  \in UV_2^{\varepsilon}\subset  UV_d^{\varepsilon}.$$
Also, $[\emptyset] =\Phi_0^{\varepsilon}(1) =  U\Phi_1^{\varepsilon}(P_1(u, \theta)) \in UV_1^{\varepsilon}$ when $\Gamma=0$. In sum, we have 
$$
B_{m-1}^{\varepsilon}
\subset UV_m^{\varepsilon}
\subset UV_d^{\varepsilon}
$$
for every $ 1\le m\le d$.  
 Therefore,
$$
V_{d-1}^{\varepsilon}
=
\bigoplus_{m=1}^{d}B_{m-1}^{\varepsilon}
\subset UV_d^{\varepsilon}.
$$

\paragraph{Case 2:  $\Gamma=0$ and $e=1-2g$.} Now we assume $d \ge 2$. 
Note that $e=1-2g$ implicitly implies that $g \ge 1$ because $e<0$. Thus,  $d|e| \ge 2g$.   Then (\ref{eq27}) still holds when $d \ge 3$ or  $ d =2 $ and $z \ne \theta^g \in \mathcal{E}$.  The argument in Lemma \ref{lem18} still show that 
$$\Phi_{m-1}^{\varepsilon}(z)
=U\Phi_m^{\varepsilon}\bigl(P_m(u,\theta)z\bigr). 
$$
when $m \ge 3$,  or $m=2$ and $z \ne \theta^g$.  By the same argument in  Case 1, we have  the following statements:
\begin{enumerate}
 
\item
When $m \ge 3$,  
$$
\Phi_{m-1}^{\varepsilon}(u^kz)
=
U^{k+1}
\Phi_m^{\varepsilon}\bigl(P_m(u,\theta)z\bigr)
$$
 where $0\le k \le (m-1)|e| -g$.    
 Then 
$$B_{m-1}^{\varepsilon}
\subset UV_m^{\varepsilon}\subset UV_d^{\varepsilon}$$
when $m \ge 3$.

\item
When $m=2$ and $0\le k \le |e|-g$ and $z \ne \theta^g$, we have  
\begin{equation} \label{eq28}
{\Phi}_{1}^{\varepsilon}(u^kz) =  U^{k+1} \Phi_2^{\varepsilon}\bigl(P_2(u,\theta)z\bigr)- \Psi_0^{\varepsilon}(1 \otimes \tilde{\Phi}^{\varepsilon}
_1(u^kz))[\emptyset] .
\end{equation}
If $z$ contains $e_c^r \ne 1$,  then  $\Psi_0^{\varepsilon}(1 \otimes \tilde{\Phi}^{\varepsilon}
_1(u^kz)) =0$ by Lemma \ref{lem14}.  Thus,  ${\Phi}_{1}^{\varepsilon}(u^kz)  \in UV_2 \subset UV_d^{\varepsilon}$ in these cases.

\end{enumerate}

 By Lemma \ref{lem7}, Lemma \ref{lem14}, and (\ref{eq28}), we have 
   \begin{equation*}
 \begin{split}
&\Phi^{\varepsilon}_1(u^{|e|-g+1-k} \theta^k) =U^{|e|-g-k+2} \Phi^{\varepsilon}_2(P_2(u,\theta) \theta^k) -\frac{g!}{(g-k)!}2^{g-k} [\emptyset],  \ \  1\le k \le g-1 \\
&U\Phi^{\varepsilon}_2(u^{2|e|-g} \theta^g) =g! [\emptyset] + \Phi^{\varepsilon}_1(\theta^g)\\
&U\Phi^{\varepsilon}_1(u^{|e|-g}) =2^g [\emptyset] + \frac{(-1)^{g+1}}{g!}\Phi^{\varepsilon} _1(\theta^g) + \sum_{k=1}^{g-1} \frac{(-1)^{k+1}}{k!} \Phi^{\varepsilon} _1(u^{|e|-g-k+1}\theta^k)\\
 \end{split}
\end{equation*}
Combing the first and third equations above, we have 
   \begin{equation*}
 \begin{split}
U \Phi^{\varepsilon}_1(u^{|e|-g})  = (1+(-1)^{g+1}) [\emptyset] +   \frac{(-1)^{g+1}}{g!}\Phi^{\varepsilon} _1(\theta^g) + \sum_{k=1}^{g-1} \frac{(-1)^{k+1}}{k!} U^{|e| -g -k + 2} \Phi_2(P_2(u,\theta)\theta^k) 
 \end{split}
\end{equation*}
Then we obtain 
   \begin{equation*}
 \begin{split}
&[\emptyset]= U\Phi^{\varepsilon}_1(u^{|e|-g})  + \frac{(-1)^g}{g!}  U\Phi^{\varepsilon}_2(u^{2|e|-g} \theta^g) - P  \\
&\Phi^{\varepsilon} _1(\theta^g)  =  (1+(-1)^{g+1}) U\Phi^{\varepsilon}_2(u^{2|e|-g} \theta^g)  -g! U\Phi^{\varepsilon}_1(u^{|e|-g})  + g!P,
 \end{split}
\end{equation*}
where  $P:=  \sum_{k=1}^{g-1} \frac{(-1)^{k+1}}{k!} U^{|e| -g -k + 2} \Phi^{\varepsilon}_2(P_2(u,\theta)\theta^k) $.  Hence,  $[\emptyset] ,\Phi^{\varepsilon} _1(\theta^g)    \in UV_2^{\varepsilon} \subset UV_d^{\varepsilon} $. In particular, $B_0^{\varepsilon} =\F[\emptyset] \in UV_2^{\varepsilon}  \subset UV_d^{\varepsilon} .$

By (\ref{eq28}), we have $ \Phi^{\varepsilon} _1(u^kz) \in  UV_2^{\varepsilon}$ except $z =\theta^g$.  By Lemma \ref{lem14},  $ \Phi^{\varepsilon} _1(u^k\theta^g)= U^k \Phi^{\varepsilon} _1(\theta^g)  \in   UV_2^{\varepsilon} $ for $0 \le k\le |e|-g.$
Therefore, $B_{1}^{\varepsilon}
\subset UV_2^{\varepsilon} \subset UV_d^{\varepsilon}$.

We have prove that $B_{m-1}^{\varepsilon} \subset UV_d^{\varepsilon}$ when $m \ge 3$. Taking account of all the above discussion, we have 
$$
V_{d-1}^{\varepsilon}
=
\bigoplus_{m=1}^{d}B_{m-1}^{\varepsilon}
\subset UV_d^{\varepsilon}.
$$
when $d \ge 2$. 
\end{proof}

The following lemma is an analogy of Lemma \ref{lem11} for the $U$ map. 
\begin{lemma}
Assume $e<\chi(\Sigma)$.  Assume $\Gamma=0$ or $\Gamma>2g-2$.  Then 
$$ \coker U_d  \xrightarrow{\pi_d}    \coker \wedge u.
$$
are isomorphisms for $d \ge 2$. In particular, $\dim  \coker U_d = \dim \ker U_d=2^{2g}  $  when $d \ge 2$.
\end{lemma} 
\begin{proof}
Let $\sigma \in V^{\varepsilon}_d$.   By Lemma \ref{lem2} and Lemma \ref{lem3}, we have 
  \begin{equation} \label{eq23}
 \begin{split}
\sigma =  \sum_{m=0}^d\sum_{j_m}  \langle   \Phi_m^{\varepsilon}(x_{m, j_m}), \sigma \rangle_{ech} (\Phi_m^{\varepsilon}(x_{m, j_m}) )^*,\\
 \end{split}
\end{equation}
where $(\Phi_m^{\varepsilon}(x_{m, i_m}) )^*:=  \sum_{j} (q^{m|e|+\Gamma})^{-1}_{i_m,j_m} \Phi_m^{\varepsilon}(x_{m, j_m})$ is the dual basis with respect to $\langle \cdot, \cdot \rangle_{ech}$.  Note that $\{(\Phi_m^{\varepsilon}(x_{m, i_m}) )^*\}$ are  still linear independent.

 Since $\langle \cdot, \cdot \rangle_{ech}$ is nondegenerate, we identify (noncanonically)  $\coker U_d$ with 
$$(\operatorname{Im} U_d)^{\perp} :=\{ \sigma \vert   \langle U\sigma', \sigma \rangle_{ech} =0, \forall \sigma' \in V^{\varepsilon}_d \}. $$
Let $\sigma \in (\operatorname{Im} U_d)^{\perp} $. By Lemma \ref{lem16}, we have  $V^{\varepsilon}_{d-1} \subset UV^{\varepsilon}_d$.  Then 
$\langle   \sigma', \sigma \rangle_{ech} =0 $  for all  $\sigma' \in V^{\varepsilon}_{d-1} $.  By Lemma \ref{lem7}, for  $d\ge 2$,  we have $$U\Phi_d^{\varepsilon}(u^{k-1}z) =\Phi_d^{\varepsilon}(u^kz)$$ for $1\le k \le d|e|+\Gamma-g $ and $z \in \mathcal{E}$.  Hence, $\langle    \Phi_d^{\varepsilon}(u^kz), \sigma \rangle_{ech} =0 $.     Consequently, Equation (\ref{eq23}) becomes
  \begin{equation*}
 \begin{split}
\sigma  =  \sum_{x_{d, j_d} \in \mathcal{E}}    \langle   \Phi_d^{\varepsilon}(x_{d, j_d}), \sigma \rangle_{ech} (\Phi_d^{\varepsilon}(x_{d, j_d}) )^*.\\
 \end{split}
\end{equation*}
Thus,  $ \dim  \coker U_d =  \dim (\operatorname{Im} U_d)^{\perp}  \le \#\mathcal{E} =2^{2g}$.  Since  $\pi_d: \coker U_d \to \coker \wedge u$ is surjective by the exact sequence  (\ref{eq19}),   $\dim \coker U_d  \ge \coker \wedge u=2^{2g}$.  In sum, $$\dim  \coker U_d  =2^{2g}. $$ The dimension count and the exact sequence  (\ref{eq19}) imply that $\pi_d$ is an isomorphism. 
\end{proof}

The above lemma tells  us  that $\ECH^{L_{\varepsilon, d}}(Y, \lambda_{\varepsilon}, \Gamma)$  consists of  $2^{2g}$ finite $U$-towers. The final thing we need to do is to show that these finite towers converges to  $2^{2g}$  infinite towers   as  $\varepsilon \to 0 $. To this end, we need the following algebraic lemma. 

\begin{lemma} \label{lem15}
Let $\F$ be a field.  Let 
$$V_0 \xrightarrow{\iota_0} V_1    \xrightarrow{\iota_1} V_2  \xrightarrow{\iota_2} ....$$
 be a direct system of finite dimensional $\F[U]$-modules. Assume 
 \begin{enumerate}
 \item
 $\iota_n$ are injective  $U$-equivariant homomorphisms. 
 \item
$U$ is nilpotent on each $V_n$.
\item
There exists $N\ge 0$ such that  $\iota_n(V_n) \subset UV_{n+1}$ for all $n \ge N$. 
\item
 $\dim \coker U\vert_{V_n} =r$ for each $n \ge N$. 
 \end{enumerate}
 Then $V:=\varinjlim_{n\to\infty}V_n \cong \oplus_{j=1}^r \mathcal{T}^+$.
\end{lemma}
 \begin{proof}
Let $j_n: V_n \to V$ be the canonical maps  of the direct limit. Since all the structure maps $\iota_n$ are injective,  
$j_n$ are injective. We may therefore identify $V_n$ with its image in $V$ via $j_n$.
Under these identifications, we have 
$$ V_1\subseteq V_2\subseteq V_3\subseteq\cdots$$
and
$$ V=\bigcup_{n\geq 1}V_n.$$

We next show that
$$ U:V\longrightarrow V $$
is surjective. Let $x\in V$. By replacing the index by a larger one if
necessary, we may choose $n\geq N$ such that $x\in V_n$. By assumption,
$$ \iota_n(V_n)\subseteq U V_{n+1}. $$
Thus there exists $y\in V_{n+1}$ such that
$$ Uy=\iota_n(x). $$
The elements $x$ and $\iota_n(x)$ determine the same element of the
direct limit. Therefore,
$$ Uy=x $$
in $V$, and hence
$UV=V.$

For each $n$, let $U_n $ denote the restriction of $U$ on $V_n$.   By assumption, we have 
$$\ker U_n =\coker U_n =r$$
for every $n\geq N$.
Because $\iota_n$ are injection, the dimension calculation implies that the restriction of $\iota_n$ to $\ker U_n$ are isomorphisms for all $n \ge N$.

We claim that
$$ \ker\{ U:V\longrightarrow V\} = \bigcup_{n\geq N} \ker U_n. $$
The inclusion from right to left is immediate. Conversely, let $x\in V$
satisfy $Ux=0$. Choose $n\geq N$ such that $x\in V_n$. Since the
canonical map $V_n\to V$ is injective, the equality $Ux=0$ in $V$
implies that $Ux=0$ already in $V_n$. Hence $x\in \ker U_n$.

As a result, we have  
$$ \dim_{\mathbb{F}}\ker\{U:V\longrightarrow V\}=r. $$
Choose a basis
$$ e_{1,0},...,e_{r,0} $$
of $\ker U$.
Since $U:V\to V$ is surjective, for each $j$ we   recursively choose
elements
$$
e_{j,k}\in V,\qquad k\geq 1, $$
such that
$$ Ue_{j,k}=e_{j,k-1}.$$
Thus, for each $j$, we obtain an infinite $U$-tower
$$
\cdots
\xrightarrow{U} e_{j,2}
\xrightarrow{U} e_{j,1}
\xrightarrow{U} e_{j,0}
\xrightarrow{U} 0.
$$

Next, we prove that 
$\left\{ e_{j,k} \vert1\le j\le  r,  k \ge 0 \right\} $
is linearly independent. Suppose that
$$
\sum_{k=0}^{m}\sum_{j=1}^{r}a_{j,k}e_{j,k}=0.
$$
Applying $U^m$ to the above equation, and using
$
U^m e_{j,k}=0 $ for $k<m$,
 $U^m e_{j,m}=e_{j,0},
$
we obtain
$$
\sum_{j=1}^{r}a_{j,m}e_{j,0}=0.
$$
Since $e_{1,0},\ldots,e_{r,0}$ are linearly independent, it follows that
$ a_{j,m}=0$
for every $j$. Repeating this argument for $m-1,m-2,\ldots,0$, we obtain
$$a_{j,k}=0 $$
for every $j$ and $k$. Thus the collection is linearly independent.
  
We now prove that these elements span $V$. For any $x \in V = \cup_{n \ge 1} V_n$, then $x \in V_n $ for some $n \ge N$. By the assumption that $U_n$ is nilpotent, $U^{m+1} x=0$ for some $m \ge 0$.

We claim that if
$$ U^{m+1}x=0,$$
then
$$
x\in
\operatorname{Span}_{\mathbb{F}}
\left\{
e_{j,k}\ \middle|\
1\leq j\leq r,\ 0\leq k\leq m
\right\}. $$
 We argue by induction on $m$. For $m=0$, the condition $Ux=0$ implies that $x\in\ker U$, so
$
x\in
\operatorname{Span}_{\mathbb{F}}
\{e_{1,0},\ldots,e_{r,0}\}.
$

Suppose that the claim holds for $m-1$.  Let $x\in V$ satisfy $U^{m+1}x=0.$
Then
$$ U^m x\in\ker U.$$
Therefore there exist $a_1,\ldots,a_r\in\mathbb{F}$ such that
$$ U^m x=\sum_{j=1}^{r}a_j e_{j,0}. $$
Let 
$$ x' := x-\sum_{j=1}^{r}a_j e_{j,m}.$$
Since $U^m e_{j,m}=e_{j,0}$, we have
$$ U^m x' = U^m x-\sum_{j=1}^{r}a_jU^m e_{j,m}=0.$$
By the induction hypothesis, $x'$ lies in the span of the elements
$e_{j,k}$ with $0\leq k\leq m-1$. Hence $x$ lies in the span of the
elements $e_{j,k}$ with $0\leq k\leq m$.

 Hence,  the above discussion implies that 
$$
V=
\bigoplus_{j=1}^{r}
\operatorname{Span}_{\mathbb{F}}
\{e_{j,0},e_{j,1},e_{j,2},\ldots\} \cong \bigoplus_{j=1}^{r} \mathcal{T}^+.
$$

\end{proof}

\begin{proof}[Proof of Theorem \ref{thm1}]
Recall    $d_{\varepsilon}: = \min\{\lfloor \frac{L_{\varepsilon}  -3\Gamma }{3|e|}\rfloor,  \lfloor  \frac{1}{2}(1/\varepsilon -2\Gamma/|e|) -\frac{1}{2}\rfloor \}$.      Pick a sequence $\varepsilon_0> \varepsilon_1 >\varepsilon_2... \to 0$ such that  $L_{\varepsilon_n} \to \infty$ and 
$$d_{\varepsilon_1}<d_{\varepsilon_2}<d_{\varepsilon_3}<d_{\varepsilon_4}<.... $$ 
Recall $L_{\varepsilon}>0$  here is the irrational constant such that all the Reeb orbits  of $\lambda_{\varepsilon}$ with action strictly less  $L_{\varepsilon}$ only consists of the covers of the fibers over critical points of $H$. 

   Define $V_n : =V^{\varepsilon_n}_{d_{\varepsilon_n}}$ and  $L_n := L_{\varepsilon_n, d_{\varepsilon_n}}$. Then we have continuous maps  
$$\iota_n : V_n \to V_{n+1}$$
defined by $$\iota_n :=i^{L_n, L_{n+1}} \circ \ECH([\varepsilon_{n+1}, \varepsilon_n] \times Y, d\lambda_{\varepsilon_{n+1}, \varepsilon_n}).$$ Here $([\varepsilon_{n+1}, \varepsilon_n]  \times Y, d\lambda_{\varepsilon_{n+1}, \varepsilon_n})$ is an exact symplectic cobordism from $(Y, \lambda_{\varepsilon_n})$ to $(Y, \lambda_{\varepsilon_{n+1}})$.  Moreover, the continuous maps $\iota_n$ are   $U$-equivariant homomorphisms.
By the same  argument in Proposition 3.2 of \cite{NW},   $\{(V_n, \iota_n)\}$ forms a direct system. 
Copying Nelson-Weiler's Morse-Bott argument (see Section 7.1.4 of \cite{NW}),   we also have 
 $$\ECH(Y, \lambda, \Gamma) =\varinjlim_{n\to\infty}V_n$$
The argument in  Proposition 3.2 and Section 7.1.4 of \cite{NW} only relies on the formal  properties of the ECH cobordism maps. Thus, it is safe to copy their argument  possibly with  only notation changes.

Let $(X_+, \omega_{n+1}):=\left([\varepsilon_{n+1}, \varepsilon_n] \times Y,  d\lambda_{\varepsilon_{n+1}, \varepsilon_n} \right) \circ (X_+, \omega_+^{\varepsilon_n})$ denote the composition of  the symplectic cobordisms  $ \left([\varepsilon_{n+1}, \varepsilon_n] \times Y, d\lambda_{\varepsilon_{n+1}, \varepsilon_n}\right) $ and  $(X_+, \omega_+^{\varepsilon_n})$. It  is a strong symplectic cap of $ (Y,  \lambda_{\varepsilon_{n+1}}) $.    
  By Lemma \ref{lem2}, we have 
$$\iota_n(\Phi^{\varepsilon_n}_m(x_{m,i_m})) = \sum_{k=0}^{d_{\varepsilon_{n+1}}} \sum_{i_k}a_{k, i_k} \Phi^{\varepsilon_{n+1}}_k(x_{k,i_k}), $$
where  $a_{k, i_k} \in \Q$. 
Acting both side of the above equation by $\Psi^{\varepsilon_{n+1}}_{l}( x_{l,j_l} \otimes  \cdot )$, by  Lemma \ref{lem3},  we obtain
 \begin{equation}\label{eq32}
 \Psi^{\varepsilon_{n+1}}_{l}( x_{l,j_l} \otimes  \iota_n(\Phi^{\varepsilon_n}_m(x_{m,i_m})))= \sum_{i_l}a_{l, i_l}  q^{l|e| +\Gamma} (\Pi(x_{l,j_l}),  \Pi(x_{l,i_l})). 
 \end{equation}
By  the composition law of ECH cobordism maps (Theorem \ref{thm2}), we still have 
\begin{equation*}
\begin{split}
&\Psi^{\varepsilon_{n+1}}_{l}( x_{l,j_l} \otimes  \iota_n(\Phi^{\varepsilon_n}_m(x_{m,i_m})))\\
=& \ECH(X_-, \omega_-^{\varepsilon_{n+1}}, A_-(l)) (x_{l, j_l} \otimes   \left(\ECH(X_+, \omega_{n+1}, A_+(m))(x_{m, i_m})  - c(x_{m, i_m}) [\emptyset] \right))  \\
=& \Gr(X, A_{m, l}, x_{l, j_l}x_{m, i_m} )  - c(x_{m, i_m}) \ECH(X_-, \omega_-^{\varepsilon_{n+1}}, A_-(l)) (x_{l, j_l} \otimes    [\emptyset] ),
\end{split}
\end{equation*}
where  $c(x_{m, i_m}) $ is a constant which is zero unless $\Gamma=0$ and $m=1$ and it has been calculated by Lemma \ref{lem14}. Here the Gromov-Taubes invariants are defined by the symplectic form $\omega$ obtained  by gluing  $\omega_-^{\varepsilon_{n+1}}$ and  $\omega_{n+1}$.  From the construction, one can  verify that 
\begin{itemize}
\item
$[\omega] =[\omega^{\varepsilon_{n+1}}]$, where $\omega^{\varepsilon_{n+1}}$ 
is the symplectic form obtained by gluing $\omega_+^{\varepsilon_{n+1}}$ and $\omega_-^{\varepsilon_{n+1}}$.
\item 
$\omega$ is deformation equivalent to $\omega^{\varepsilon_{n+1}}$. 
\end{itemize}
 Then $\omega$ and  $\omega^{\varepsilon_{n+1}}$ determine the same homology orientation and ``Taubes chamber" (see Remark \ref{remark3}).  The Gromov-Taubes invariants $\Gr(X, A_{m, l}, x_{l, j_l}x_{m, i_m} ) $ are still computed by Theorem \ref{thm0}.

 Repeating  the argument in Lemma \ref{lem3}, we have 
\begin{equation}\label{eq35}
\Psi^{\varepsilon_{n+1}}_{l}( x_{l,j_l} \otimes  \iota_n(\Phi^{\varepsilon_n}_m(x_{m,i_m}))) =\delta_{lm}  q^{l|e| +\Gamma} (\Pi(x_{l,j_l}),  \Pi(x_{m,i_m})). 
\end{equation}
Combining the  (\ref{eq35}) with (\ref{eq32}), we have 
$$\iota_n(\Phi^{\varepsilon_n}_m(x_{m,i_m})) = \Phi^{\varepsilon_{n+1}}_m(x_{m,i_m}). $$
Thus, $\iota_n$ preserves the basis.  In particular,  $\iota_n$ are injective.

Without loss of generality, assume $d_{\varepsilon_n} \ge n$. By Lemma \ref{lem7} (also see (\ref{eq16})), we know that 
$$V_n \subset U   V_{n+1}$$
for all $n \ge 2$. 
Therefore, we have 
$$  \iota_n(V_n) = \iota_n(V^{\varepsilon_n}_{d_{\varepsilon_n}}) \subset    V^{\varepsilon_{n+1}}_{d_{\varepsilon_n}}   \subset U   V^{\varepsilon_{n+1}}_{d_{\varepsilon_{n+1}}} =U V_{n+1}. $$

By  Lemma \ref{lem11},  $\dim \coker U\vert_{V_n} =2^{2g}$ for  $n\ge 2$.   Then the direct system $\{(V_n, \iota_n)\}$ satisfies the four assumptions  in Lemma \ref{lem15}, we have 
$$\ECH(Y, \lambda, \Gamma) =\varinjlim_{n\to\infty}V_n \cong \oplus_{j=1}^{2^{2g}} \mathcal{T}^+.$$

\end{proof}

\section{Computing ECH spectrum}
We assume $\Gamma=0$ in this section throughout. Fix $k \in \Z_{\ge1}$. We take small $\varepsilon$ such that $d_{\varepsilon}  \gg  k$ through this section. Using  Lemma \ref{lem18}, we can construct a sequence $\sigma_d$ such that $U^d \sigma_d =[\emptyset]. $  The construction is as follows.

Define  polynomials
\begin{equation*}
\begin{split}
&S_d:=P_d(u,\theta)P_{d-1}(u,\theta)...P_1(u,\theta)\\
\mbox{ and }&S'_d:= \sum_{k=0}^g \frac{(-d)^k}{k!} \theta^ku^{g-k}.  
\end{split}
\end{equation*}
By a direct computation, we have 
\begin{equation*}
\begin{split}
S_d= \sum_{k=0}^g \frac{(-d)^k}{k!} \theta^ku^{|e|\frac{d(d+1)}{2}-gd-k} \mbox{ and } S_d= u^{ |e|\frac{d(d+1)}{2}-gd -g } S_d'. 
\end{split}
\end{equation*}
Note that  $P_1(u,\theta)$ is not well defined when $e=1-2g$. However, the above equations imply that $S_d$ is still well defined for $d \ge 2 $ in the case $e=1-2g$.

Assume $d \ge 3$. By  Lemma \ref{lem18}, $P_m(u,\theta) [\emptyset]=0$, and Theorem \ref{thm2},  we have 
\begin{equation} \label{eq30}
\begin{split}
&U \tilde{\Phi}_d^{\varepsilon}(S_d) =S_{d-1} U \tilde{\Phi}_d^{\varepsilon}(P_d(u, \theta)) =S_{d-1} U {\Phi}_d^{\varepsilon}(P_d(u, \theta))  \\
=& S_{d-1}{\Phi}_{d-1}^{\varepsilon}(1)   = \tilde{\Phi}_{d-1}^{\varepsilon}(S_{d-1}). 
\end{split}
\end{equation}

If we assume $e\le -2g$ additionally, then (\ref{eq30}) holds for $d \ge 1$. Therefore, we have 
\begin{equation*}
\begin{split}
U^d\tilde{\Phi}_d^{\varepsilon}(S_d)=U^{d-1}\tilde{\Phi}_{d-1}^{\varepsilon}(S_{d-1})...=U\tilde{\Phi}_1^{\varepsilon}(S_1)  = U{\Phi}_1^{\varepsilon}(P_1)=[\emptyset]. 
\end{split}
\end{equation*}
 
\begin{lemma} \label{lem19}
We have the following statements: 
\begin{enumerate}
\item
Assume $e\le -2g$. For $d\ge 1$, we have 
\begin{equation*}
\begin{split}
U^d\tilde{\Phi}_d^{\varepsilon}(S_d)=[\emptyset]. 
\end{split}
\end{equation*}

\item
Assume $e=1-2g$.  For $d \ge 2$, we have 
$$U^d \tilde{\Phi}_d^{\varepsilon}(S_d) =[\emptyset].$$
\end{enumerate}
 
\end{lemma} 
 \begin{proof}

 By (\ref{eq30}), it suffices to verify the case that  $e=1-2g$ and $d=2$. 
 
  Formally, write $P_1(u, \theta) =P^+_1(u, \theta) +\frac{(-1)^g}{g!} \theta^gu^{|e| -2g} $, where 
 $$P^+_1(u, \theta) =\sum_{k=0}^{g-1} \frac{(-1)^k}{k!}\theta^ku^{|e|-g-k}.$$
 Note that $P^+_1(u, \theta) $ is well defined in $\mathbb{A}(\Sigma) $.  Moreover, 
 $$u P^+_1(u, \theta) =\sum_{k=0}^{g-1} \frac{(-1)^k}{k!}\theta^ku^{|e|-g-k+1} =u^{|e|-g+1} -Q_1(u^{|e| -g} ) + \frac{(-1)^{g+1}}{g!} \theta^g.$$
 Then
 $$S_2= P_2(u, \theta)P_1(u, \theta) = P_2(u, \theta)P^+_1(u, \theta) +  \frac{(-1)^g}{g!} \theta^gu^{3|e| -3g}. $$
 
By Lemma \ref{lem18}, Lemma \ref{lem7}, Lemma \ref{lem14} and $U[\emptyset]=0$, we have  
 \begin{equation*}
\begin{split}
U^2 \tilde{\Phi}_2^{\varepsilon}(S_2) =&U^2 \tilde{\Phi}_2^{\varepsilon}(P_2(u, \theta)P^+_1(u, \theta)) +  \frac{(-1)^g}{g!} U^2 \tilde{\Phi}_2^{\varepsilon}(\theta^gu^{3|e| -3g})\\
=&U \Phi_1^{\varepsilon}(P^+_1(u, \theta)) + \frac{(-1)^g}{g!} U^2  \tilde{\Phi}_2^{\varepsilon}( u^{2|e| -g-1} \theta^g)\\
=& \tilde{\Phi}_1^{\varepsilon}(uP^+_1(u, \theta)) +  \frac{(-1)^g}{g!} U \Phi_2^{\varepsilon}( u^{2|e| -g} \theta^g)\\
=& U \tilde{\Phi}_1^{\varepsilon}(u^{|e|-g}) -  \tilde{\Phi}_1^{\varepsilon}(Q_1(u^{|e|-g}))  + \frac{(-1)^{g+1}}{g!} \tilde{\Phi}_1^{\varepsilon}(\theta^g) + \frac{(-1)^g}{g!}  \left( {\Phi}_1^{\varepsilon}(\theta^g) + g! [\emptyset]  \right)\\
=&U {\Phi}_1^{\varepsilon}(u^{|e|-g})  - \tilde{\Phi}_1^{\varepsilon}(Q_1(u^{|e|-g}))  \\
=& 2^{g} [\emptyset] +   {\Phi}_1^{\varepsilon}(Q_1(u^{|e|-g}) -  \tilde{\Phi}_1^{\varepsilon}(Q_1(u^{|e|-g})) \\
=& \left(2^{g}  + \sum_{k=1}^{g}
\frac{(-1)^{k}}{k!}
 \frac{g!}{(g-k)!}2^{g-k} \right)[\emptyset] =[\emptyset] .
\end{split} 
\end{equation*}

 \end{proof}
 
For $d \ge 0$, define 
\begin{equation}
I_d:=  |e|\frac{d(d+1)}{2}-gd -g + d .
\end{equation}
  By Lemma \ref{lem19},   we have 
\begin{equation*}
\begin{split}
U^{I_d}\tilde{\Phi}_d^{\varepsilon}(S'_d)= [\emptyset] \mbox{ and } \tilde{\Phi}_d^{\varepsilon}(S'_d)  \in V_d^{\varepsilon}=\ECH^{L_{\varepsilon, d}}(Y, \lambda_{\varepsilon}, 0). 
\end{split}
\end{equation*}
The above identity holds for $d \ge 1$ if $e\le -2g$; and holds for $d\ge 2$ provided that $e=1-2g$.
An immediate consequence of the above equations  is that  
$$c_{I_d}(Y, \lambda_{\varepsilon}) \le L_{\varepsilon, d} =2d|e|(1 + \varepsilon). $$

 \begin{lemma} \label{lem22}
Let  $d \ge 1$.  For  $I_{d-1} + g <k\le I_d$, we have 
 $$c_k(Y, \lambda_{\varepsilon}) =2d|e|(1+O(\varepsilon)).$$
 \end{lemma}
 \begin{proof}
  By the monotone property of ECH spectrum,  we have 
$$c_{I_{d-1}}(Y, \lambda_{\varepsilon})   \le c_k(Y, \lambda_{\varepsilon})\le c_{I_d}(Y, \lambda_{\varepsilon}) \le L_{\varepsilon, d} = 2d|e|(1+\varepsilon). $$
 If  $c_k(Y, \lambda_{\varepsilon})\le L_{\varepsilon, d-1}$, then there exists $\sigma \in V^{\varepsilon}_{d-1}$  such that  $U^k \sigma =[\emptyset]$. Without loss of generality, we assume $\sigma$ is homogenous.  Since $U$ has degree $-2$, we have $\operatorname{gr}(\sigma)=2k$.  By Lemma \ref{lem2},  $\sigma =\sum_{m=0}^{d-1} a_{m, i_m}\Phi_{m}^{\varepsilon}(x_{m, i_m})$ for some $a_{m, i_m} \in \Q$.  By Remark \ref{remark1}, we have 
\begin{equation} \label{eq29}
\begin{split}
 k=&\frac{1}{2} \operatorname{gr}(\sigma) = \frac{m(m+1)}{2}|e| + m(1-g) -\frac{1}{2}\degop x_{m, i_m}\\
 \le&\frac{d(d-1)}{2}|e| + (d-1)(1-g)=I_{d-1} + g<k.
\end{split}
\end{equation} 
 We obtain a contradiction. Hence,  $c_k(Y, \lambda_{\varepsilon})> L_{\varepsilon, d-1}$. By spectrality of ECH spectrum (see Lemma 4.1 of \cite{GHC} for example),  we have an orbit set $\alpha$ such that  $[\alpha] =0$ and 
 $$\max\{2(d-1)|e|(1+\varepsilon), 1\}=L_{\varepsilon, d-1}<c_k(Y, \lambda_{\varepsilon})  =2d(\alpha)|e|(1+O(\varepsilon))  \le L_{\varepsilon, d}=2d|e|(1+\varepsilon). $$
Here $|O(\varepsilon) | \le \frac{\varepsilon}{10}$ because $|H|_{C^0} \le \frac{1}{10}$. Thus, we have $d(\alpha) =d$.

 In sum, we have 
 $$c_k(Y, \lambda_{\varepsilon}) =2d|e|(1+O(\varepsilon)). $$

\end{proof}

To deal with the case that  $I_{d-1}<k \le I_{d-1} +g$, we need to rule out the probability that there exists $\sigma \in   V^{\varepsilon}_{d-1}$  such that $U^k \sigma =[\emptyset]$. 
\begin{lemma} \label{lem20}
Let $d\ge 1$. Assume that $I_d<k \le I_d +g$.  Then 
$$[\emptyset] \notin U^kV_d^{\varepsilon}.$$
\end{lemma}
\begin{proof}
Let $\sigma \in  V_d^{\varepsilon}$ such that  $U^k\sigma =[\emptyset]$ and $\operatorname{gr}(\sigma) =2k$.  
By Lemma \ref{lem2},  we write $\sigma =\sum_{m=0}^{d} a_{m, i_m}\Phi_{m}^{\varepsilon}(x_{m, i_m})$.  Note that $I_d -I_{d-1}=d|e|-g+1$.  If $a_{m, i_m} \ne 0 $ for some $m \le d-1$, then  by (\ref{eq29}), we have 
$$k\le I_{d-1} +g = I_d -d|e|+2g-1\le  I_d. $$
This contradicts the assumption that  $k>I_d$.
Consequently,  we have  
$$\sigma \in  V_d^{\varepsilon} \setminus V_{d-1}^{\varepsilon}  \mbox{ and } \sigma =\Phi_d^{\varepsilon}(x)$$ for some $x \in \operatorname{Span}_{\Q} \mathcal{B}_{d|e|}$.

Write $k=I_d +j$ for some $1\le j\le g$. The index computation implies that 
\begin{equation*}
\begin{split}
 I_d+ j=k=&\frac{1}{2} \operatorname{gr}(\sigma) = \frac{d(d+1)}{2}|e| + d(1-g) -\frac{1}{2}\degop x \\
=&I_{d} +g  -\frac{1}{2}\degop x
\end{split}
\end{equation*} 
Hence, $\degop x =2g-2j$. 
Write  
$$x = \sum_{a=0}^{g-j} c_a u^{g-j-a} \theta^{a} + x_{prim},$$ 
where $x_{prim}$ contains  factors $e_{r}^c\ne 1$. 
 
 For any $y \in \mathbb{A}(\Sigma)$,  by Theorem \ref{thm0} and composition law of ECH cobordism maps, we have 
\begin{equation*}
\begin{split}
&0= \Psi_1^{\varepsilon}(y\otimes \Phi_0^{\varepsilon}(1))  =\Psi_1^{\varepsilon}(y\otimes[\emptyset] ) =  \Psi_1^{\varepsilon}(y\otimes U^k \Phi_d^{\varepsilon}(x))=\Psi_1^{\varepsilon}(y\otimes U^k \tilde{\Phi}_d^{\varepsilon}(x)) =\Gr(X, A_{d, 1}, yu^k x).  
\end{split}
\end{equation*}  
We want to find $y$ such that  the right hand side  of the above equation to be nonzero and obtain a contradiction. The Gromov-Taubes invariants $\Gr(X, A_{d, 1}, yu^k x)$ are nonzero  only if $$I(A_{d, 1}) = 2k + \degop y + \degop x. $$
This implies that  $ \degop y  =2g-2$. We take $y =u^{g-1-b} \theta^{b} $, where $0\le b \le g-1.$

Then by Theorem \ref{thm0},  we have 
\begin{equation*} 
\begin{split}
 \Gr(X, A_{d, 1}, yu^k x)=& \sum_{a=0}^{g-j} c_a  \Gr(X, A_{d, 1},   u^{g-1-b}  \theta^b u^k  u^{g-j-a}\theta^a ) \\
 =& \sum_{a=0}^{g-j} c_a  \Gr(X, A_{d, 1},   u^{g-1-b + k + g-j-a} \theta^{a+b}) I(a+b \le g)\\
 =&\sum_{a=0}^{g-j} c_a \frac{g!}{(g-a-b)!}d^{ g- a-b} I(a+b \le g).  
\end{split}
\end{equation*}  
Here $ I(a+b \le g) =1$ provided that  $a+b \le g$,  and equals to  zero otherwise. 
Then
$$\sum_{a=0}^{g-j} c_a \frac{g!d^{g-a-b}}{(g-b-a)!} I(a+b \le g) =0. $$
Take $j-1\le b \le g-1$. Equivalently, $1\le g-b\le  g-j+1$.   Let $l= g-b-1$. Then the above equations is equivalent to 
$$\sum_{a=0}^{g-j} M_{l, a}c_a  =0, $$
where $M_{l, a} =  \frac{g!d^{l+1-a}}{(l+1-a)!} $   if $0\le a \le l+1$, else, $M_{l, a} =0$. One can show that the determinant of the matrix is  
$$\det (M_{l,a}) = \frac{(g!)^{g-j+1} d^{g-j+1}}{(g-j+1)!}.$$  Thus, it is an invertible,  and  $c_a=0$ for $0 \le a\le g-j$. 

Hence, $x =x_{prim}$.  We have another equation
\begin{equation*}
\begin{split}
&1=\Psi_0^{\varepsilon}(1\otimes \Phi_0^{\varepsilon}(1)) =\Psi_0^{\varepsilon}(1\otimes[\emptyset] )=\Psi_0^{\varepsilon}(1\otimes U^k \Phi_d^{\varepsilon}(x_{prim})) = \Gr(X, A_{d, 0}, u^k x_{prim}). \\
\end{split}
\end{equation*}  
However, by Theorem \ref{thm0},   $\Gr(X, A_{d, 0}, u^k x_{prim}) =0$. We obtain a contradiction.  In sum, such a $\sigma$ cannot exist. 

\end{proof}

\begin{lemma} \label{lem21}
Let $d \ge 2$. Assume $I_{d-1} <k \le I_{d-1} +g$.  Then 
 $$c_k(Y, \lambda_{\varepsilon}) =2d|e|(1+O(\varepsilon)). $$
\end{lemma}
\begin{proof}  
 If  $c_k(Y, \lambda_{\varepsilon})\le L_{\varepsilon, d-1}=\max\{2(d-1)|e|(1+\varepsilon), 1\}$, then there exists $\sigma \in V^{\varepsilon}_{d-1}$  such that $U^k \sigma =[\emptyset]$.  However, this contradicts with Lemma \ref{lem20}. Hence,  $c_k(Y, \lambda_{\varepsilon})> L_{\varepsilon, d-1}$.  
 
 By spectrality of ECH spectrum,  we have an orbit set $\alpha$ such that  $[\alpha] =0$ and 
 $$\max\{2(d-1)|e|(1+\varepsilon), 1\}=L_{\varepsilon, d-1}<c_k(Y, \lambda_{\varepsilon})  =2d(\alpha)|e|(1+O(\varepsilon))  \le L_{\varepsilon, d}=2d|e|(1+\varepsilon). $$
Here $|O(\varepsilon) | \le \frac{\varepsilon}{10}$. Thus, we have $d(\alpha) =d$.

Then 
 $$c_k(Y, \lambda_{\varepsilon}) =2d|e|(1+O(\varepsilon)). $$
\end{proof}

\begin{proof}[Proof of Theorem \ref{thm3}]
Because $I_d$ is increasing with respect to $d$, for any $k \in \Z_{\ge 1}$,  we can find a unique $d$ such that $I_{d-1}<k \le I_d$. 
By Lemma \ref{lem21}  and Lemma \ref{lem22}, 
 $$c_k(Y, \lambda_{\varepsilon}) =2d|e|(1+O(\varepsilon)). $$
 Taking  $\varepsilon \to 0 $, we have $$c_k(Y, \lambda ) =2d|e|. $$
\end{proof}

 To prove Corollary \ref{corollary1}, it suffices to prove the following lemma. 
 \begin{lemma} \label{lem23}
 Let $(M, \Omega)$ be a symplectic 4-manifold.   Suppose we have a symplectic hyperplane section $\Sigma \subset M$  with degree $k_{\Sigma}>0$. Let $(Y, \lambda)$ be the prequantization bundle over $(\Sigma,  k_{\Sigma}\Omega \vert_{\Sigma})$ with Euler number $e =-[\Sigma]^2$. Then
 $$c_k(M \setminus \Sigma, \Omega) = \frac{1}{2k_{\Sigma}} c_k(Y, \lambda). $$
 \end{lemma}
\begin{proof}  
Let $e =-[\Sigma]^2.$   
Let $E^{\vee}$ denote the normal bundle of $\Sigma$.   By the assumption, we have     $c_1(E^{\vee}) =k_{\Sigma}[\omega_{\Sigma}]$, where $\omega_\Sigma =\Omega{\vert_{\Sigma} }.   $ The Chern number of $E^{\vee}$ is $[\Sigma]^2$.

Before  computing  the ECH capacities, we first  give the description  on the symplectic neighborhood of $\Sigma \subset M$. There is a standard  symplectic form $\Omega_{can}$ on  the normal bundle $E^{\vee}$ defined by
$$\Omega_{can} := - d(e^{-r^2} \alpha_{E^{\vee}})   .$$
where $r $ is the radius coordinate of $E^{\vee}$ and $\alpha_{E^{\vee}} $ is defined in (\ref{eq33}).  Note that $d\alpha_{E^{\vee}} = -k_{\Sigma} \omega_{\Sigma}$.  One can extend $\Omega_{can}$  smoothly across  the zero section so that it becomes a symplectic form  of $E^{\vee}$.   The Liouville vector field is  of $\Omega_{can} $ on $E^{\vee} \setminus \Sigma$ is 
$$Z  = -\frac{1}{2r}  \partial_r.  $$

By  Proposition 2.1 of \cite{BK},   the  open disk subbundle  $\left(\mathring{\mathbb{D}}_{\delta}E^{\vee}, \frac{1}{k_{\Sigma}} \Omega_{can}  \right) $ is symplectic to an open  neighourhood  $U_{\delta}  \subset M$  open  neighborhood of $\Sigma$, where $\mathbb{D}_{\delta} E^{\vee}:= \{ v \in E^{\vee}: |v| < \delta\}$.   
Let $Y=-\partial U_{\delta}$.   Then $Y$ is a circle bundle over $\Sigma$ with Euler number $e<0$. The induced contact form on   is  
  $$ \frac{1}{k_{\Sigma}} \Omega_{can}(Z,  \cdot) =-\frac{e^{-\delta^2}}{k_{\Sigma}}  \alpha_{E^{\vee}} =\frac{e^{-\delta^2}}{2k_{\Sigma}} \lambda . $$   
Then $(\overline{U}_{\delta}, \Omega ) $ is a strong symplectic cap of  the prequantization bundle $(Y, \frac{e^{-\delta^2}}{2k_{\Sigma}}  \lambda)$. Let $X_{\delta} :=M \setminus U_{\delta}$.     By  Proposition 2.1 of \cite{BK} (also see \cite{Gir}),    $(X_{\delta}, \Omega \vert_{X_{\delta}} =d\lambda_{X_{\delta}})$ is an exact symplectic  filling  (in fact, a Weinstein filling) of $(Y, \frac{e^{-\delta^2}}{2k_{\Sigma}}  \lambda)$.

For any Liouville domain $(X', d\lambda_{X'}) \hookrightarrow (M \setminus\Sigma,  \Omega)$,  because $X'$ is compact, we can find a sufficiently small $\delta>0$ such that $(X', d\lambda_{X'})   \hookrightarrow (X_{\delta}, d\lambda_{X_{\delta}})$.  By the monotonicity of ECH capacities,  we have 
$$c_k(X', d\lambda_{X'})  \le  c_k (X_{\delta}, d\lambda_{X_{\delta}}) = c_k(Y, \frac{e^{-\delta^2}}{2k_{\Sigma}}  \lambda)=\frac{e^{-\delta^2}}{2k_{\Sigma}} c_k(Y, \lambda)\le \frac{1}{2k_{\Sigma}} c_k(Y, \lambda).$$
Thus, $c_k(M \setminus \Sigma, \Omega ) = \sup \{c_k(X', d\lambda_{X'})  \vert (X', d\lambda_{X'})   \hookrightarrow  (M \setminus\Sigma,  \Omega)\} \le  \frac{1}{2k_{\Sigma}} c_k(Y, \lambda)$.  
On the other hand,  $ (X_{\delta}, d\lambda_{X_{\delta}})$ is a sequence of Liouville domains in $M \setminus \Sigma$. Hence,  
$$c_k(M \setminus \Sigma, \Omega ) \ge \limsup_{\delta \to 0} c_k (X_{\delta}, d\lambda_{X_{\delta}}) =  \limsup_{\delta \to 0}c_k(Y, \frac{e^{-\delta^2}}{2k_{\Sigma}}\lambda)=  \frac{1}{2k_{\Sigma}} c_k(Y, \lambda).$$
In sum,  $c_k(M \setminus \Sigma, \Omega ) =  \frac{1}{2k_{\Sigma}} c_k(Y, \lambda).$
\end{proof}

\begin{proof}[Proof of Corollary \ref{corollary1}]
 By assumption,  we have  
$$-e = [\Sigma]^2 = k_{\Sigma}^2 \int_M \Omega \wedge \Omega >2g-2. $$
Then the prequantization bundle $(Y, \lambda)$ in Lemma \ref{lem23} has Euler number $e<2-2g$. Thus, we can apply the results in Theorem \ref{thm3}. 
The conclusion of Corollary \ref{corollary1} follows directly from Lemma \ref{lem23}  and  Theorem \ref{thm3}. 
\end{proof}

Guanheng Chen  \\
Shenzhen University\\
Email: \texttt{ghchen@szu.edu.cn}
\end{document}